\documentclass[onefignum,onetabnum]{siamart190516}
\usepackage{hyperref}
\hypersetup{
  colorlinks=true,
  linkcolor=blue,
  citecolor=blue,
  urlcolor=blue
}
\usepackage{amsmath,amssymb}
\usepackage{mathrsfs}
\usepackage{graphicx}
\graphicspath{{./pic/}}
\usepackage{verbatim}
\usepackage{threeparttable,booktabs}
\usepackage{underscore}
\usepackage{diagbox}
\usepackage{multirow}
\usepackage{makecell}

\newtheorem{assumption}{Assumption}[section]

\newtheorem{example}{Example}[section]

\usepackage{dsfont}
\usepackage{epstopdf}
\usepackage{booktabs}
\usepackage{threeparttable}
\usepackage{verbatim,bm}
\usepackage{algorithm,algpseudocode}

\numberwithin{equation}{section}

\renewcommand{\bold}{\boldsymbol}
\allowdisplaybreaks

\makeatletter
\let\cref@override@label@type\@gobble   
\makeatother

\title{Dual Variational Neural Network for the $p$-Laplace Problem\thanks{The work of T. Hu is supported by National Natural Science Foundation of China (Project 125B2022). The work of G. Li is supported by Hong Kong Research Grants Council (Project 17317022). The work of Y. Xu is supported by National Natural Science Foundation of China (Projects 12250013, 12261160361 and 12271367). The work of Z. Z. is supported by by National Natural
Science Foundation of China (Project 12422117 and Project 12426312) and Hong Kong Research
Grants Council (15302323), and an internal grant of Hong Kong Polytechnic
University (Project ID: P0053938, Work Programme: 4-ZZVA)}}
\author{Tianhao Hu\thanks{Department of Mathematics, The Chinese University of Hong Kong, Shatin, N.T., Hong Kong, China. (email: \texttt{thhu@link.cuhk.edu.hk}, \texttt{fengruwang@cuhk.edu.hk})}\and Guanglian Li\thanks{Department of Mathematics, The University of Hong Kong, Pok Fu Lam Road, Hong Kong SAR, P.R. China. (\texttt{lotusli@maths.hku.hk}).} \and Fengru Wang\footnotemark[2] \and Yifeng Xu\thanks{Department of Mathematics and Scientific Computing Key Laboratory of Shanghai Universities, Shanghai Normal University, Shanghai 200234, China. (email: \texttt{yfxu@shnu.edu.cn})} \and Zhi Zhou\thanks{Department of Applied Mathematics, The Hong Kong Polytechnic University, Kowloon, Hong Kong, China. (\texttt{zhizhou@polyu.edu.hk})}}

\date{\today}

\begin{document}

\maketitle

\begin{abstract}
The reliable and accurate numerical approximation of the $p$-Laplacian is particularly challenging in the extreme regimes $p \to 1^{+}$ and $p \gg 1$, where the operator becomes either highly singular or strongly degenerate, often causing severe instability in standard numerical methods. To address these difficulties, we propose a novel deep learning based framework, termed the \emph{dual variational neural network}, for $p$-Laplace problems. The approach is based on a mixed formulation and an $\bold{L}^{q}$-based Helmholtz decomposition, which decouples the original problem into two convex subproblems: a linear Poisson problem for the irrotational component and an unconstrained minimization problem over divergence-free fields for the solenoidal component. Following the decomposition, we employ two neural networks using a gradient--curl representation to approximate the flux, and further establish an error analysis of the neural approximation. The analysis relies on fundamental vector inequalities together with tools from statistical learning theory. Numerical experiments demonstrate robust convergence of the proposed method in challenging settings, including the extreme cases $p \to 1^{+}$ and $p \gg 1$, as well as the $p(x)$-Laplace equation.
\end{abstract}

\section{Introduction}
Let $\Omega \subset \mathbb{R}^{3}$ be a simply-connected bounded domain with a $C^{1,1}$ boundary.
For any $p\in (1,\infty)$ (with the conjugate exponent $q = \frac{p}{p-1}$), a given source $f \in L^{q}(\Omega)$ and some $g \in W^{\frac1q,p}(\partial\Omega)$, consider the following boundary value problem for the $p$-Laplacian:
\begin{align}\label{problem}
\left\{
\begin{aligned}
-\nabla \cdot \bigl( |\nabla u|^{p-2} \nabla u \bigr) &= f && \text{in } \Omega, \\
u &= g && \text{on } \partial\Omega,
\end{aligned}
\right.
\end{align}
where $|\cdot|$ denotes the Euclidean norm of vectors.
The model \eqref{problem} is fundamental for describing physical phenomena with power-law nonlinearities.
It arises naturally in shear-dependent stress fields in non-Newtonian fluids~\cite{Astarita1974, 10.1093/qjmam/27.2.193}, nonlinear dielectric response of composite materials~\cite{Otani1984, LIPTON200348}, and diffusion processes governed by constitutive relations with power-law fluxes~\cite{doi:10.1137/16M1067792}. Thus, the development of robust and efficient numerical methods for \eqref{problem} is of great practical importance.

The nonlinearity associated with the $p$-Laplacian poses a formidable challenge in the construction of accurate numerical approximations.
This is related to the pathology of the term $|\nabla u|^{p-2}$: as $p \to 1^{+}$, the model \eqref{problem} is singular since $|\nabla u|^{p-2}$ blows up near critical points of $u$, leading to the formation of sharp kinks and facets in the solution $u$, whereas as $p \to \infty$, it becomes degenerate since $|\nabla u|^{p-2}$ vanishes in regions of small gradient, leading to flat zones and steep transition layers.
These undesirable properties induce various singularities of the weak solution that adversely affect the performance of standard numerical solvers. For example, standard conforming finite element methods (FEMs) for the $p$-Laplace problem \cite{M2AN_1975__9_2_41_0,Ciarlet:1978,BarrettLiu1993} often require adaptive mesh refinement in order to effectively resolve localized solution singularities, especically for extreme $p$  \cite{LiuYan:2001,CarstensenKlose:2003,BDK:2012,LiuChen:2020}. 

In the last few years, deep learning-based approaches, e.g., physics-informed neural networks (PINNs) \cite{RAISSI2019686} and deep Ritz method (DRM) \cite{yu2018deep}, have emerged as popular meshless neural PDE solvers. However, these methods often also struggle with the $p$-Laplacian. Indeed, the strong nonlinearity of the model \eqref{problem} leads to highly complex loss landscapes, and the loss functions in PINNs and DRM lack the structural mechanisms to enforce critical physical constraints on the flux variable (the second equation of \eqref{first-order_sys}). Thus, the training process suffers from serious stagnation or divergence, especially as $p$ approaches its extreme limits.

In many applications, the primary quantity of interest is not the potential $u$, but the flux $\boldsymbol{\sigma} = -|\nabla u|^{p-2}\nabla u$.
Using the flux $\bold\sigma$, problem~\eqref{problem} can be recast as a first-order system for $(u, \boldsymbol{\sigma})$:
\begin{align}\label{first-order_sys}
\left\{
\begin{aligned}
\boldsymbol{\sigma} + |\nabla u|^{p-2}\nabla u &= \mathbf{0} && \text{in } \Omega, \\
\nabla \cdot \boldsymbol{\sigma} &= f && \text{in } \Omega, \\
u &= g && \text{on } \partial\Omega.
\end{aligned}
\right.
\end{align}
The mixed formulation facilitates the approximation of $\boldsymbol{\sigma}$ and leads to locally conservative schemes, e.g., mixed FEMs for 
the $p$-Laplacian~\cite{Farhloul:1998, Farhloul2000}. However, it involves a nonlinear saddle-point discrete system which remains numerically challenging. We employ the Helmholtz decomposition of the $\bold{L}^{q}$-based divergence space \cite{FABES1998323}, and decouple the mixed problem \eqref{first-order_sys} into two more tractable subproblems: a linear Poisson equation for the irrotational component of the flux $\boldsymbol{\sigma}$ and an unconstrained minimization over divergence-free (solenoidal) fields. These two subproblems are well-posed and admit convex formulations.
The splitting isolates the nonlinearity and provides a powerful new avenue for constructing numerical approximations. Moreover, the formulation also extends to the $p(x)$-Laplace problem under the log-H\"older continuity~\cite{diening2011lebesgue, Aramaki2022, Aramaki2022Var}. 

Building on the analytic structure, we develop a novel deep learning solver, called the Dual Variational Neural Network (DVNN), for $p$-Laplace problems. The DVNN approximates the flux $\boldsymbol{\sigma}$ by explicitly parameterizing its Helmholtz decomposition using two neural networks: one representing the potential $\phi$ to capture the irrotational part, and the other representing a vector potential $\boldsymbol{\psi}$ to capture the solenoidal part via its curl. It is a neural analogue of the gradient-curl representation~\cite{Amrouche:2011}, and preserves the divergence-free constraint on the solenoidal component. The algorithm proceeds by solving the two subproblems sequentially.
Moreover, we perform an error analysis for the DVNN approximations and establish error bounds that are explicit in terms of the architecture and the number of sampling points. The numerical experiments clearly show the robust convergence of the DVNN for extreme values of $p$ (e.g., $p=1.1$ and $p=500$) and for $p(x)$-Laplace problems. These scenarios are particularly challenging for existing numerical methods, e.g., FEM, PINNs, and DRM, but the DVNN maintains high accuracy and robust convergence. To the best of our knowledge, the DVNN is the first theoretically grounded neural solver  for $p$-Laplace problems that converges robustly for a broad range of $p$ values.

The rest of the paper is organized as follows. In Section~\ref{sec:algorithm}, we discuss the mixed variational formulation and the DVNN algorithm, and present a rigorous convergence analysis in Section~\ref{sec:error}.
In Section~\ref{sec:experiment}, we provide numerical experiments to illustrate the robustness and accuracy of the method in several challenging scenarios, including the $p(x)$-Laplace problem.
In Section~\ref{sec:conclusion}, we give a brief conclusion. In the appendix, we collect several technical and lengthy proofs of the error estimates. Throughout, the notation $c$ denotes a generic positive constant whose value may change from one occurrence to another, and $(\cdot,\cdot)$ denotes the $L^2(\Omega)$ inner product 

\section{Numerical algorithm}\label{sec:algorithm}

In this section, we develop the dual variational neural network (DVNN) based on the mixed variational formulation of problem~\eqref{first-order_sys}~\cite{Farhloul2000} and $\bold{L}^q$-based Helmholtz decomposition \cite{FABES1998323}. The approach also extends to the $p(\cdot)$-Laplace problem. 

\subsection{Mixed formulation}

Following~\cite{Farhloul2000}, we employ the two Sobolev spaces  
\begin{equation*}
\bold{X}_p\equiv\bold{W}^p({\rm div}, \Omega)=\{\boldsymbol{v}\in \bold{L}^p(\Omega), \nabla\cdot \boldsymbol{v}\in L^p(\Omega)\}, \quad  \bold{V}_p\equiv \bold{W}^p({\rm div0}, \Omega)=\{\boldsymbol{v}\in \bold{L}^p(\Omega), \nabla\cdot\boldsymbol{v}=0\},
\end{equation*}
equipped with the graph norm, and the associated trace space $W^{-\frac{1}{q},q}(\partial \Omega)$ for the normal component of functions in $\bold{X}_q$ (see, e.g.,~\cite[Theorems A.2 and A.3]{LiWangXu:2025} and~\cite[Lemma 1.2.2]{sohr2012navier}), which is also the dual space of $W^{\frac{1}{q},p}(\partial\Omega)$.
The first equation in \eqref{first-order_sys} implies $|\boldsymbol{\sigma}| = |\nabla{u}|^{p-1}$ and $ |\bold{\sigma}|^{q} = |\nabla{u}|^{pq-q} = |\nabla{u}|^p$, i.e., the formula $\nabla u = -|\bold{\sigma}|^{q-2} \bold{\sigma}$ holds.
In addition, $\|\bold{\sigma}\|_{\bold{L}^q(\Omega)}^q = \|\nabla{u}\|_{\bold{L}^p(\Omega)}^p$, i.e., $\boldsymbol{\sigma}\in \bold{L}^q(\Omega)$.
By the second equation in \eqref{first-order_sys}, $\nabla \cdot \bold{\sigma} \in L^q(\Omega)$. Hence $\bold{\sigma} \in \bold{X}_q$.
By integration by parts and Green's formula (\cite[(A.3)]{LiWangXu:2025} and~\cite[Lemma 1.2.3]{sohr2012navier}), the weak formulation of problem~\eqref{first-order_sys} is to find $(\bold{\sigma},u)\in \bold{X}_{q} \times L^p(\Omega)$ such that  
\begin{equation}\label{eqn:weakform}
\left\{
\begin{aligned}
(|\boldsymbol{\sigma}|^{q-2}\boldsymbol{\sigma},\boldsymbol{\tau})-(u,\nabla\cdot \boldsymbol{\tau})&=-\langle \boldsymbol{\tau}\cdot \boldsymbol{n}, g \rangle, \quad  \forall \bold{\tau}\in \bold{X}_{q},\\
(\nabla\cdot\boldsymbol{\sigma}, v)&=(f,v),\quad \forall v\in L^p(\Omega),
\end{aligned}\right.
\end{equation}
where $\boldsymbol{n}$ is the unit outward normal to the boundary $\partial\Omega$ and $\langle\cdot,\cdot\rangle$ denotes the duality pairing between $W^{-\frac{1}{q},q}(\partial\Omega)$ and $W^{\frac{1}{q},p}(\partial\Omega)$. The unique solvability of problem~\eqref{eqn:weakform} with $g=0$ has been proved in~\cite{Farhloul:1998, Farhloul2000}. When $g\not\equiv0$, an equivalent optimization problem reads 
\begin{equation}\label{min:origin}
\begin{aligned}
\inf_{\boldsymbol{\sigma}\in \bold{W}^q({\rm div},\Omega)} & \tfrac{1}{q}\|\boldsymbol{\sigma}\|_{L^q(\Omega)}^q + \langle \boldsymbol{\sigma}\cdot \boldsymbol{n}, g \rangle,\\
\text{s.t.}\quad &\nabla\cdot \boldsymbol{\sigma}=f,\quad\mbox{a.e. in }\Omega.
\end{aligned}
\end{equation}

\begin{theorem}\label{thm:unique-sol_min}
Problem~\eqref{min:origin} has a unique minimizer $\bold{\sigma}^\ast \in \bold{X}_q$.
\end{theorem}
\begin{proof}
The proof is similar to~\cite[Theorem 2.1]{Farhloul:1998}. We provide a proof for completeness.
By the inf-sup condition~\cite[Proposition 2.1]{Farhloul2000}
\begin{equation*}
\inf_{v\in L^p(\Omega)}\sup_{\bold{\tau}\in \bold{X}_q}\dfrac{(\nabla \cdot \bold{\tau}, v )}{\|\bold{\tau}\|_{\bold{X}_q}\|v\|_{L^p(\Omega)}} \geq \beta>0,
\end{equation*} 
there exists a unique $\bold{\sigma}_f \in \bold{X}_q/\bold{V}_q$ satisfying $\bold{\nabla}\cdot\bold{\sigma}_f  = f$.
Thus, problem~\eqref{min:origin} can be recast into 
\begin{equation}\label{min:aux_kernel}
\inf_{\bold{\tau} \in \bold{V}_q} \left\{ \mathcal{J}(\bold{\tau} + \bold{\sigma}_{f}) := \tfrac{1}{q}\|\bold{\tau} + \boldsymbol{\sigma}_f\|_{L^q(\Omega)}^q + \langle (\bold{\tau} + \boldsymbol{\sigma}_f )\cdot \boldsymbol{n}, g \rangle\right\}.
\end{equation}
By the continuity of the normal-component trace operator $\gamma_n:\bold{X}_q\to W^{-\frac{1}{q},q}(\partial\Omega)$ (\cite[Theorem A.2]{LiWangXu:2025} and~\cite[Lemma 1.2.2]{sohr2012navier}), the triangle inequality and the fact $\nabla\cdot\bold{\tau}=0$, we have
\begin{equation*}
\mathcal{J}(\bold{\tau}+\bold{\sigma}_f) \geq \tfrac{1}{q} \left|\|\bold{\tau}\|_{\bold{L}^q(\Omega)} - \|\bold{\sigma}_f\|_{\bold{L}^q(\Omega)}\right|^q - c\|\bold{\tau}\|_{\bold{L}^q(\Omega)}\|g\|_{W^{\frac{1}{q},p}(\partial\Omega)} - c\|\bold{\sigma}_f\|_{\bold{X}_q}\|g\|_{W^{\frac{1}{q},p}(\partial\Omega)}.
\end{equation*}
Since $q>1$, $\mathcal{J}$ is coercive on $\bold{V}_q$.
Next, since $\gamma_n: \bold{X}_q\to W^{-\frac1q,q}(\partial\Omega)$ is bounded linear, it is also weakly continuous~\cite[Theorem 3.10]{Brezis:2011}.
This and the weak lower semi-continuity (w.l.s.c.) of $\|\cdot\|_{\bold{L}^q(\Omega)}$ yield the w.l.s.c. of $\mathcal{J}$.
By the direct method in calculus of variations and the strict convexity of $\mathcal{J}$, \eqref{min:aux_kernel} has a unique minimizer $\bold{\sigma}_0\in\bold{V}_q$, and it follows that $\bold{\sigma}^\ast:=\bold{\sigma}_0 + \bold{\sigma}_f$.  
\end{proof}

\subsection{Problem transformation}
In the proof of Theorem~\ref{thm:unique-sol_min}, the decomposition $\bold{X}_q=\bold{X}_q/\bold{V}_q \oplus \bold{V}_q$ plays a vital role: the unique solvability of the constraint $\nabla\cdot\bold{\sigma} + f = 0$ is posed in the quotient space $\bold{X}_q/\bold{V}_q$, but the minimizer is sought in $\bold{V}_q$. However, the constraint posed on a quotient space $ \boldsymbol{X}_q/\boldsymbol{V}_q$ is numerically impracticable, and there is no explicit formulation for functions in the space $\bold{V}_q$. Thus, we adopt an $\bold{L}^q$-based Helmholtz decomposition and split problem \eqref{min:origin} into two subproblems: a Poisson equation with $f$ being the source (see \eqref{eqn:Poisson}) and an unconstrained optimization problem posed on the kernel space of $\bold{W}^q(\mathrm{div},\Omega)$ (see \eqref{min:unconstrained}), which is further characterized by the curl of $\bold{W}^{1,q}(\Omega)$, cf. Lemma \ref{lem:represent} below. 

We use the following 
$\bold{L}^q$-based Helmholtz decomposition \cite[Theorem 11.2]{FABES1998323}.

\begin{lemma}\label{lemma:helmholtz}
$
\bold{L}^q(\Omega)=\nabla W_0^{1,q}(\Omega) \oplus \bold{W}^q({\rm div0}, \Omega).$
\end{lemma}

In view of Lemma~\ref{lemma:helmholtz}, each $\bold{\sigma} \in \bold{X}_{q}$ can be decomposed as $ \bold{\sigma} = \nabla \phi + \bold{\tau} $ with $\phi\in W_0^{1,q}(\Omega)$ and $\bold{\tau} \in \bold{V}_q$. Motivated by the constraint in \eqref{min:origin},  consider the boundary value problem  
\begin{align}\label{eqn:Poisson}
\left\{
\begin{aligned}
\Delta \phi&=f\quad \mbox{in }\Omega,\\
\phi&=0\quad \mbox{on }\partial\Omega.
\end{aligned}
\right.
\end{align}
Problem~\eqref{eqn:Poisson} has a unique solution $\phi^\ast \in W^{2,q}(\Omega)\cap W_0^{1,q}(\Omega)$ for $f\in L^q(\Omega)$~\cite[Theorem 9.15]{GT:2001}. This allows us to reformulate problem~\eqref{min:origin} as 
\begin{equation}\label{min:unconstrained}
\inf_{\bold{\tau}\in \bold{V}_q} \left\{\mathcal{J}(\bold{\tau}+\nabla \phi^\ast) = \tfrac{1}{q} \| \boldsymbol{\tau} + \nabla\phi^\ast\|_{\bold{L}^q(\Omega)}^q+\langle (\boldsymbol{\tau} + \nabla\phi^\ast) \cdot \boldsymbol{n} , g \rangle \right\}.
\end{equation}

\begin{theorem}\label{thm:equivalence}
Problem~\eqref{min:unconstrained} has a unique minimizer $\bold{\tau}^\ast \in \bold{V}_q$ with
\begin{equation}\label{est:stab}
    \|\bold{\tau}^\ast\|_{\bold{L}^q(\Omega)}\leq c(q,\Omega)\Big(\|f\|_{L^q(\Omega)}+\|g\|_{W^{\frac{1}{q},p}(\partial\Omega)}^{\frac pq}\Big).
\end{equation}
Moreover, problem~\eqref{min:unconstrained} is equivalent to problem~\eqref{min:origin} in the sense that $\bold{\sigma}^\ast = \bold{\tau}^\ast + \nabla \phi^\ast$.
\end{theorem}

\begin{proof}
The unique solvability of problem~\eqref{min:unconstrained} can be proved similarly to Theorem~\ref{thm:unique-sol_min}. Then we prove the stability estimae \eqref{est:stab}. Since $\mathbf{0} \in \boldsymbol{V}_q$, the minimizing property gives
\begin{equation*}
\tfrac1q\|\bold{\tau}^*+\nabla \phi^\ast\|_{\bold{L}^q(\Omega)}^q+\langle (\boldsymbol{\tau}^* + \nabla\phi^\ast) \cdot \boldsymbol{n} , g \rangle\le\tfrac1q\|\nabla \phi^\ast\|_{\bold{L}^q(\Omega)}^q+\langle\nabla\phi^\ast \cdot \boldsymbol{n} , g \rangle,
\end{equation*}
By the trace theorem for $\bold{\tau}^\ast \in \bold{X}_q$ (\cite[Theorem A.2]{LiWangXu:2025}, \cite[Lemma 1.2.2]{sohr2012navier}) and $\nabla\cdot\bold{\tau}^\ast =0$, we have
\begin{equation*}
\begin{aligned}
&\tfrac1q\|\bold{\tau}^*+\nabla \phi^\ast\|_{\bold{L}^q(\Omega)}^q-c\left(\|\bold{\tau}^*+\nabla \phi^\ast\|_{\bold{L}^q(\Omega)} + \|\nabla \phi^\ast\|_{\bold{L}^q(\Omega)} \right) \|g\|_{W^{\frac{1}{q},p}(\partial\Omega)} \\
\leq &  \tfrac1q\|\bold{\tau}^*+\nabla \phi^\ast\|_{\bold{L}^q(\Omega)}^q - c \|\bold{\tau}^*\|_{\bold{L}^q(\Omega)} \leq 
\tfrac1q\|\nabla \phi^\ast\|_{\bold{L}^q(\Omega)}^q.
\end{aligned}
\end{equation*}
By Young's inequality, 
$\tfrac{1}{2q}\|\bold{\tau}^*+\nabla \phi^\ast\|_{\bold{L}^q(\Omega)}^q\le\tfrac{3}{2q}\|\nabla \phi^\ast\|_{\bold{L}^q(\Omega)}^q+c\|g\|_{W^{\frac{1}{q},p}(\partial\Omega)}^p$.
Further,  the regularity theory \cite[Lemma 9.17]{GT:2001} gives $\|\nabla\phi^*\|_{\bold{L}^{q}(\Omega)} \leq c \|f\|_{\bold{L}^q(\Omega)}$. Thus, there exists $c=c(q,\Omega)$ such that
\begin{align*}
&\|\boldsymbol{\tau}^*\|_{\bold{L}^q(\Omega)}\leq\|\bold{\tau}^*+\nabla \phi^\ast\|_{\bold{L}^q(\Omega)}+\|\nabla \phi^\ast\|_{\bold{L}^q(\Omega)}\\
\leq &c\Big(\|\nabla\phi^*\|_{\bold{L}^{q}(\Omega)}+\|g\|_{W^{\frac{1}{q},p}(\partial\Omega)}^{\frac{p}{q}}\Big)\leq c\Big(\|f\|_{L^q(\Omega)}+\|g\|_{W^{\frac{1}{q},p}(\partial\Omega)}^{\frac{p}{q}}\Big).
\end{align*}
Finally, let $\bold{\sigma}^\ast \in \bold{X}_q$ be the unique minimizer of problem \eqref{min:origin} and $\bold{\tau}^\ast\in \bold{V}_q$ be the unique minimizer of~\eqref{min:unconstrained}, respectively.
By Lemma \ref{lemma:helmholtz} again, $\bold{\sigma}^\ast$ can be expressed as $\bold{\sigma}_0^\ast + \nabla \psi^\ast$, with $\bold{\sigma}_0^\ast \in \bold{V}_q$ and $\psi^\ast \in W^{2,q}(\Omega)\cap W_0^{1,q}(\Omega)$ solving~\eqref{eqn:Poisson}.
Due to the uniqueness, $\psi^\ast = \phi^\ast$. Then by the minimizing property of $\bold{\tau}^\ast$ and $\bold{\sigma}^\ast$, we have
\begin{equation*}
\mathcal{J}(\bold{\tau}^\ast + \nabla \phi^\ast) \leq \mathcal{J}(\bold{\sigma}_0^\ast + \nabla\phi^\ast) = \mathcal{J}(\bold{\sigma}^\ast) \leq \mathcal{J}(\bold{\tau} + \nabla \phi^\ast),\quad \forall \bold{\tau}\in \bold{V}_q.
\end{equation*}
This implies  $\bold{\tau}^\ast + \nabla \phi^\ast = \bold{\sigma}^\ast$, since problems~\eqref{min:origin} and~\eqref{min:unconstrained} both have a unique minimizer.
\end{proof}

\subsection{Dual variational neural network}

By the Helmholtz decomposition in Lemma \ref{lemma:helmholtz}, problem~\eqref{min:origin} is split into a Poisson equation \eqref{eqn:Poisson} and an unconstrained minimization  problem~\eqref{min:unconstrained}, which can be solved sequentially using neural networks (NNs). Specifically, we approximate two unknowns $\phi$ and $\boldsymbol{\tau}$ by two NNs:
$\phi_{\theta}:\mathbb{R}^{3}\rightarrow\mathbb{R}$ with trainable parameters $\theta$ and $\boldsymbol{\psi}_{\eta}:\mathbb{R}^{3}\rightarrow\mathbb{R}^{3}$ with trainable parameters $\eta$, respectively. Then a neural  solution of problem \eqref{min:origin} is obtained via a gradient-curl representation:
\begin{equation}\label{eqn:flux_rep}
    \boldsymbol{\sigma}_{\theta,\eta} = \nabla\phi_{\theta} + \nabla \times \boldsymbol{\psi}_{\eta}.
\end{equation}
This construction is motivated by the Helmholtz decomposition in Lemma~\ref{lemma:helmholtz} and the following characterization of divergence-free vector fields~\cite[Lemma 4.1]{Amrouche:2011}. 
\begin{lemma}\label{lem:represent}
Let $\Omega \subset \mathbb{R}^3$ be a bounded, simply-connected domain with a $C^{1,1}$ boundary $\partial\Omega$.
A vector field $\boldsymbol{\tau} \in \bold{X}_q$ satisfies $\nabla \cdot \boldsymbol{\tau} = 0$ if and only if there exists $\boldsymbol{\psi} \in \bold{W}^{1,q}(\Omega)$ such that $\boldsymbol{\tau} = \nabla \times \boldsymbol{\psi}$.
\end{lemma}

We employ fully connected NNs for $\phi_{\theta}$ and $\boldsymbol{\psi}_\eta$, with the tanh activation  $\rho:\mathbb{R}\to \mathbb{R}$. The NN classes to approximate $\phi$ and $\boldsymbol{\psi}$ are denoted by $\mathcal{N}_i:=\mathcal{N}_i(L,W,B)$, with $i\in\{\phi,\boldsymbol{\psi}\}$, respectively, of depth at most $L$, width at most $W$, and all parameters bounded by $B$.
First, we approximate the solution $\phi^*$ of problem~\eqref{eqn:Poisson} by minimizing the PINN loss:
\begin{equation}\label{eqn:loss1}
L_{\rm p}(\phi)=\|\Delta \phi-f\|_{L^q(\Omega)}+\lambda\left\|\phi\right\|_{W^{1,q}(\partial\Omega)},
\end{equation}
where the penalty parameter $\lambda>0$ balances the interior PDE residual and the Dirichlet boundary condition. Note that in the loss $L_{\rm p}(\phi)$, we employ the $W^{1,q}(\partial\Omega)$ norm to enforce the Dirichlet boundary condition and the $L^q(\Omega)$ for the PDE residual, which is crucial to ensure the  convergence of the NN approximations in the $W^{1,q}(\Omega)$-norm.
In practice, the integrals are approximated using the Monte Carlo method.
Let $U(\Omega)$ and $U(\partial\Omega)$ denote the uniform distributions on $\Omega$ and $\partial\Omega$, respectively.
Given independent and identically distributed (i.i.d.)\ samples $\mathbb{X}=\{X_i\}_{i=1}^{N_d}\sim U(\Omega)$ and $\mathbb{Y}=\{Y_j\}_{j=1}^{N_b}\sim U(\partial\Omega)$, we define the empirical loss
\begin{equation}\label{eqn:emploss1}
\widehat{L}_{\rm p}(\phi_\theta)=\Big(\frac{|\Omega|}{N_d}
\sum_{i=1}^{N_d}|\Delta\phi_\theta(X_i)-f(X_i)|^q\Big)^{\!\frac1q}
+\lambda\Big(\frac{|\partial\Omega|}{N_b}
\sum_{j=1}^{N_b}|\phi_\theta(Y_j)|^q+|\nabla_t\phi_\theta(Y_j)|^q\Big)^{\!\frac1q},
\end{equation}
where $\nabla_t := \nabla - \boldsymbol{n}(\boldsymbol{n} \cdot \nabla)$ denotes the tangential gradient on the boundary $\partial\Omega$.
The minimizer $\widehat{\theta}$ of the empirical loss $\widehat{L}_{\rm p}$ over $\phi_\theta\in\mathcal{N}_\phi$ yields an approximation $\phi_{\widehat{\theta}}$:
\begin{equation}\label{optimization1}
\widehat{\theta} = \operatorname*{argmin}_{\theta}\widehat{L}_{\rm p}(\phi_\theta).
\end{equation}
By replacing the solution $\phi^\ast$ by the approximation $\phi_{\widehat{\theta}}$ and substituting the representation $\boldsymbol{\tau} = \nabla \times \boldsymbol{\psi}$ into the functional $\mathcal{J}$ in~\eqref{min:unconstrained}, we obtain the loss for the solenoidal component $\boldsymbol{\psi}$:
\begin{equation}\label{eqn:loss2prac}
L_{\rm s}(\boldsymbol{\psi})=\tfrac{1}{q}\|\nabla\phi_{\widehat{\theta}}+\nabla\times\boldsymbol{\psi}\|_{\bold{L}^q(\Omega)}^q+\langle (\nabla\phi_{\widehat{\theta}}+\nabla\times\boldsymbol{\psi})\cdot \bold{n} , g \rangle.
\end{equation}
and its empirical counterpart:
\begin{equation}\label{eqn:emploss2}
\widehat{L}_{\rm s}(\boldsymbol{\psi}_\eta)=\,
\frac{|\Omega|}{qN_d}\sum_{i=1}^{N_d}\bigl|\nabla\phi_{\widehat{\theta}}(X_i)+\nabla \times \boldsymbol{\psi}_\eta(X_i)\bigr|^q +\frac{|\partial\Omega|}{N_b}\sum_{j=1}^{N_b} g(Y_j)\bigl(\nabla\phi_{\widehat{\theta}}(Y_j)+\nabla \times \boldsymbol{\psi}_\eta(Y_j)\bigr)\cdot \boldsymbol{n}.
\end{equation}
Then by minimizing the loss $\widehat{L}_{\rm s}$ over $\boldsymbol{\psi}_\eta\in\mathcal{N}_{\boldsymbol{\psi}}$, we get a NN approximation $\boldsymbol{\psi}_{\widehat{\eta}}$:
\begin{equation}\label{optimization2}
\widehat{\eta} = \operatorname*{argmin}_{\eta}\widehat{L}_{\rm s}(\boldsymbol{\psi}_\eta).
\end{equation}
The flux approximation $\boldsymbol{\sigma}_{\widehat{\theta},\widehat{\eta}}$ is given by $\boldsymbol{\sigma}_{\widehat{\theta},\widehat{\eta}} = \nabla\phi_{\widehat{\theta}} + \nabla \times \boldsymbol{\psi}_{\widehat{\eta}}.$

The overall two‑step procedure is summarized in Algorithm~\ref{alg:dualform}. By the construction, the DVNN has two advantages: (i) the Poisson solver for $\phi$ depends only on $q$; (ii) the functional $\mathcal{J}(\bold{\tau}+\nabla\phi^*)$ in the second stage remains uniformly convex and coercive on $\boldsymbol{V}_q$ for all $p\in(1,\infty)$. Moreover, the two losses in \eqref{eqn:loss1} and \eqref{eqn:loss2prac} are defined in standard Sobolev spaces and free from singular and degenerate behavior.

\begin{algorithm}[H]
\caption{The DVNN for the $p$-Laplace problem}\label{alg:dualform}
\begin{algorithmic}[1] 
\State Solve problem~\eqref{optimization1} using the empirical loss $\widehat{L}_{\rm p}$ in \eqref{eqn:emploss1}.
\State Solve problem~\eqref{optimization2} using the empirical loss $\widehat{L}_{\rm s}$ in \eqref{eqn:emploss2}.
\State Output the numerical solution $\boldsymbol{\sigma}_{\widehat{\theta},\widehat{\eta}}=\nabla\phi_{\widehat{\theta}}+\nabla\times\boldsymbol{\psi}_{\widehat{\eta}}$.
\end{algorithmic}
\end{algorithm}

\subsection{$p(\cdot)$-Laplace problem}\label{subsec:px_extension}

Heterogeneous physical systems with spatially varying nonlinear responses are naturally described by variable-exponent $p(x)$-Laplacian models, e.g., electrorheological fluids in smart materials~\cite{Rajagopal2001Mathematical} and composite materials with spatially varying microstructure~\cite{Acerbi2001Regularity}. The variable-exponent $p(x)$-Laplace problem reads
\begin{equation}\label{eq:px_problem}
\left\{\begin{aligned}
-\nabla\!\cdot\!\bigl(|\nabla u|^{p(x)-2}\nabla u\bigr)&=f\quad\text{in }\Omega,\\
u&=g\quad\text{on }\partial\Omega,
\end{aligned}\right.
\end{equation}
where the variable exponent $p:\Omega\to(1,\infty)$. To ensure the well-posedness of problem \eqref{eq:px_problem}, we assume that $p\in\mathcal{P}^{\log}_+(\Omega)$, i.e., $p$ satisfies the following log-H\"older continuity condition
\begin{equation}\label{eq:log_holder}
|p(x)-p(y)|\leq C_{\log}[{\log(e+|x-y|^{-1})}]^{-1},\quad \forall\,x,y\in\Omega
\end{equation}
with essential bounds $1<p^-\leq p(x)\leq p^+<\infty$.
Then the variable-exponent Lebesgue space $L^{p(\cdot)}(\Omega)$ equipped with the Luxemburg norm
\begin{equation*}
\|u\|_{L^{p(\cdot)}(\Omega)}:=\inf\Bigl\{\lambda>0:\int_\Omega\Bigl|\frac{u(x)}{\lambda}\Bigr|^{p(x)}\mathrm{d}x\leq 1\Bigr\}
\end{equation*}
is a separable and reflexive Banach space~\cite{diening2011lebesgue}.
The Sobolev space $W^{1,p(\cdot)}(\Omega)$, endowed with the graph norm, consists of all functions $u\in L^{p(\cdot)}(\Omega)$ with the distributional gradient $\nabla u\in \bold{L}^{p(\cdot)}(\Omega)$ and $\bold{W}^{p(\cdot)}(\mathrm{div},\Omega)$ is defined similarly.

Let $\boldsymbol{\sigma}:=-|\nabla u|^{p(x)-2}\nabla u$ be the flux and $q(x)=\frac{p(x)}{p(x)-1}$ the pointwise conjugate exponent. The mixed formulation of problem~\eqref{eq:px_problem} reads: find $(u, \boldsymbol{\sigma}) \in L^{p(\cdot)}(\Omega) \times \bold{W}^{q(\cdot)}(\text{div}, \Omega)$ such that
\begin{equation} \label{eqn:px-weakform}
\left\{
\begin{aligned}
(|\boldsymbol{\sigma}|^{q(x)-2}\boldsymbol{\sigma}, \boldsymbol{\tau} ) - (u, \nabla\cdot \boldsymbol{\tau} ) &= -\langle \boldsymbol{\tau}\cdot \boldsymbol{n}, g \rangle_{W^{-\frac{1}{q(\cdot)},q(\cdot)}(\partial\Omega), W^{\frac{1}{q(\cdot)},p(\cdot)}(\partial\Omega)}, \\
( \nabla\cdot\boldsymbol{\sigma},  v )&= (f,v)
\end{aligned}\right.
\end{equation}
for all $(v, \boldsymbol{\tau}) \in L^{p(\cdot)}(\Omega) \times \bold{W}^{q(\cdot)}(\text{div}, \Omega)$.
The constitutive relation implies $|\boldsymbol{\sigma}|^{q(x)-2}\boldsymbol{\sigma}=-\nabla u$ and $\boldsymbol{\sigma}\in\bold{L}^{q(\cdot)}(\Omega)$.
Thus the problem admits the following constrained minimization formulation
\begin{equation}\label{eq:px_constrained_min}
\inf_{\substack{\boldsymbol{\sigma}\in\bold{L}^{q(\cdot)}(\mathrm{div},\Omega)\\ \nabla\cdot\boldsymbol{\sigma}=f}}
\Bigl\{ \mathcal{J}_{p(\cdot)}(\boldsymbol{\sigma})
:=\int_\Omega \frac{|\boldsymbol{\sigma}(x)|^{q(x)}}{q(x)}\,\mathrm{d}x
+\langle \boldsymbol{\tau}\cdot \boldsymbol{n}, g \rangle_{W^{-\frac{1}{q(\cdot)},q(\cdot)}(\partial\Omega), W^{\frac{1}{q(\cdot)},p(\cdot)}(\partial\Omega)} \Bigr\}.
\end{equation}
Note that the $\bold{L}^{q(\cdot)}$-based Helmholtz decomposition remains valid under the log-H\"older continuity condition~\cite[Theorem~3.4]{Aramaki2022}: for a bounded $C^1$ domain $\Omega$, $\bold{L}^{q(\cdot)}(\Omega)=\nabla W^{1,q(\cdot)}(\Omega)\oplus \bold{W}^{q(\cdot)}_0(\mathrm{div}0,\Omega)$, where $\bold{W}^{q(\cdot)}_0(\mathrm{div}0,\Omega)$ is the $\|\cdot\|_{\bold{W}^{q(\cdot)}(\mathrm{div},\Omega)}$ closure of smooth solenoidal fields with compact support.
Thus, $\boldsymbol{\sigma}=\nabla\widetilde{\phi}+\widetilde{\boldsymbol{\tau}}$, with $\phi\in W^{1,q(\cdot)}(\Omega)$ and $\widetilde{\boldsymbol{\tau}}\in \bold{L}^{q(\cdot)}_0(\mathrm{div}0,\Omega)$.
Then $\nabla\cdot\boldsymbol{\sigma}=\Delta\widetilde{\phi}=f$. Consider the following problem
\begin{align}\label{eqn:Poissonv-variable}
\left\{\begin{aligned}
\Delta \phi&=f\quad \mbox{in }\Omega,\\
\phi&=0\quad \mbox{on }\partial\Omega.
\end{aligned}
\right.
\end{align}
Under condition \eqref{eq:log_holder}, by~\cite[Theorem 3.1]{Aramaki2022Var}, there exists a unique solution $\phi^*\in W^{2,q(\cdot)}(\Omega)\cap W_0^{1,q(\cdot)}(\Omega)$. Let $\bold{W}^{q(\cdot)}({\rm div}0,\Omega)$ be the closure of smooth solenoidal fields in $\bold{W}^{q(\cdot)}(\Omega)$. Since $\Delta(\widetilde{\phi}-\phi^*)=0$, there holds $\nabla(\widetilde{\phi}-\phi^*)\in \bold{W}^{q(\cdot)}({\rm div}0,\Omega)$. So $\boldsymbol{\sigma}=\nabla\phi^*+\nabla(\widetilde{\phi}-\phi^*)+\widetilde{\boldsymbol{\tau}}:=\nabla\phi^*+\boldsymbol{\tau}^*$ with $\boldsymbol{\tau}^*\in \bold{W}^{q(\cdot)}({\rm div}0,\Omega)$.
This decomposition enables the algorithmic decoupling: (i) solve problem \eqref{eqn:Poissonv-variable}, and (ii) minimize $\mathcal{J}_{p(\cdot)}(\nabla\phi+\boldsymbol{\tau})$ over solenoidal fields $\boldsymbol{\tau}\in\bold{W}^{q(\cdot)}(\mathrm{div}0,\Omega)$.
Moreover, any divergence‑free vector field in $\boldsymbol{W}^{q(\cdot)}(\mathrm{div},\Omega)$ can be represented as the curl of a vector potential via the Biot-Savart operator \cite{ENCISO201885}:
$\boldsymbol{\tau}(x)=\nabla\times\mathcal{BS}[\boldsymbol{\tau}](x) = \nabla\times\frac{1}{4\pi}\int_{\Omega} \frac{\boldsymbol{\tau}(y)\times (x-y)}{|x-y|^3}\,{\rm d}y$.
Since $\mathcal{BS}$ is a Calder\'{o}n–Zygmund singular integral operator, it is bounded from $\bold{W}^{q(\cdot)}(\mathrm{div},\Omega)$ to $\bold{W}^{1,q(\cdot)}(\Omega)$ \cite[Corollary 6.3.10]{diening2011lebesgue}, i.e., $\mathcal{BS}[\boldsymbol{\tau}]\in \bold{W}^{1,q(\cdot)}(\Omega)$.
This justifies the gradient‑curl representation used in the DVNN for the variable-exponent case.

The DVNN for problem \eqref{eq:px_problem} proceeds identically as Algorithm~\ref{alg:dualform}:
\begin{align*}
\widehat{L}_{\rm p}(\phi_\theta)&=\frac{|\Omega|}{N_d}\sum_{i=1}^{N_d}\!|\Delta\phi_\theta(X_i)\!-\!f(X_i)|^{q(X_i)}
+\lambda\frac{|\partial\Omega|}{N_b}\sum_{j=1}^{N_b}\left(|\phi_\theta(Y_j)|^{q(Y_j)}+|\nabla_t\phi_\theta(Y_j)|^{q(Y_j)}\right),\\
\widehat{L}_{\rm s}(\boldsymbol{\psi}_\eta)&=\frac{|\Omega|}{N_d}\sum_{i=1}^{N_d}\frac{\bigl|\nabla\phi_{\widehat{\theta}^*}(X_i)+\nabla\times\boldsymbol{\psi}_\eta(X_i)\bigr|^{q(X_i)}}{q(X_i)}
+\frac{|\partial\Omega|}{N_b}\sum_{j=1}^{N_b}g(Y_j)\bigl(\nabla\phi_{\widehat{\theta}^*}+\nabla\times\boldsymbol{\psi}_\eta\bigr)(Y_j)\!\cdot\!\boldsymbol{n}(Y_j).
\end{align*}
The structural features with the scheme for problem \eqref{min:unconstrained}, including the exact enforcement of the divergence constraint via the gradient-curl representation and robustness for extreme exponent values, carry directly over to the variable-exponent setting.
The numerical experiments in Section~\ref{sec:experiment} confirm that the method can handle sharp spatial transitions in $p(x)$.

\section{Convergence analysis}\label{sec:error}

In this section, we perform an error analysis of the DVNN for the $p$-Laplace problem following ~\cite{JiaoLai:2022cicp,doi:10.1137/23M1601195}.
The analysis deals with the two steps in Algorithm~\ref{alg:dualform} separatelyeq. Without loss of generality, we may assume $\Omega\subset(-1,1)^3$.
The first step involves the empirical PINN loss $\widehat{L}_{\rm p}$ in \eqref{eqn:emploss1}, for which we use the following error decomposition.
\begin{lemma}\label{lem:decomp}
Let $\phi^*$ be the solution of~\eqref{eqn:Poisson} and $\phi_{\widehat{\theta}}$ be the minimizer of the loss $\widehat{L}_{\rm p}$ in \eqref{eqn:emploss1} over the neural network class $\mathcal{N}_\phi$.
Then there exists $c=c(\Omega,q,\lambda)$ such that  
\begin{equation}\label{eq:error_decomp}
\|\phi^*-\phi_{\widehat{\theta}}\|_{W^{1,q}(\Omega)} \leq c\big(\inf_{\phi_{\theta} \in \mathcal{N}_\phi}\|\phi_{\theta}-\phi^*\|_{W^{2,q}(\Omega)}+\sup _{\phi_{\theta} \in \mathcal{N}_\phi}|L_{\rm p}(\phi_{\theta})-\widehat{L}_{\rm p}(\phi_{\theta})|\big).
\end{equation} 
\end{lemma}

\begin{proof}
For any $\phi_{\theta}\in \mathcal{N}_\phi$, by the minimizing property of $\phi_{\widehat{\theta}}$ to the loss $\widehat{L}_{\rm p}$ over $\mathcal{N}_\phi$, we have
\begin{align*}
L_{\rm p}(\phi_{\widehat{\theta}}) &= [L_{\rm p}(\phi_{\widehat{\theta}})-\widehat{L}_{\rm p}(\phi_{\widehat{\theta}})]+[\widehat{L}_{\rm p}(\phi_{\widehat{\theta}})-\widehat{L}_{\rm p}(\phi_{\theta})]+[\widehat{L}_{\rm p}(\phi_{\theta})-L_{\rm p}(\phi_{\theta})]+L_{\rm p}(\phi_{\theta}) \\
&\leq 2\sup_{\phi_{\theta} \in \mathcal{N}_\phi}\big|L_{\rm p}(\phi_{\theta})-\widehat{L}_{\rm p}(\phi_{\theta})\big|+L_{\rm p}(\phi_{\theta}).
\end{align*}
Since $\Delta \phi^* = f$ in $\Omega$ and $\phi^* = 0$ on $\partial\Omega$, we can rewrite the loss $L_{\rm p}$ for any $\phi_\theta\in \mathcal{N}$ as
\begin{align*}
L_{\rm p}(\phi_{\theta}) &= \|\Delta \phi_\theta-f\|_{L^q(\Omega)}+\lambda\left\|\phi_\theta\right\|_{W^{1,q}(\partial\Omega)} \\
&= \|\Delta (\phi_\theta-\phi^*)\|_{L^q(\Omega)}+\lambda\left\|\phi_\theta-\phi^*\right\|_{W^{1,q}(\partial\Omega)}.
\end{align*}
By the trace theorem, there exists $c= c(\Omega,q,\lambda) > 0$ such that
$
L_{\rm p}(\phi_{\theta}) \leq c\|\phi_\theta-\phi^*\|_{W^{2,q}(\Omega)}$.
By taking the infimum over $\phi_\theta\in\mathcal{N}_\phi $, we obtain
\begin{equation}\label{eq:loss_bound}
L_{\rm p}(\phi_{\widehat{\theta}})\leq 2\sup_{\phi_{\theta} \in \mathcal{N}_\phi}\big|L_{\rm p}(\phi_{\theta})-\widehat{L}_{\rm p}(\phi_{\theta})\big|+c\inf_{\phi_\theta\in\mathcal{N}_\phi}\|\phi_\theta-\phi^*\|_{W^{2,q}(\Omega)}.
\end{equation}
Now  consider the following two boundary value problems
\begin{align}
\left\{
\begin{aligned}
\Delta v_1&=\Delta(\phi_{\widehat{\theta}}-\phi^*),\quad \mbox{in }\Omega,\\
v_1&=0,\quad \mbox{on }\partial\Omega.
\end{aligned}
\right.\quad
\mbox{and}\quad
\left\{
\begin{aligned}
\Delta v_2&=0,\quad \mbox{in }\Omega,\\
v_2&=\phi_{\widehat{\theta}}-\phi^*,\quad \mbox{on }\partial\Omega.
\end{aligned}
\right.
\end{align}
The elliptic regularity theory~\cite[Lemma 9.17]{GT:2001} implies that there exists $c=c(\Omega)$ such that
\begin{equation*}
    \|v_1\|_{W^{2,q}(\Omega)}\leq c\|\Delta(\phi_{\widehat{\theta}}-\phi^*)\|_{L^{q}(\Omega)}.
\end{equation*}
Since $v_2$ is the harmonic extension of $\phi_{\widehat{\theta}}-\phi^*$ in $\Omega$, there exists $c=c(\Omega)$ such that \cite[Theorem 5.1]{Jerison1995Inhomogeneous}
$\|v_2\|_{W^{1,q}(\Omega)}\leq c\|\phi_{\widehat{\theta}}-\phi^*\|_{W^{1-\frac{1}{q},q}(\partial\Omega)}\leq c\|\phi_{\widehat{\theta}}-\phi^*\|_{W^{1,q}(\partial\Omega)}.$
By the triangle inequality, we deduce
\begin{equation*}
\begin{aligned}
    &\|\phi_{\widehat{\theta}}-\phi^*\|_{W^{1,q}(\Omega)}=\|v_1+v_2\|_{W^{1,q}(\Omega)}\leq\|v_1\|_{W^{2,q}(\Omega)}+\|v_2\|_{W^{1,q}(\Omega)}\\
    \leq &c\left(\|\Delta(\phi_{\widehat{\theta}}-\phi^*)\|_{L^{q}(\Omega)}+\|\phi_{\widehat{\theta}}-\phi^*\|_{W^{1,q}(\partial\Omega)}\right) = cL_{\rm p}(\phi_{\widehat{\theta}}).
\end{aligned}
\end{equation*}
Combining this with \eqref{eq:loss_bound} yields the desired assertion~\eqref{eq:error_decomp}.
\end{proof}

Next we estimate the error  $\boldsymbol{\sigma}^\ast-\bold{\sigma}_{\widehat\theta,\widehat\eta}$. In view of Theorem \ref{thm:unique-sol_min} and Lemma \ref{lem:represent}, the solutions of problems \eqref{min:origin}, \eqref{eqn:Poisson} and \eqref{min:unconstrained} satisfy $\boldsymbol{\sigma}^\ast = \nabla\phi^\ast + \boldsymbol{\tau}^\ast$ with $\boldsymbol{\tau}^\ast = \nabla\times \boldsymbol{\psi}^\ast$ and $\boldsymbol{\psi}^\ast \in \boldsymbol{W}^{1,q}(\Omega) $. Meanwhile, Algorithm \ref{alg:dualform} outputs $\boldsymbol{\sigma}_{\widehat{\theta},\widehat{\eta}}=\nabla\phi_{\widehat{\theta}}+\nabla\times\boldsymbol{\psi}_{\widehat{\eta}}$ through problems \eqref{optimization1} and \eqref{optimization2} posed over NN classes $\mathcal{N}_\phi$ and $\mathcal{N}_{\boldsymbol{\psi}}$. Thus we decompose the error $\boldsymbol{\sigma}^\ast - \boldsymbol{\sigma}_{\widehat{\theta},\widehat{\eta}}$ into
\begin{equation}\label{equ:error_decom}
    \boldsymbol{\sigma}^\ast - \boldsymbol{\sigma}_{\widehat{\theta},\widehat{\eta}} = \nabla( \phi^\ast -\phi_{\widehat{\theta}}) + \nabla \times( \boldsymbol{\psi}^\ast - \boldsymbol{\psi}_{\widehat{\eta}}). 
\end{equation}
The first term can be directly bounded using Lemma \ref{lem:decomp}. However, the loss $L_{\rm s}$ in \eqref{eqn:emploss2} uses the numerical solution $\phi_{\widehat{\theta}}$ to approximate $\phi^*$. Thus the error $\nabla \times (\boldsymbol{\psi}^\ast - \boldsymbol{\psi}_{\widehat{\eta}})$ arises not only from the NN approximation and the quadrature, but also from the loss perturbation. To deal with the latter, we employ an intermediate problem:
\begin{equation}\label{min:uncons_inter1}
\widehat{\boldsymbol{\tau}}=\mathop{{\rm argmin}}\limits_{\boldsymbol{\tau}\in \bold{W}^q({\rm div0},\Omega)}\tfrac{1}{q}\|\nabla\phi_{\widehat{\theta}}+\boldsymbol{\tau}\|_{L^q(\Omega)}^q+\langle (\nabla\phi_{\widehat{\theta}}+\boldsymbol{\tau})\cdot \boldsymbol{n} , g \rangle,
\end{equation}
with the approximation $\phi_{\widehat{\theta}}$ given by problem \eqref{optimization1}. By repeating the proof of Theorem~\ref{thm:unique-sol_min}, the unique solvability of problem \eqref{min:uncons_inter1} holds. Further, by Lemma \ref{lem:represent}, the minimizer $\widehat{\boldsymbol{\tau}}$ can be represented by $\widehat{\boldsymbol{\tau}}=\nabla\times \widehat{\boldsymbol{\psi}}$ for some $\widehat{\boldsymbol{\psi}}\in \boldsymbol{W}^{1,q} (\Omega)$ and satisfies 
\begin{equation}\label{min:loss2prac}
    L_{\rm s}(\widehat{\bold{\psi}}) \leq L_{\rm s}(\bold{\psi}),\quad \forall \bold{\psi}\in \bold{W}^{1,q}(\Omega).
\end{equation} 
Then we have the following decomposition
\begin{equation}\label{equ:error_decom_s}
        \nabla \times (\boldsymbol{\psi}^\ast -  \boldsymbol{\psi}_{\widehat{\eta}}) = \nabla \times (\boldsymbol{\psi}^\ast - \widehat{\bold{\psi}}) + \nabla\times (\widehat{\bold{\psi}} - \boldsymbol{\psi}_{\widehat{\eta}}).
\end{equation}
Using Lemma \ref{lem:represent} again, we rewrite the functional $\mathcal{J}$ on $\bold{V}_q$ in \eqref{min:unconstrained} as 
\begin{equation}\label{eqn:loss2}
L_{\rm s}^*(\boldsymbol{\psi}) = \tfrac{1}{q} \|\nabla\phi^*+\nabla\times\boldsymbol{\psi}\|_{L^q(\Omega)}^q+ \langle (\nabla\phi^*+\nabla\times\boldsymbol{\psi})\cdot \boldsymbol{n} , g \rangle,\quad \forall \bold{\psi}\in \bold{W}^{1,q}(\Omega).
\end{equation}  
Since $\bold{\sigma}^\ast = \nabla \times \bold{\psi}^\ast$ is a minimizer of problem \eqref{min:unconstrained}, $L_{\rm s}^*(\boldsymbol{\psi}^\ast)\leq L_{\rm s}^*(\widehat{\boldsymbol{\psi}})$. Actually, this inequality can be refined using sharper estimates for the $\bold{L}^q(\Omega)$-norm in Proposition \ref{prop:convexsmooth} below. The proof of Proposition \ref{prop:convexsmooth} relies on several basic inequalities for vectors \cite[Lemma 2.2]{doi:10.1137/0731022}.

\begin{lemma}\label{lemma:pinequalities}
For all \( \xi, \eta \in \mathbb{R}^d \) with \( d \geq 1 \), there exists  $c=c(p,d)$ such that
\begin{align}\label{ineq:p<2convex}
       (|\xi|^{p-2} \xi - |\eta|^{p-2} \eta) \cdot (\xi - \eta) &\geq c |\xi - \eta|^{2}(|\xi| + |\eta|)^{p-2},\quad 1<p\leq 2,\\
\label{ineq:p<2smooth}
        | |\xi|^{p-2} \xi - |\eta|^{p-2} \eta| &\leq c |\xi - \eta|^{p-1},\quad 1<p\leq 2,\\
\label{ineq:p>2convex}
        (|\xi|^{p-2} \xi - |\eta|^{p-2} \eta) \cdot (\xi - \eta) &\geq c |\xi - \eta|^{p},\quad p\geq 2,\\
\label{ineq:p>2smooth}
        | |\xi|^{p-2} \xi - |\eta|^{p-2} \eta| &\leq c |\xi - \eta|(|\xi| + |\eta|)^{p-2},\quad p\geq 2. 
\end{align}
\end{lemma}

\begin{proposition}\label{prop:convexsmooth}
let
$S(\bold{u},\bold{v}):=\int_{0}^1 (\|\bold{v}\|_{\bold{L}^q(\Omega)}^q + \|\bold{v} + t(\bold{u}-\bold{v})\|_{\bold{L}^q(\Omega)}^q )^{1-\frac{2}{q}} t {\rm d}t$ for any $ \bold{u},\bold{v}\in \bold{L}^q(\Omega)$.
Then the following statements hold.  
\begin{itemize}
    \item when $1<q<2$, 
\begin{subequations}
    \begin{equation}\label{est:q<2convex}
    \tfrac{1}{q}\|\bold{u}\|_{\bold{L}^q(\Omega)}^q \ge \tfrac{1}{q}\|\bold{v}\|_{\bold{L}^q(\Omega)}^q + (|\bold{v}|^{q-2}\bold{v}, \bold{u}-\bold{v})  + c(q)  S(\bold{u},\bold{v}) \|\bold{u}-\bold{v}\|_{\bold{L}^q(\Omega)}^2, 
        \end{equation}
    \begin{equation}\label{est:q<2smooth}
            \tfrac{1}{q}\|\bold{u}\|_{\bold{L}^q(\Omega)}^q \le \tfrac{1}{q}\|\bold{v}\|_{\bold{L}^q(\Omega)}^q + ( |\bold{v}|^{q-2}\bold{v} ,\bold{u}-\bold{v} ) + c(q) \|\bold{u}-\bold{v}\|_{\bold{L}^q(\Omega)}^q;
        \end{equation}
    \end{subequations}  
    \item when $q\geq 2$, 
\begin{subequations}
\begin{equation}\label{est:q>=2convex}
    \tfrac{1}{q}\|\bold{u}\|_{\bold{L}^q(\Omega)}^q \ge \tfrac{1}{q}\|\bold{v}\|_{\bold{L}^q(\Omega)}^q + (|\bold{v}|^{q-2}\bold{v},\bold{u}-\bold{v} )  + c(q) \|\bold{u}-\bold{v}\|_{\bold{L}^q(\Omega)}^q,
        \end{equation}
\begin{equation}\label{est:q>=2smooth}
             \tfrac{1}{q}\|\bold{u}\|_{\bold{L}^q(\Omega)}^q \le \tfrac{1}{q}\|\bold{v}\|_{\bold{L}^q(\Omega)}^q + ( |\bold{v}|^{q-2}\bold{v}, \bold{u}-\bold{v} )+ c(q)S(\bold{u},\bold{v})\|\bold{u}-\bold{v}\|_{\bold{L}^q(\Omega)}^2.
        \end{equation}
    \end{subequations}
\end{itemize}
\end{proposition}
\begin{proof}
Let \(\bold{w} = \bold{u} - \bold{v}\) and $\phi(t)=\tfrac{1}{q}\|\bold{v}+t\bold{w}\|^q_{\bold{L}^q(\Omega)}$ for $t\in[0,1]$. By the chain rule, there holds $\phi'(t)=(|\bold{v}+t\bold{w}|^{q-2} (\bold{v}+t\bold{w}), \bold{w})$.  The fundamental theorem of calculus implies  
\begin{align}
    &\tfrac{1}{q}\|\bold{u}\|_{\bold{L}^q(\Omega)}^q - \tfrac{1}{q}\|\bold{v}\|_{\bold{L}^q(\Omega)}^q  =  \phi'(0) + \int_{0}^1 (\phi'(t) - \phi'(0)){\rm d}t
   =(|\bold{v}|^{q-2} \bold{v}, \bold{w}) + \int_0^1 {\rm I}(t)\,{\rm d}t,\label{lem:convexsmooth_pf01}
    \end{align}
with ${\rm I}(t)= (|\bold{v}+t\bold{w}|^{q-2} (\bold{v}+t\bold{w})-|\bold{v}|^{q-2} \bold{v}, \bold{w})$. It remains to suitably bound the term ${\rm I}(t)$.
In the case $1<q<2$, by the estimate \eqref{ineq:p<2convex}, we have
\[
    {\rm I}(t)  \geq c(q)t(|\bold{w}|^2,|\bold{v}| + |\bold{v}+t\bold{w}|)^{q-2}.
\]
By the triangle inequality, $|\bold{w}|^2(|\bold{v}|+|\bold{v}+t\bold{w}|)^{-(2-q)} \leq t^{-(2-q)}|\bold{w}|^q$. Hence $\int_{0}^1\int_\Omega |\bold{w}|^2(|\bold{v}|+|\bold{v}+t\bold{w}|)^{-(2-q)}   {\rm d}x {\rm d}t  < \infty$. By H\"older's inequality (with exponents $\frac{2}{q}$ and $\frac{2}{2-q}$), there holds
\begin{align*}
\|\bold{w}\|_{\bold{L}^q(\Omega)}^q=&( |\bold{w}|^q(|\bold{v}|+|\bold{v}+t\bold{w}|)^{-(2-q)\frac{q}{2}},(|\bold{v}|+|\bold{v}+t\bold{w}|)^{(2-q)\frac{q}{2}})\\
\leq &\||\bold{w}|^2(|\bold{v}|+|\bold{v}+t\bold{w}|)^{-(2-q)}\|_{L^1(\Omega)}^{\frac{q}{2}}\||\bold{v}|+|\bold{v}+t\bold{w}|\|_{L^q(\Omega)}^{\frac{q(2-q)}{2}}.
\end{align*}
Collecting the last two estimates and using the inequalities $(a+b)^{q}\leq\sqrt{2}(a^2+b^2)^{\frac{q}{2}}\leq \sqrt{2}(a^{q}+b^{q})$ for any $a,b\geq 0$ ($\frac{q}{2}\in (0,1)$), we obtain
\begin{align*}
\|\bold{w}\|_{\bold{L}^q(\Omega)}^q\leq& \big(\tfrac{1}{c(q)t}{\rm I}(t) \big)^\frac{q}{2}\big(\|\sqrt{2}(|\bold{v}|^q+|\bold{v}+t\bold{w}|^q)\|_{L^1(\Omega)})^{\frac{2-q}{2}},
\end{align*}
i.e., 
$c(q) t( \|\bold{v}\|_{\bold{L}^q(\Omega)}^q + \|\bold{v} + t\bold{w}\|_{\bold{L}^q(\Omega)}^q )^{1-\frac{2}{q}}  \|\bold{w}\|^2_{\bold{L}^q(\Omega)} \leq {\rm I}(t).$ Now integrating in \(t\) over $(0,1)$ in \eqref{lem:convexsmooth_pf01} yields \eqref{est:q<2convex}.
Next, by the estimate \eqref{ineq:p<2smooth}, we have
\[
   {\rm I}(t) \le \|c(q)|t\bold{w}|^{q-1}|\bold{w}|\|_{L^1(\Omega)}= c(q)t^{q-1}\|\bold{w}\|_{\bold{L}^q(\Omega)}^q,
\]
which implies \eqref{est:q<2smooth} after integrating in $t$ over $(0,1)$ in \eqref{lem:convexsmooth_pf01}. When $q\geq 2$,  \eqref{ineq:p>2convex} implies
\[
   {\rm I}(t) \ge \|c(q)\tfrac{1}{t}|t\bold{w}|^{q}\|_{L^1(\Omega)}= c(q)t^{q-1}\|\bold{w}\|_{\bold{L}^q(\Omega)}^q.
\] 
This directly leads to \eqref{est:q>=2convex}. Finally, by the estimate \eqref{ineq:p>2smooth}, we get
\[
    {\rm I}(t)  \le c(q) t\|\bigl(|\bold{v}+t\bold{w}|+|\bold{v}|\bigr)^{q-2}|\bold{w}|^2\|_{L^1(\Omega)}.
\]
Hölder's inequality (with exponents \(\frac{q}{q-2}\) and \(\frac{q}{2}\)) and  $(a+b)^q\leq 2^{q-1}(a^q+b^q)$ with $q>1$ leads to
\begin{align*}
    &\| \bigl(|\bold{v}+t\bold{w}|+|\bold{v}|\bigr)^{q-2}|\bold{w}|^2\|_{L^1(\Omega)}
\le \big(\| |\bold{v}+t\bold{w}|+|\bold{v}|\|_{L^q(\Omega)}\big)^{{q-2}}\|\bold{w}\|_{\bold{L}^q(\Omega)}^{2}\\
\leq& 2^{(q-1)(1-\frac{2}{q})}(\|\bold{v}+t\bold{w}\|_{\bold{L}^q(\Omega)}^{q}+\|\bold{v}\|_{\bold{L}^q(\Omega)}^{q})^{1-\frac{2}{q}}\|\bold{w}\|_{\bold{L}^q(\Omega)}^2. 
\end{align*}
The last two estimates and integrating in $t$ over $(0,1)$ complete the proof of the proposition.
\end{proof}

Let $\nabla\times\boldsymbol{\psi}_{\eta^*}$ denote the best approximation of $\boldsymbol{\tau}^*$ over $\mathcal{N}_{\boldsymbol{\psi}}$ in the  $\bold{L}^q(\Omega)$-norm:
\begin{equation}\label{min:best-approximation}
\boldsymbol{\psi}_{\eta^*}=\operatorname*{argmin}_{\boldsymbol{\psi}_\eta\in\mathcal{N}_{\boldsymbol{\psi}}}\|\nabla\times\boldsymbol{\psi}_{\eta}-\boldsymbol{\tau}^*\|_{\bold{L}^q(\Omega)}.
\end{equation}
Note that problem \eqref{min:best-approximation} attains its minimum since $\|\nabla\times\boldsymbol{\psi}_{\eta}-\boldsymbol{\tau}^*\|_{\bold{L}^q(\Omega)}$ is continuous in $\eta$, whose components are bounded by $B$.
In the analysis below, we impose the following condition.
\begin{assumption}\label{assumption1}
$\|\phi_{\widehat{\theta}}-\phi^*\|_{W^{1,q}(\Omega)} \leq 1$, $ \|\nabla\times\boldsymbol{\psi}_{\eta^*}-\boldsymbol{\tau}^*\|_{\bold{L}^q(\Omega)}\leq 1$ and $ \sup_{\boldsymbol{\psi}_{\eta} \in \mathcal{N}_{\boldsymbol{\psi}}}\big|L_{\rm s}(\boldsymbol{\psi}_{\eta})-\widehat{L}_{\rm s}(\boldsymbol{\psi}_{\eta})\big|\leq 1$.
\end{assumption}
\begin{proposition}\label{prop:bound}
Under Assumption \ref{assumption1}, with $c=c(q,f,g,\Omega)$, there holds
\begin{equation*}
    \max\{\|\nabla\phi_{\widehat{\theta}}\|_{\bold{L}^q(\Omega)},\|\nabla\times\widehat{\boldsymbol{\psi}}\|_{\bold{L}^q(\Omega)}, \|\nabla\times\boldsymbol{\psi}_{\eta^*}\|_{\bold{L}^q(\Omega)},\|\nabla\times\bold{\psi}_{\widehat{\eta}}\|_{\bold{L}^q(\Omega)},\|\nabla\phi^*\|_{\bold{L}^q(\Omega)},\|\boldsymbol{\tau}^*\|_{\bold{L}^q(\Omega)}\} \leq c.
\end{equation*}
\end{proposition}
\begin{proof}
Since $\|\phi_{\widehat{\theta}}-\phi^*\|_{W^{1,q}(\Omega)}\leq1$, the elliptic regularity theory \cite[Lemma 9.17]{GT:2001} yields
\begin{equation}\label{prop:bound_pf01}
            \|\nabla\phi_{\widehat{\theta}}\|_{\bold{L}^q(\Omega)}\leq \|\phi^*\|_{W^{1,q}(\Omega)}+1\leq c\|f\|_{L^{q}(\Omega)}+1.    
    \end{equation}
Using the estimate \eqref{est:stab}, we also get 
\begin{equation}\label{prop:bound_pf02}
\|\nabla\times\boldsymbol{\psi}_{\eta^*}\|_{\bold{L}^q(\Omega)}\leq\|\boldsymbol{\tau}^*\|_{\bold{L}^q(\Omega)}+1\leq c(q,\Omega)(\|f\|_{L^q(\Omega)}+\|g\|_{W^{\frac1q,p}(\partial\Omega)}^{\frac pq})+1.
    \end{equation}
By repeating the argument of Theorem \ref{thm:equivalence} and noting \eqref{prop:bound_pf01}, we obtain 
\begin{equation}\label{prop:bound_pf03}
\|\nabla\times\widehat{\boldsymbol{\psi}}\|_{\bold{L}^q(\Omega)}\leq c(q,\Omega)(\|\nabla\phi_{\widehat{\theta}}\|_{\bold{L}^q(\Omega)}+\|g\|_{W^{\frac{1}{q},p}(\partial\Omega)}^{\frac{p}{q}}) \leq c(q,\Omega)( \|f\|_{L^q(\Omega)} + \|g\|_{W^{\frac{1}{q},p}(\partial\Omega)}^{\frac{p}{q}} + 1).
\end{equation}
It remains to prove that $\|\nabla\times\bold{\psi}_{\widehat{\eta}}\|_{\bold{L}^q(\Omega)}$ is bounded. By the minimizing property of $\boldsymbol{\psi}_{\widehat{\eta}}$ to $\widehat{L}_{\rm s}(\boldsymbol{\psi}_\eta)$ over the set  $\mathcal{N}_{\boldsymbol{\psi}}$, we have 
\begin{align}
L_{\rm s}(\boldsymbol{\psi}_{\widehat{\eta}})-L_{\rm s}(\widehat{\boldsymbol{\psi}})
=&[L_{\rm s}(\boldsymbol{\psi}_{\widehat{\eta}})-\widehat{L}_{\rm s}(\boldsymbol{\psi}_{\widehat{\eta}})]+[\widehat{L}_{\rm s}(\boldsymbol{\psi}_{\widehat{\eta}})-\widehat{L}_{\rm s}(\boldsymbol{\psi}_{\eta^*})]+[\widehat{L}_{\rm s}(\boldsymbol{\psi}_{\eta^*})-L_{\rm s}(\boldsymbol{\psi}_{\eta^*})]
\nonumber\\
& +[L_{\rm s}(\boldsymbol{\psi}_{\eta^*})-L_{\rm s}(\widehat{\boldsymbol{\psi}})]\leq 2\sup_{\psi_{\eta} \in \mathcal{N}_{\boldsymbol{\psi}}}\big|L_{\rm s}(\boldsymbol{\psi}_{\eta})-\widehat{L}_{\rm s}(\boldsymbol{\psi}_{\eta})\big|+L_{\rm s}(\boldsymbol{\psi}_{\eta^*})-L_{\rm s}(\widehat{\boldsymbol{\psi}}),\label{eqn:thm1errordecom}
\end{align}
where $\bold{\psi}_{\eta^\ast}$ is the minimizer to problem \eqref{min:best-approximation}. 
The proof extensively uses the duality map $D_q: \bold{L}^q(\Omega) \to \bold{L}^p(\Omega)$, defined by $
D_q(\boldsymbol{v}) = |\boldsymbol{v}|^{q-2}\boldsymbol{v}$, and consists of four steps. First, we bound $L_{\rm s}(\boldsymbol{\psi}_{\widehat{\eta}})-L_{\rm s}(\widehat{\boldsymbol{\psi}})$ from below using $S(\nabla \phi_{\widehat{\theta}} + \nabla\times\boldsymbol{\psi}_{\widehat{\eta}}, \nabla \phi_{\widehat{\theta}} + \nabla\times\widehat{\boldsymbol{\psi}})\|\nabla\times(\boldsymbol{\psi}_{\widehat{\eta}}-\widehat{\boldsymbol{\psi}})\|_{\bold{L}^q(\Omega)}$ when $1<q<2$ and $\|\nabla\times(\boldsymbol{\psi}_{\widehat{\eta}}-\widehat{\boldsymbol{\psi}})\|_{\bold{L}^q(\Omega)}$ when $q\geq2$. Then we bound \eqref{eqn:thm1errordecom} from above by $c(q,f,g,\Omega)$, which implies the boundedness of $\|\nabla\times\boldsymbol{\psi}_{\widehat{\eta}}\|_{\bold{L}^q(\Omega)}$ when $q\geq2$. Last, we show that $S(\nabla \phi_{\widehat{\theta}} + \nabla\times\boldsymbol{\psi}_{\widehat{\eta}}, \nabla \phi_{\widehat{\theta}} + \nabla\times\widehat{\boldsymbol{\psi}})\|\nabla\times(\boldsymbol{\psi}_{\widehat{\eta}}-\widehat{\boldsymbol{\psi}})\|_{\bold{L}^q(\Omega)}$ is larger than a strictly increasing function of $\|\nabla\times(\boldsymbol{\psi}_{\widehat{\eta}}-\widehat{\boldsymbol{\psi}})\|_{\bold{L}^q(\Omega)}$, from which we deduce the boundedness when $1<q<2$.

\noindent {\it Step 1.} Lower bound on $L_{\rm s}(\boldsymbol{\psi}_{\widehat{\eta}})-L_{\rm s}(\widehat{\boldsymbol{\psi}})$. Inequalities \eqref{est:q<2convex} and \eqref{est:q>=2convex} imply
\begin{equation}\label{thm:errordecomp_pf02}
\tfrac{1}{q}\|\nabla\phi_{\widehat{\theta}}+\nabla\times\boldsymbol{\psi}_{\widehat{\eta}}\|_{\bold{L}^q(\Omega)}^q  \geq \tfrac{1}{q}\|\nabla\phi_{\widehat{\theta}}+\nabla\times\widehat{\boldsymbol{\psi}}\|_{\bold{L}^q(\Omega)}^q+( \widehat{d}, \widehat{e}_{\bold\psi}) +c_1\|\widehat{e}_{\bold\psi}\|_{\bold{L}^q(\Omega)}^{\max\{q,2\}},
\end{equation}
with $\widehat{e}_{\bold\psi} = \nabla\times(\boldsymbol{\psi}_{\widehat{\eta}}-\widehat{\boldsymbol{\psi}})$, $$
\widehat{d} = D_q(\nabla\phi_{\widehat{\theta}} + \nabla\times\widehat{\boldsymbol{\psi}}) \quad \mbox{and}\quad
c_1=\left\{\begin{aligned}
    &c(q)S( \nabla\phi_{\widehat{\theta}} + \nabla\times\boldsymbol{\psi}_{\widehat{\eta}}, \nabla\phi_{\widehat{\theta}} + \nabla\times\widehat{\boldsymbol{\psi}}), &&1<q<2,\\
    &c(q), &&q\geq2.
\end{aligned}\right.$$
Since $\widehat{\boldsymbol{\psi}}$ minimizes $L_{\rm s}(\boldsymbol{\psi})$ over $\bold{W}^{1,q}(\Omega)$ (cf. \eqref{min:loss2prac}),  it satisfies the Euler-Lagrange equation
\begin{equation}\label{eqn:deri0}
(\widehat{d}, \nabla \times \boldsymbol{\psi})+\langle \nabla \times \bold{\psi}  \cdot\bold{n} , g \rangle=0,\quad \forall \bold{\psi} \in \bold{W}^{1,q}(\Omega).
\end{equation}
Then adding the term $\langle(\nabla\phi_{\widehat{\theta}}+\nabla\times\boldsymbol{\psi}_{\widehat{\eta}})\cdot\bold{n},g\rangle$ to \eqref{thm:errordecomp_pf02} and using \eqref{eqn:deri0} with $\bold{\psi}=  \boldsymbol{\psi}_{\widehat{\eta}}-\widehat{\boldsymbol{\psi}}$ yield
\begin{align*}
L_{\rm s}(\boldsymbol{\psi}_{\widehat{\eta}})
\geq& \tfrac{1}{q}\|\nabla\phi_{\widehat{\theta}}+\nabla\times\widehat{\boldsymbol{\psi}}\|_{\bold{L}^q(\Omega)}^q+\langle(\nabla\phi_{\widehat{\theta}}+\nabla\times\widehat{\boldsymbol{\psi}})\cdot\bold{n},g\rangle+c_1\|\widehat{e}_{\bold\psi}\|_{\bold{L}^q(\Omega)}^{\max\{q,2\}}\\
&+ [(\widehat{d}, \widehat{e}_{\bold\psi})+\langle(\widehat{e}_{\bold\psi}\cdot\bold{n},g\rangle]
=L_{\rm s}(\widehat{\boldsymbol{\psi}})+c_1\| \widehat{e}_{\bold\psi}\|_{\bold{L}^q(\Omega)}^{\max\{q,2\}}.
\end{align*}
Thus the following estimate holds
\begin{equation}\label{eqn:thm1convex}
c_1\|\widehat{e}_{\bold\psi}\|_{\bold{L}^q(\Omega)}^{\max\{q,2\}}\leq L_{\rm s}(\boldsymbol{\psi}_{\widehat{\eta}})-L_{\rm s}(\widehat{\boldsymbol{\psi}}).
\end{equation}
{\it Step 2.} Upper bound on $L_{\rm s}(\boldsymbol{\psi}_{\eta^\ast})-L_{\rm s}(\widehat{\boldsymbol{\psi}})$. By \eqref{est:q<2smooth} and \eqref{est:q>=2smooth} , we deduce
\begin{equation}\label{eqn:qqqqq}
\tfrac{1}{q}\|\nabla\phi_{\widehat{\theta}}+\nabla\times\boldsymbol{\psi}_{\eta^*}\|_{\bold{L}^q(\Omega)}^q \leq \tfrac{1}{q}\|\nabla\phi_{\widehat{\theta}}+\nabla\times\widehat{\boldsymbol{\psi}}\|_{\bold{L}^q(\Omega)}^q+(\widehat{d}, e_{\bold\psi}^*)+c_2\| e_{\bold\psi}^*\|_{\bold{L}^q(\Omega)}^{\min\{q,2\}},
\end{equation}
with $$
e_{\bold\psi}^*=\nabla\times(\boldsymbol{\psi}_{\eta^*}-\widehat{\boldsymbol{\psi}}) \quad\mbox{and}\quad c_2=\left\{\begin{aligned}
    &c(q), &&1<q<2,\\
    &c(q)S( \nabla \phi_{\widehat{\theta}} + \nabla\times\boldsymbol{\psi}_{\eta^*}, \nabla \phi_{\widehat{\theta}} + \nabla\times\widehat{\boldsymbol{\psi}}), &&q\geq2.
\end{aligned}\right.$$
By the estimates \eqref{prop:bound_pf01}, \eqref{prop:bound_pf02} and \eqref{prop:bound_pf03}, $\|\nabla\phi_{\widehat{\theta}}\|_{\bold{L}^q(\Omega)}$, $\|\nabla\times\boldsymbol{\psi}_{\eta^*}\|_{\bold{L}^q(\Omega)}$ and $\|\nabla\times\widehat{\boldsymbol{\psi}}\|_{\bold{L}^q(\Omega)}$ are uniformly bounded. Thus we can bound the factor $S(\nabla\phi_{\widehat{\theta}} + \nabla\times\boldsymbol{\psi}_{\eta^*},\nabla\phi_{\widehat{\theta}} + \nabla\times\widehat{\boldsymbol{\psi}})$ for $q\geq 2$ by
\[
    S(\nabla\phi_{\widehat{\theta}} +\nabla\times\boldsymbol{\psi}_{\eta^*},\nabla\phi_{\widehat{\theta}} +\nabla\times\widehat{\boldsymbol{\psi}}) = \int_{0}^1 (\|\nabla\phi_{\widehat{\theta}} + \nabla\times\widehat{\boldsymbol{\psi}}\|_{\bold{L}^q(\Omega)}^q + \|\nabla\phi_{\widehat{\theta}} + \nabla\times\widehat{\boldsymbol{\psi}} + t e_{\bold\psi}^*\|_{\bold{L}^q(\Omega)}^q )^{1-\frac{2}{q}} t {\rm d}t    \leq c, 
\]
with $c=c(q,f,g,\Omega)$. Consequently, from \eqref{eqn:qqqqq}, we deduce
\begin{equation*}
\begin{aligned}
\tfrac{1}{q}\|\nabla\phi_{\widehat{\theta}}+\nabla\times\boldsymbol{\psi}_{\eta^*}\|_{\bold{L}^q(\Omega)}^q & \leq \tfrac{1}{q}\|\nabla\phi_{\widehat{\theta}}+\nabla\times\widehat{\boldsymbol{\psi}}\|_{\bold{L}^q(\Omega)}^q+( \widehat{d},  e_{\bold\psi}^*) + c \| e_{\bold\psi}^*\|_{\bold{L}^q(\Omega)}^{\min\{q,2\}},
\end{aligned}
\end{equation*}
with $c=c(q,f,g,\Omega)$ for $q\in (1,\infty)$. Then adding the term $\langle(\nabla\phi_{\widehat{\theta}}+\nabla\times\boldsymbol{\psi}_{\eta^*})\cdot\bold{n},g\rangle$ to the inequality and using the Euler-Lagrange equation \eqref{eqn:deri0} (with $\bold\psi=\bold{\psi}_{\eta^*} - \widehat{\bold{\psi}}$) give
\begin{align*}
L_{\rm s}(\boldsymbol{\psi}_{\eta^*})
\leq& \tfrac{1}{q}\|\nabla\phi_{\widehat{\theta}}+\nabla\times\widehat{\boldsymbol{\psi}}\|_{\bold{L}^q(\Omega)}^q+(\widehat{d}, e_{\bold\psi}^*) + c \| e_{\bold\psi}^*\|_{\bold{L}^q(\Omega)}^{\min\{q,2\}}
+\langle(\nabla\phi_{\widehat{\theta}}+\nabla\times\boldsymbol{\psi}_{\eta^*})\cdot\bold{n},g\rangle\\
=&L_{\rm s}(\widehat{\boldsymbol{\psi}})+c\| e_{\bold\psi}^*\|_{\bold{L}^q(\Omega)}^{\min\{q,2\}}.
\end{align*}
This directly implies 
\begin{equation}\label{eqn:thm1conti}
c\|\nabla\times(\boldsymbol{\psi}_{\eta^*}-\widehat{\boldsymbol{\psi}})\|_{\bold{L}^q(\Omega)}^{\min\{q,2\}}\geq L_{\rm s}(\boldsymbol{\psi}_{\eta^*})-L_{\rm s}(\widehat{\boldsymbol{\psi}}).
\end{equation}
Since $\bold{\psi}_{\eta^\ast}$ is the minimizer to problem \eqref{min:best-approximation}, using \eqref{eqn:thm1errordecom}, \eqref{eqn:thm1convex}, \eqref{eqn:thm1conti} and the triangle inequality, we obtain with $c=c(q,f,g,\Omega)$, when $q\geq2$,  
\begin{align}\label{eqn:psietaq>2}
\|\widehat{e}_{\bold\psi}\|_{\bold{L}^q(\Omega)}
    \leq c\big(\|\boldsymbol{\tau}^*-\nabla\times\widehat{\boldsymbol{\psi}}\|_{\bold{L}^q(\Omega)}^{\frac{2}{q}}+\inf_{\bold{\psi}_{\eta} \in \mathcal{N}_{\bold{\psi}}}\|\nabla\times\boldsymbol{\psi}_{\eta}-\boldsymbol{\tau}^*\|_{\bold{L}^q(\Omega)}^{\frac{2}{q}}+\sup_{\bold{\psi}_{\eta} \in \mathcal{N}_{\boldsymbol{\psi}}}\big|L_{\rm s}(\boldsymbol{\psi}_{\eta})-\widehat{L}_{\rm s}(\boldsymbol{\psi}_{\eta})\big|^{\frac{1}{q}}\big),
    \end{align}
and when $1<q<2$,
\begin{align}\label{eqn:psietaq<2}
&S(\nabla \phi_{\widehat{\theta}} + \nabla\times\boldsymbol{\psi}_{\widehat{\eta}}, \nabla \phi_{\widehat{\theta}} + \nabla\times\widehat{\boldsymbol{\psi}})^{\frac{1}{2}}\|\widehat{e}_{\bold\psi}\|_{\bold{L}^q(\Omega)}\\
\leq&c\big(\|\boldsymbol{\tau}^*-\nabla\times\widehat{\boldsymbol{\psi}}\|_{\bold{L}^q(\Omega)}^{\frac{q}{2}}+\inf_{\bold{\psi}_\eta\in \mathcal{N}_{\bold{\psi}}}\|\nabla\times\boldsymbol{\psi}_{\eta}-\boldsymbol{\tau}^*\|_{\bold{L}^q(\Omega)}^{\frac{q}{2}}+\sup_{\bold{\psi}_{\eta} \in \mathcal{N}_{\boldsymbol{\psi}}}\big|L_{\rm s}(\boldsymbol{\psi}_{\eta})-\widehat{L}_{\rm s}(\boldsymbol{\psi}_{\eta})\big|^{\frac{1}{2}}\big).\nonumber
\end{align}
{\it Step 3.} Boundedness of  $\|\bold{\psi}_{\widehat{\eta}}\|_{\bold{L}^q(\Omega)}$ for $q\geq 2$. 
Note $\bold{\tau}^\ast = \nabla \times \bold{\psi}^\ast$. By the minimizing property of $\widehat{\boldsymbol{\psi}}$ to the loss $L_{\rm s}(\boldsymbol{\psi})$, there holds
\begin{align}
L_{\rm s}^*(\widehat{\boldsymbol{\psi}})-L_{\rm s}^*(\boldsymbol{\psi}^*)= &[L_{\rm s}^*(\widehat{\boldsymbol{\psi}})-L_{\rm s}(\widehat{\boldsymbol{\psi}})]+[L_{\rm s}(\widehat{\boldsymbol{\psi}})-L_{\rm s}(\boldsymbol{\psi}^*)]+[L_{\rm s}(\boldsymbol{\psi}^*)-L_{\rm s}^*(\boldsymbol{\psi}^*)]\nonumber\\
\leq &[L_{\rm s}^*(\widehat{\boldsymbol{\psi}})-L_{\rm s}(\widehat{\boldsymbol{\psi}})] +[L_{\rm s}(\boldsymbol{\psi}^*)-L_{\rm s}^*(\boldsymbol{\psi}^*)].\label{eqn:errordecom2}
\end{align}
Repeating the argument for \eqref{eqn:thm1convex} leads to
\begin{equation}\label{est:loss-ast_convex}
c\|\nabla\times\widehat{\boldsymbol{\psi}}-\boldsymbol{\tau}^*\|_{\bold{L}^{q}(\Omega)}^{\max\{q,2\}}\leq L_{\rm s}^*(\widehat{\boldsymbol{\psi}})-L_{\rm s}^*(\boldsymbol{\psi}^*).
\end{equation}
Next, by the definitions of $L_{\rm s}^\ast(\widehat{\boldsymbol{\psi}})$ and $L_{\rm s}(\widehat{\boldsymbol{\psi}})$, \eqref{prop:bound_pf01}, \eqref{prop:bound_pf03}, the trace theorem for  $\bold{X}_{q}$ and the assumption $\|\phi_{\widehat{\theta}}-\phi^*\|_{W^{1,q}(\Omega)}\leq1$, with $\widehat{e}_\phi=\phi^*-\phi_{\widehat\theta}$, by \eqref{est:q<2smooth} and \eqref{est:q>=2smooth}, we find 
\begin{align*}
L_{\rm s}^*(\widehat{\boldsymbol{\psi}})-L_{\rm s}(\widehat{\boldsymbol{\psi}})
= &\tfrac{1}{q}\|\nabla\phi^*+\nabla\times\widehat{\boldsymbol{\psi}}\|_{\bold{L}^q(\Omega)}^q-\tfrac{1}{q}\|\nabla\phi_{\widehat{\theta}} + \nabla\times\widehat{\boldsymbol{\psi}}\|_{\bold{L}^q(\Omega)}^q+\langle \nabla \widehat{e}_\phi\cdot\bold{n},g\rangle \\
\leq & (\nabla \widehat{e}_\phi, \widehat{d})+c\|\nabla \widehat{e}_\phi\|_{\bold{L}^q(\Omega)}^{\min\{q,2\}}+\langle \nabla \widehat{e}_\phi\cdot\bold{n},g\rangle\\
=&\|\nabla \widehat{e}_\phi\|_{\bold{L}^q( \Omega)}\|\nabla\phi_{\widehat{\theta}} + \nabla\times\widehat{\boldsymbol{\psi}}\|_{\bold{L}^q( \Omega)}^{q-1}+c\|\nabla \widehat{e}_\phi\|_{\bold{L}^q(\Omega)}^{\min\{q,2\}}+\langle \nabla \widehat{e}_\phi\cdot\bold{n},g\rangle\\
\leq& c\|\widehat{e}_\phi\|_{W^{1,q}(\Omega)} \leq c(q,f,g).
\end{align*}
Similarly,
$L_{\rm s}(\boldsymbol{\psi}^*)-L_{\rm s}^*(\boldsymbol{\psi}^*) \leq  c (q,f,g)$. These two estimates, together with \eqref{eqn:errordecom2} and \eqref{est:loss-ast_convex}, yield
\begin{equation}\label{prop:bound_pf05}
\|\nabla\times\widehat{\boldsymbol{\psi}}-\boldsymbol{\tau}^*\|_{\bold{L}^{q}(\Omega)}^{\max\{q,2\}}\leq c(q,f,g).
\end{equation}
By \eqref{prop:bound_pf05}, \eqref{eqn:psietaq>2} with $\bold{\psi}_{\eta}=\bold{0}$ in $\|\nabla\times\boldsymbol{\psi}_{\eta}-\boldsymbol{\tau}^*\|_{\bold{L}^q(\Omega)}^{\frac{2}{q}}$, \eqref{est:stab} and Assumption \ref{assumption1}, we get $\|\widehat{e}_{\bold\psi}\|_{\bold{L}^q(\Omega)}\leq c(q,f,g)$ (recall $\widehat{e}_{\bold\psi} = \nabla\times(\boldsymbol{\psi}_{\widehat{\eta}}-\widehat{\boldsymbol{\psi}})$) for $q\geq2$, which, along with \eqref{prop:bound_pf03}, implies  
\begin{equation}\label{prop:bound_pf06}
    \|\nabla\times\boldsymbol{\psi}_{\widehat{\eta}}\|_{\bold{L}^q(\Omega)} \leq c(q,f,g).
\end{equation}
{\it Step 4.} Boundedness of $\|\bold{\psi}_{\widehat{\eta}}\|_{\bold{L}^q(\Omega)}$ for $1<q<2$. It follows from the estimate \eqref{eqn:psietaq<2} with $\bold{\psi}_{\eta}=\bold{0}$ in $\|\nabla\times\boldsymbol{\psi}_{\eta}-\boldsymbol{\tau}^*\|_{\bold{L}^q(\Omega)}^{\frac{q}{2}}$, \eqref{est:stab}, \eqref{prop:bound_pf05} and Assumption \ref{assumption1} that
$S(\nabla \phi_{\widehat{\theta}} + \nabla\times\boldsymbol{\psi}_{\widehat{\eta}},\nabla \phi_{\widehat{\theta}} + \nabla\times\widehat{\boldsymbol{\psi}})\|\widehat{e}_{\bold\psi}\|_{\bold{L}^q(\Omega)}^2\leq c(q,f,g).$
Next, by the elementary inequality $(a+b)^t\leq a^t+b^t$ when $a,b\geq0$ and $0<t<1$, we get
 \begin{align*}
& ( \|\nabla \phi_{\widehat{\theta}} + \nabla\times\widehat{\boldsymbol{\psi}}\|_{\bold{L}^q(\Omega)}^q + \|\nabla \phi_{\widehat{\theta}} + \nabla\times\widehat{\boldsymbol{\psi}} + t\widehat{e}_{\bold\psi}\|_{\bold{L}^q(\Omega)}^q )^{1-\frac{2}{q}} \\
    \geq & ( \|\nabla \phi_{\widehat{\theta}} + \nabla\times\widehat{\boldsymbol{\psi}}\|_{\bold{L}^q(\Omega)}^{2-q} + \|\nabla \phi_{\widehat{\theta}} + \nabla\times \widehat{\boldsymbol{\psi}} + t\widehat{e}_{\bold\psi}\|_{\bold{L}^q(\Omega)}^{2-q} )^{-1} \\
     \geq &  ( 2\|\nabla \phi_{\widehat{\theta}} + \nabla\times\widehat{\boldsymbol{\psi}}\|_{\bold{L}^q(\Omega)}^{2-q} + t^{2-q}\|\widehat{e}_{\bold\psi}\|_{\bold{L}^q(\Omega)}^{2-q} )^{-1} 
    \geq 
    (\|\widehat{e}_{\bold\psi}\|_{\bold{L}^q(\Omega)}^{2-q} + 2\|\nabla \phi_{\widehat{\theta}} + \nabla\times\widehat{\boldsymbol{\psi}}\|_{\bold{L}^q(\Omega)}^{2-q})^{-1}.
\end{align*}
This, the definition of $S(\nabla \phi_{\widehat{\theta}} + \nabla\times\boldsymbol{\psi}_{\widehat{\eta}},\nabla \phi_{\widehat{\theta}} +  \nabla\times\widehat{\boldsymbol{\psi}})$, \eqref{prop:bound_pf01} and \eqref{prop:bound_pf03} imply with $c=c(q,f,g,\Omega)$ 
\begin{align*}
S(\nabla \phi_{\widehat{\theta}} + \nabla\times\boldsymbol{\psi}_{\widehat{\eta}}, \nabla \phi_{\widehat{\theta}} + \nabla\times\widehat{\boldsymbol{\psi}})\|\widehat{e}_{\bold\psi}\|_{\bold{L}^q(\Omega)}^2    \geq (2\|\widehat{e}_{\bold\psi}\|_{\bold{L}^q(\Omega)}^{2-q} + c)^{-1}\|\widehat{e}_{\bold\psi}\|_{\bold{L}^q(\Omega)}^2:=h(\|\widehat{e}_{\bold\psi}\|_{\bold{L}^q(\Omega)}).
\end{align*}
Since $h(s)=\frac{s^2}{2s^{2-q}+c}$ with $c>0$ and $s\geq0$ is continuous and strictly increasing, there exists an increasing inverse $h^{-1}$. Therefore, from the inequalities
\[
h\bigl(\|\widehat{e}_{\bold\psi}\|_{\bold{L}^q(\Omega)}\bigr) \leq S(\nabla \phi_{\widehat{\theta}} + \nabla\times\boldsymbol{\psi}_{\widehat{\eta}}, \nabla \phi_{\widehat{\theta}} + \nabla\times\widehat{\boldsymbol{\psi}})\|\widehat{e}_{\bold\psi}\|_{\bold{L}^q(\Omega)}^2 \leq c,
\]
we know
$\|\widehat{e}_{\bold\psi}\|_{\bold{L}^q(\Omega)} \leq h^{-1}(c)$,
which, together with \eqref{prop:bound_pf03} again, yields the desired assertion.
\end{proof}

Now we bound the error of $\boldsymbol{\sigma}_{\widehat{\theta},\widehat{\eta}}$ in terms of the approximation error $\mathcal{E}_{\rm app}$ due to DNNs and the statistical error $\mathcal{E}_{\rm stat}$ arising from Monte Carlo quadrature.

\begin{theorem}\label{thm:errordecom}
Let $\boldsymbol{\sigma}^*$ be the solution of problem \eqref{min:origin} and $\boldsymbol{\sigma}_{\widehat{\theta},\widehat{\eta}}$ be the approximation given by Algorithm~\ref{alg:dualform}. Under Assumption \ref{assumption1}, there exists $c=c(q,f,g,\Omega,\lambda)$ such that
\begin{equation}
\|\boldsymbol{\sigma}^*-\boldsymbol{\sigma}_{\widehat{\theta},\widehat{\eta}}\|_{\boldsymbol{L}^q(\Omega)}\leq c(\mathcal{E}_{\rm app}+\mathcal{E}_{\rm stat}),
\end{equation}
where the approximation error $\mathcal{E}_{\rm app}$ and the statistical error $\mathcal{E}_{\rm stat}$ are, respectively, given by
\begin{align*}
\mathcal{E}_{\rm app}&=\left\{
\begin{aligned}
    &\inf_{\phi_{\theta} \in \mathcal{N}_{\phi}}\|\phi_{\theta}-\phi^*\|_{W^{2,q}(\Omega)}^{\frac{q}{4}}+\inf_{\boldsymbol{\psi}_\eta\in\mathcal{N}_{\boldsymbol{\psi}}}\|\nabla\times\boldsymbol{\psi}_{\eta}-\boldsymbol{\tau}^*\|_{\bold{L}^q(\Omega)}^{\frac{q}{2}},\quad 1<q<2,\\
    &\inf_{\phi_{\theta} \in \mathcal{N}_{\phi}}\|\phi_{\theta}-\phi^*\|_{W^{2,q}(\Omega)}^{\frac{2}{q^2}}+\inf_{\boldsymbol{\psi}_\eta\in\mathcal{N}_{\boldsymbol{\psi}}}\|\nabla\times\boldsymbol{\psi}_{\eta}-\boldsymbol{\tau}^*\|_{\bold{L}^q(\Omega)}^{\frac{2}{q}},\quad q \geq 2,
\end{aligned}\right.\\
\mathcal{E}_{\rm stat}&=\left\{
\begin{aligned}
&\sup _{\phi_{\theta} \in \mathcal{N}_{\phi}}\left|L_{\rm p}(\phi_{\theta})-\widehat{L}_{\rm p}(\phi_{\theta})\right|^{\frac{q}{4}}+\sup_{\boldsymbol{\psi}_{\eta} \in \mathcal{N}_{\boldsymbol{\psi}}}\big|L_{\rm s}(\boldsymbol{\psi}_{\eta})-\widehat{L}_{\rm s}(\boldsymbol{\psi}_{\eta})\big|^{\frac{1}{2}},\quad 1<q<2,\\
&\sup _{\phi_{\theta} \in \mathcal{N}_{\phi}}\left|L_{\rm p}(\phi_{\theta})-\widehat{L}_{\rm p}(\phi_{\theta})\right|^{\frac{2}{q^2}}+\sup_{\boldsymbol{\psi}_{\eta} \in \mathcal{N}_{\boldsymbol{\psi}}}\big|L_{\rm s}(\boldsymbol{\psi}_{\eta})-\widehat{L}_{\rm s}(\boldsymbol{\psi}_{\eta})\big|^{\frac{1}{q}},\quad q\geq 2.
\end{aligned}\right.
\end{align*}
\end{theorem}
\begin{proof}
In view of Theorem \ref{thm:equivalence}, we have $\boldsymbol{\sigma}^*=\nabla\phi^*+\boldsymbol{\tau}^*$ with $\phi^\ast$ and $\bold{\tau}^\ast = \nabla \times \bold{\psi}^\ast$ solving \eqref{eqn:Poisson} and \eqref{min:unconstrained}, respectively. Using the splittings \eqref{equ:error_decom} and \eqref{equ:error_decom_s}, we have $$\boldsymbol{\sigma}^*-\boldsymbol{\sigma}_{\widehat{\theta},\widehat{\eta}} = \nabla ( \phi^\ast -  \phi_{\widehat{\theta}} ) + \nabla \times (\boldsymbol{\psi}^\ast - \widehat{\bold{\psi}} ) -\widehat{e}_{\bold\psi},\quad \mbox{with }\widehat{e}_{\bold\psi} = \nabla\times(\boldsymbol{\psi}_{\widehat{\eta}}-\widehat{\boldsymbol{\psi}}).$$  
Consequently,
\begin{equation}\label{thm:errordecomp_pf01}
\|\boldsymbol{\sigma}^*-\boldsymbol{\sigma}_{\widehat{\theta},\widehat{\eta}}\|_{\bold{L}^{q}(\Omega)}
\leq\|\phi_{\widehat{\theta}}-\phi^*\|_{W^{1,q}(\Omega)}+\|\widehat{e}_{\bold\psi}\|_{\bold{L}^{q}(\Omega)}+\|\nabla\times(\widehat{\bold{\psi}}-\boldsymbol{\psi}^*)\|_{\bold{L}^{q}(\Omega)}.
\end{equation}
By Proposition \ref{prop:bound}, $\|\nabla \phi_{\widehat{\theta}}\|_{\bold{L}^q(\Omega)}$, $\|\nabla\times\boldsymbol{\psi}_{\widehat{\eta}}\|_{\bold{L}^q(\Omega)}$ and $\|\nabla\times\widehat{\boldsymbol{\psi}}\|_{\bold{L}^q(\Omega)}$ are all bounded. Then it follows from \eqref{eqn:psietaq<2} that for $1<q<2$, there holds with $c=c(q,f,g,\Omega)$,
\begin{align*}
   &c\Big(\|\boldsymbol{\tau}^*-\nabla\times\widehat{\boldsymbol{\psi}}\|_{\bold{L}^q(\Omega)}^{\frac{q}{2}}+\inf_{\boldsymbol{\psi}_\eta\in\mathcal{N}_{\boldsymbol{\psi}}}\|\nabla\times\boldsymbol{\psi}_{\eta}-\boldsymbol{\tau}^*\|_{\bold{L}^q(\Omega)}^{\frac{q}{2}}+\sup_{\psi_{\eta} \in \mathcal{N}_{\boldsymbol{\psi}}}\big|L_{\rm s}(\boldsymbol{\psi}_{\eta})-\widehat{L}_{\rm s}(\boldsymbol{\psi}_{\eta})\big|^{\frac{1}{2}}\Big)\\
   \geq &cS( \nabla \phi_{\widehat{\theta}} + \nabla\times\boldsymbol{\psi}_{\widehat{\eta}}, \nabla \phi_{\widehat{\theta}} + \nabla\times\widehat{\boldsymbol{\psi}})^{\frac{1}{2}}\|\widehat{e}_{\bold\psi}\|_{\bold{L}^q(\Omega)} \\
= & c\left(\int_{0}^1 ( \|\nabla \phi_{\widehat{\theta}} + \nabla \times \widehat{\bold\psi}\|_{\bold{L}^q(\Omega)}^q + \|\nabla \phi_{\widehat{\theta}} + \nabla\times\widehat{\boldsymbol{\psi}} + t\widehat{e}_{\bold\psi}\|_{\bold{L}^q(\Omega)}^q )^{1-\frac{2}{q}} t {\rm d}t\right)^{\frac{1}{2}}\|\widehat{e}_{\bold\psi}\|_{\bold{L}^q(\Omega)}
\geq  c\|\widehat{e}_{\bold\psi}\|_{\bold{L}^q(\Omega)}.
\end{align*}
By combining this with \eqref{eqn:psietaq>2} for the case $q\geq2$, we obtain with $c=c(q,f,g,\Omega)$
\begin{align*}
&\|\widehat{e}_{\bold\psi}\|_{\bold{L}^q(\Omega)}\\
   \leq &c\left\{\begin{aligned}
    &\|\boldsymbol{\tau}^*-\nabla\times\widehat{\boldsymbol{\psi}}\|_{\bold{L}^q(\Omega)}^{\frac{q}{2}}+\inf_{\boldsymbol{\psi}_\eta\in\mathcal{N}_{\boldsymbol{\psi}}}\|\nabla\times\boldsymbol{\psi}_{\eta}-\boldsymbol{\tau}^*\|_{\bold{L}^q(\Omega)}^{\frac{q}{2}}+\sup_{\psi_{\eta} \in \mathcal{N}_{\boldsymbol{\psi}}}\big|L_{\rm s}(\boldsymbol{\psi}_{\eta})-\widehat{L}_{\rm s}(\boldsymbol{\psi}_{\eta})\big|^{\frac{1}{2}}, &&1<q<2,\\
    &\|\boldsymbol{\tau}^*-\nabla\times\widehat{\boldsymbol{\psi}}\|_{\bold{L}^q(\Omega)}^{\frac{2}{q}}+\inf_{\boldsymbol{\psi}_\eta\in\mathcal{N}_{\boldsymbol{\psi}}}\|\nabla\times\boldsymbol{\psi}_{\eta}-\boldsymbol{\tau}^*\|_{\bold{L}^q(\Omega)}^{\frac{2}{q}}+\sup_{\psi_{\eta} \in \mathcal{N}_{\boldsymbol{\psi}}}\big|L_{\rm s}(\boldsymbol{\psi}_{\eta})-\widehat{L}_{\rm s}(\boldsymbol{\psi}_{\eta})\big|^{\frac{1}{q}}, &&q\geq2.
\end{aligned}\right.
\end{align*}
The estimate \eqref{prop:bound_pf05} implies 
\begin{equation*}
\|\nabla\times\widehat{\boldsymbol{\psi}}-\boldsymbol{\tau}^*\|_{\bold{L}^{q}(\Omega)} = \|\nabla\times\widehat{\boldsymbol{\psi}}-\nabla\times\boldsymbol{\psi}^*\|_{\bold{L}^{q}(\Omega)} \leq c\|\phi_{\widehat{\theta}}-\phi^*\|_{W^{1,q}(\Omega)}^{\min\{\frac{1}{q},\frac{1}{2}\}}.
\end{equation*}
The proof is completed by collecting the preceding three estimates and Lemma \ref{lem:decomp}.
\end{proof}

By appropriately estimating the approximation error $\mathcal{E}_{\rm app}$ and the statistical error $\mathcal{E}_{\rm stat}$ (under the following a priori regularity assumption and the boundedness condition), we can derive an error bound on the DVNN approximation $\boldsymbol{\sigma}_{\widehat{\theta},\widehat{\eta}}$. The proof is lengthy and is given in Appendix \ref{append}.
\begin{assumption}\label{assump:regularity}
 {\rm(i)} $\phi^*\in W^{3,q}(\Omega)$, and $\boldsymbol{\psi}^*\in\bold{W}^{2,q}(\Omega)$; {\rm(ii)}
$f\in L^\infty(\Omega)$, and $g\in L^\infty(\partial\Omega)$.
\end{assumption}
\begin{theorem}\label{thm:finalthm}
Under Assumptions \ref{assumption1}--\ref{assump:regularity}, for any $\epsilon>0$, small $\zeta>0$ and $\mu>0$, if the numbers $N_d$ and $N_b$ of the sampling points are chosen such that
{\small
\begin{equation}\label{eqn:ndnb}
    N_d=\left\{\begin{aligned}
    &\mathcal{O}\left(\epsilon^{-\frac{72q^2L+36qL-12q^2+16}{q(1-\mu)}}\right),1<q<2,\\
    &\mathcal{O}\left(\epsilon^{-\frac{(18qL+6q)\max\{q,4\}+4q}{1-\mu}}\right),q\geq2,
\end{aligned}\right.\,\, \mbox{and}\,\, 
N_b=\left\{\begin{aligned}
    &\mathcal{O}\left(\epsilon^{-\frac{(36L+18L/q-6)\max\{q^2,2\}+16}{q(1-\mu)}}\right),1<q<2,\\
    &\mathcal{O}\left(\epsilon^{-\frac{36qL+18L-2q}{1-\mu}}\right),q\geq2,
\end{aligned}\right.
\end{equation}}
then with probability at least $1- 4\zeta$ and $c=c(q,f,g,L,\Omega,\lambda)$, there holds
\begin{equation}
\|\boldsymbol{\sigma}^*-\boldsymbol{\sigma}_{\widehat{\theta},\widehat{\eta}}\|_{\boldsymbol{L}^q(\Omega)}\leq c\left\{
\begin{aligned}
    &\epsilon^{\frac{q}{4}}+\epsilon^{\frac{1}{4(1-\mu)}}\log^{\frac{1}{4}}\tfrac{1}{\epsilon}+\epsilon^{\frac{2}{1-\mu}}\log^{\frac{1}{4}}\tfrac{1}{\zeta},\quad1<q<2,\\
    &\epsilon^{\frac{2}{q^2}}+\epsilon^{\frac{1}{2q(1-\mu)}}\log^{\frac{1}{2q}}\tfrac{1}{\epsilon}+\epsilon^{\frac{4}{q(1-\mu)}}\log^{\frac{1}{2q}}\tfrac{1}{\zeta},\quad q\geq2.
\end{aligned}\right.
\end{equation}
\end{theorem}

\section{Numerical experiments and discussions }\label{sec:experiment}

In this section, we present numerical results to illustrate the accuracy and robustness of the DVNN. $N_r=10,000$ points in $\Omega$ and $N_b=2,000$ points on $\partial\Omega$ are selected uniformly at random in the empirical losses $\widehat{L}_{\rm p}(\phi_\theta)$ and $\widehat{L}_{\rm s}(\boldsymbol{\psi}_\eta)$. The loss $\widehat{L}_{\rm p}(\phi_\theta)$ is minimized using a two-stage optimization strategy~\cite{wang2025highprecisionpinnsunbounded}: ADAM~\cite{KingmaBa:2015} (with default parameters), followed by self-scaled BFGS (SSBFGS) \cite{OrenLuenberger:1974};  the loss $\widehat{L}_{\rm s}(\psi_\eta)$ is minimized using  ADAM only. We measure the accuracy of an NN approximation $\boldsymbol{\sigma}_{\widehat{\theta},\widehat{\eta}}$ of the exact flux $\boldsymbol{\sigma}^*$ by the relative $\bold{L}^q(\Omega)$-error  $e_{\bold{\sigma}}=\|\boldsymbol{\sigma}^*-\boldsymbol{\sigma}_{\widehat{\theta},\widehat{\eta}}\|_{\bold{L}^q(\Omega)}/\|\boldsymbol{\sigma}^*\|_{\bold{L}^q(\Omega)}$, and the relative $L^p$ error  $e_u=\|\hat{u}-u^*\|_{L^p(\Omega)}/\|u^*\|_{L^p(\Omega)}$ for the approximate potential $\hat{u}$.  The Python source codes for reproducing all numerical experiments are available in the github repository \url{https://github.com/hhjc-web/plaplace}. First we give one example adapted from~\cite[Example 5.1]{BarrettLiu1993}.

\begin{example}\label{exam1}
 $\Omega=\{x\in\mathbb{R}^3:|x|<1\}$.
The solution is $u(x)=1-|x|^{\frac{p}{p-1}}$ with the source $f(x)=3(\frac{p}{p-1})^{p-1}$ and $g=0$. Consider the following two cases: {\rm(i)} $p=1.1$ and {\rm(ii)} $p=500$.
\end{example}

In case (i), $p$ is close to 1, and 
the $p$-Laplace problem exhibits strong nonlinearity.
Fig.~\ref{fig:exam1.1} shows the neural network approximation $\sigma_{\widehat{\theta},\widehat{\eta},1}$ of the flux component $\sigma_1^\ast$, with the architecture given in Table \ref{table:exam1.1}.
The DVNN clearly yields an accurate approximation  with a relative error $e_{\bold{\sigma}}=2.40\times10^{-3}$.

\begin{figure}[hbt!]
\centering\setlength{\tabcolsep}{2pt}
\begin{tabular}{cccc}
\includegraphics[width=.32\textwidth]{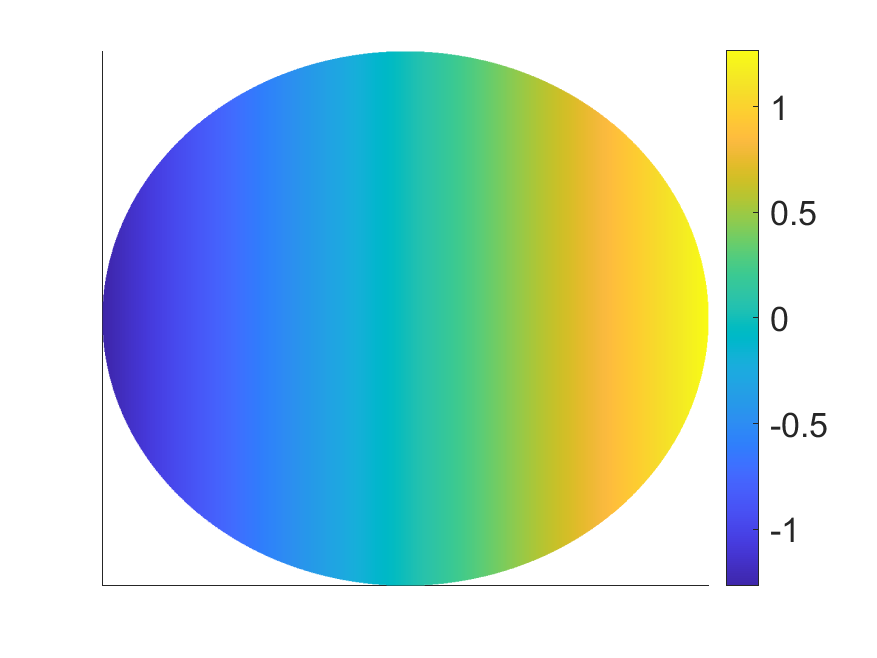} & \includegraphics[width=.32\textwidth]{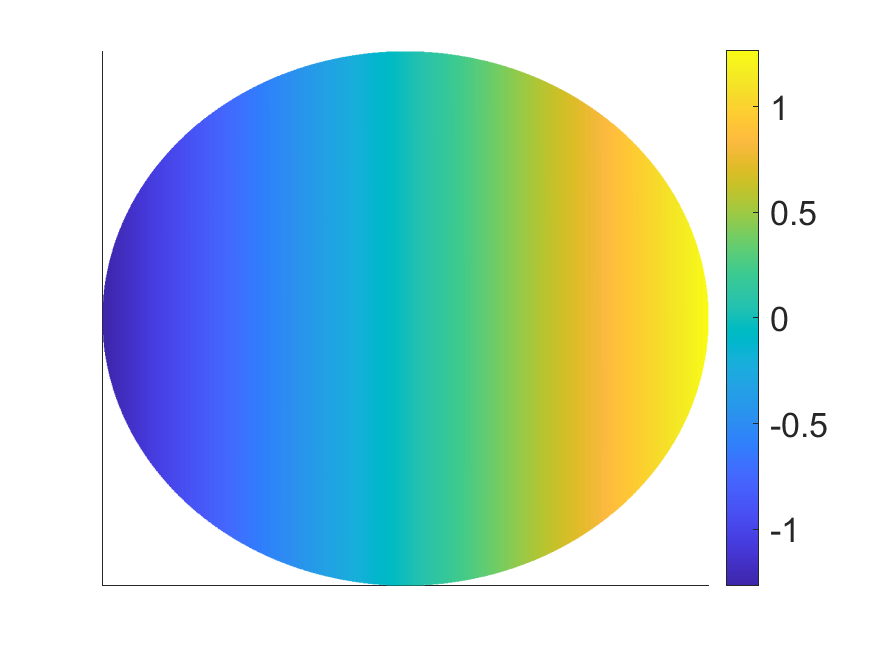} & \includegraphics[width=.32\textwidth]{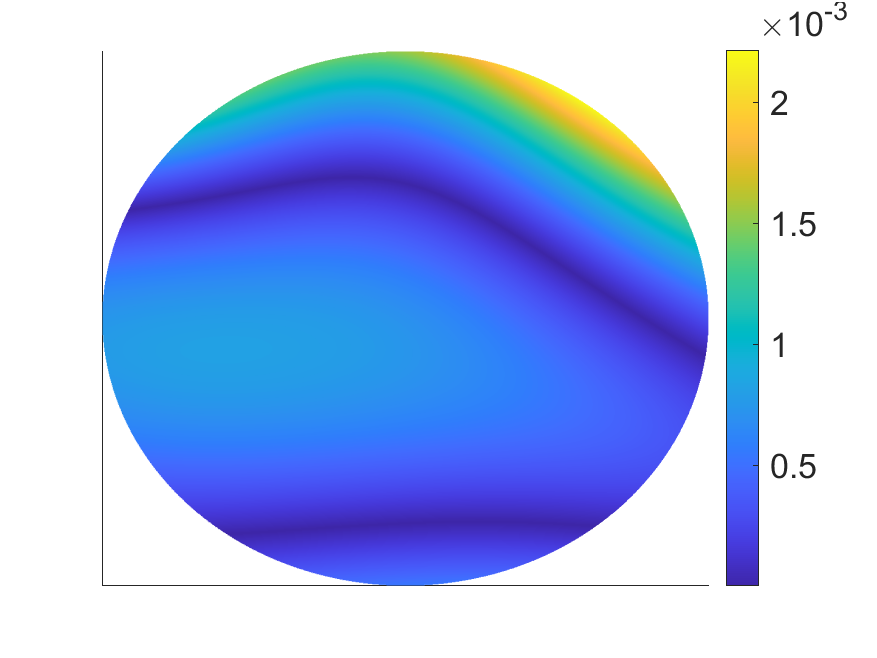}\\
(a) $\sigma_1^*$ & (b) $\sigma_{\widehat{\theta},\widehat{\eta},1}$  & (c) $|\sigma_1^*-\sigma_{\widehat{\theta},\widehat{\eta},1}|$
\end{tabular}
\caption{\label{fig:exam1.1} The DNN approximation of the flux component $\sigma_1^*$ for case (i) in Example~\ref{exam1} (slices at $x_3=0$).}
\end{figure}

With the approximation $\boldsymbol{\sigma}_{\widehat{\theta},\widehat{\eta}}$, one may  approximate the gradient by $\nabla \widehat u = -|\boldsymbol{\sigma}_{\widehat{\theta},\widehat{\eta}}|^{q-2}\boldsymbol{\sigma}_{\widehat{\theta},\widehat{\eta}}$ (cf. \eqref{first-order_sys}) and then the exact solution $u^\ast$ of problem \eqref{problem} at $x_3=0$, using the formula 
$\widehat{u}(x_1,x_2,0)=\widehat{u}(x_1,-\sqrt{1-x_1^2},0)+\int_{-\sqrt{1-x_1^2}}^{x_2}\partial_{x_2}\widehat{u}(x_1,t,0){\rm d}t$
and the trapezoidal rule for evaluating the integral.
Fig.~\ref{fig:exam1.1com} displays the approximations of $u^\ast$ at the slice $x_3=0$ by the DVNN, PINN, DRM and PINN with the mixed formulation (PINN-M), whose loss functions are given, respectively, by
\begin{align*}
L_{{\rm PINN}}(u)&=\|-\nabla \cdot \bigl( |\nabla u|^{p-2} \nabla u\bigl) - f\|_{L^q(\Omega)}+\lambda\|u-g\|_{L^p(\partial\Omega)},\\
L_{{\rm DRM}}(u)&=\tfrac{1}{p}\|\nabla u\|_{\bold{L}^p(\Omega)}^p-( f,u) +\lambda\|u-g\|_{L^p(\partial\Omega)},\\
L_{{\rm PINN-M}}(u,\bold{\sigma})&=\|\bold{\sigma}+\nabla \cdot \bigl( |\nabla u|^{p-2} \nabla u\bigl)\|_{L^q(\Omega)}+\|\nabla\cdot\bold{\sigma}-f\|_{L^q(\Omega)}+\lambda\|u-g\|_{L^p(\partial\Omega)}.
\end{align*}
The DVNN achieves better accuracy for the solution $u^\ast$ than PINN, DRM and PINN-M.
The subpar performance of PINN, DRM and PINN-M is attributed to their highly complex loss landscapes induced by the nonlinear term $|\nabla u|^{p-2}$.
It leads to pathological convergence or premature stagnation of gradient-based optimization algorithms during the training.
The parameters and training results are summarized in Table~\ref{table:exam1.1}. Note that the NN architecture for the DRM differs from the others, since the choice produces better resolution, though the resolution is still not good.

\begin{figure}[hbt!]
\centering\setlength{\tabcolsep}{2pt}
\begin{tabular}{cccc}
\includegraphics[width=.32\textwidth]{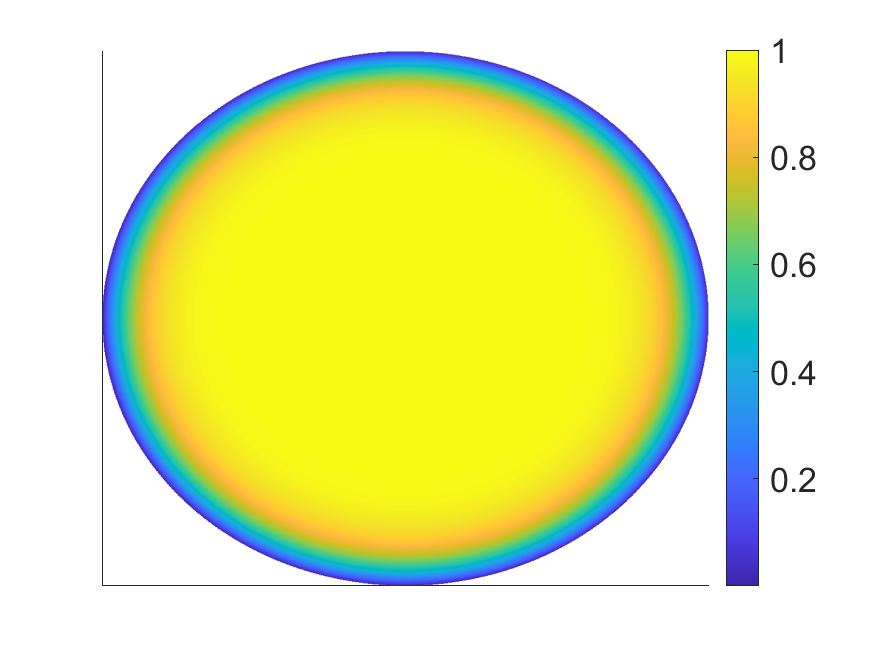} & \includegraphics[width=.32\textwidth]{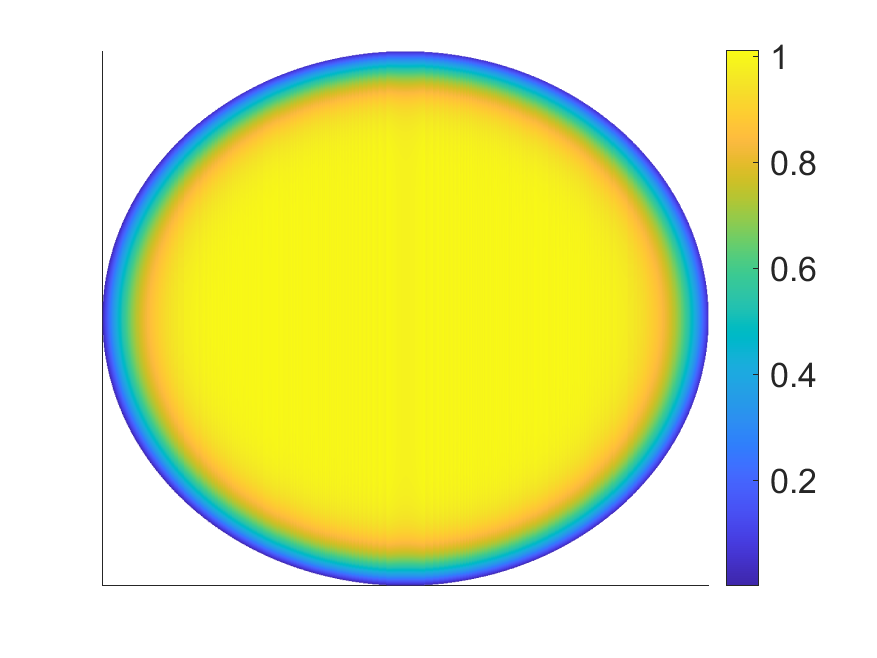} & \includegraphics[width=.32\textwidth]{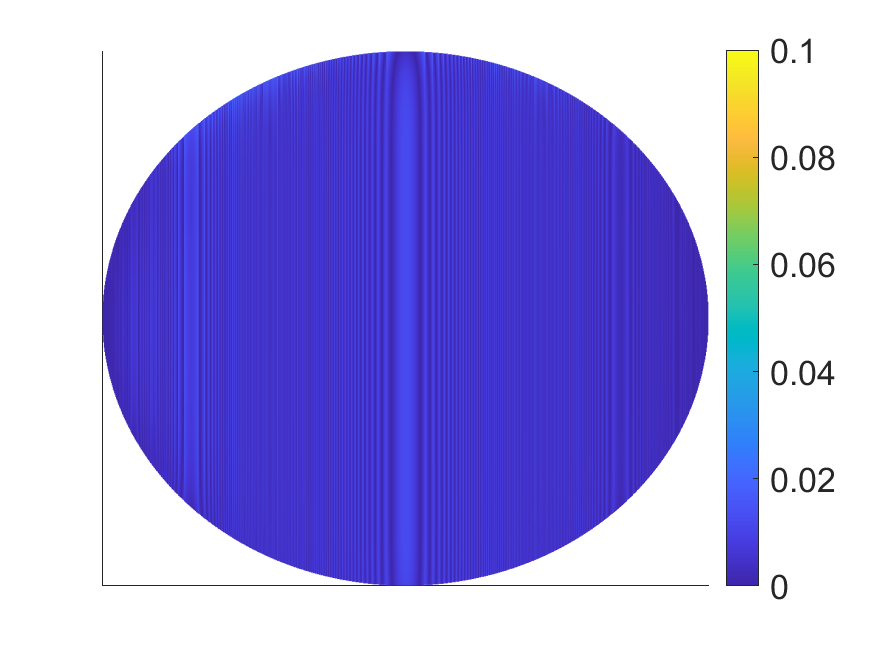}\\
\includegraphics[width=.32\textwidth]{true_solution-11} & \includegraphics[width=.32\textwidth]{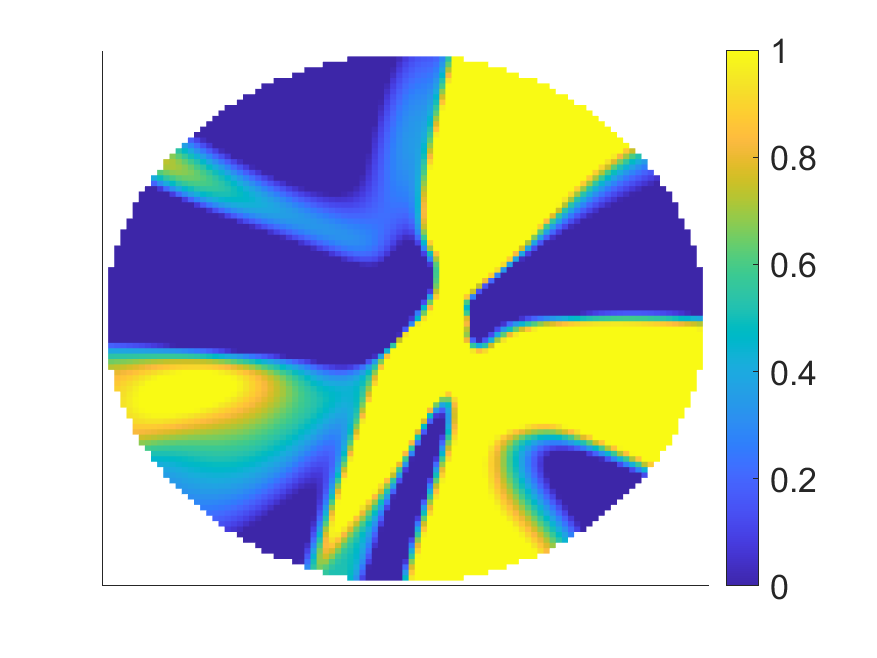} & \includegraphics[width=.32\textwidth]{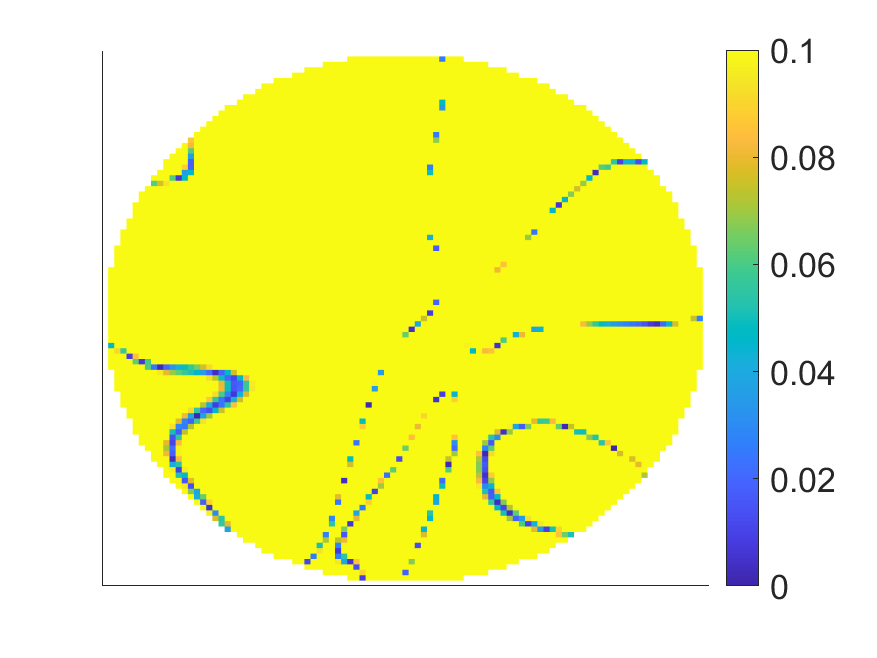}\\
 \includegraphics[width=.32\textwidth]{true_solution-11} & \includegraphics[width=.32\textwidth]{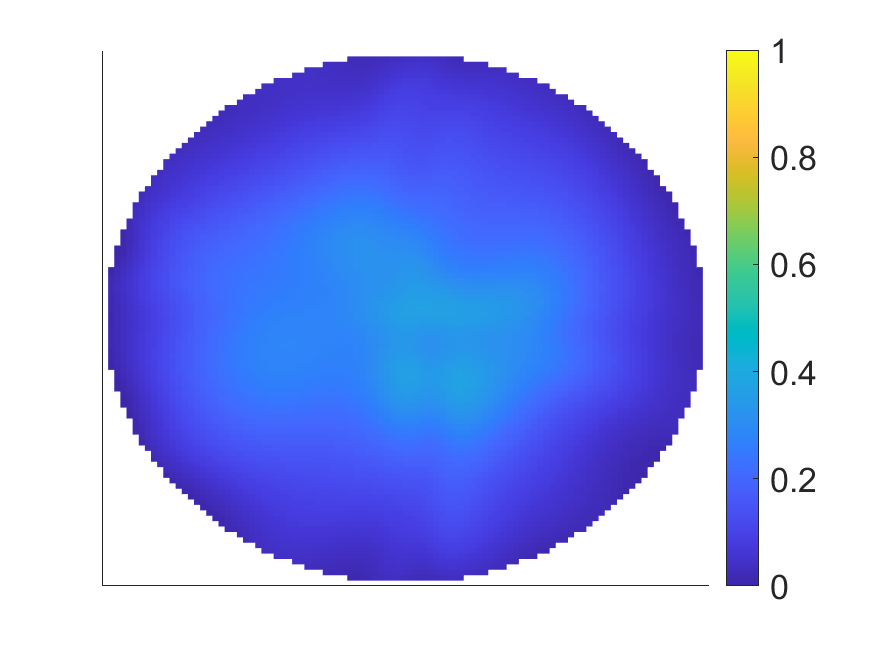} & \includegraphics[width=.32\textwidth]{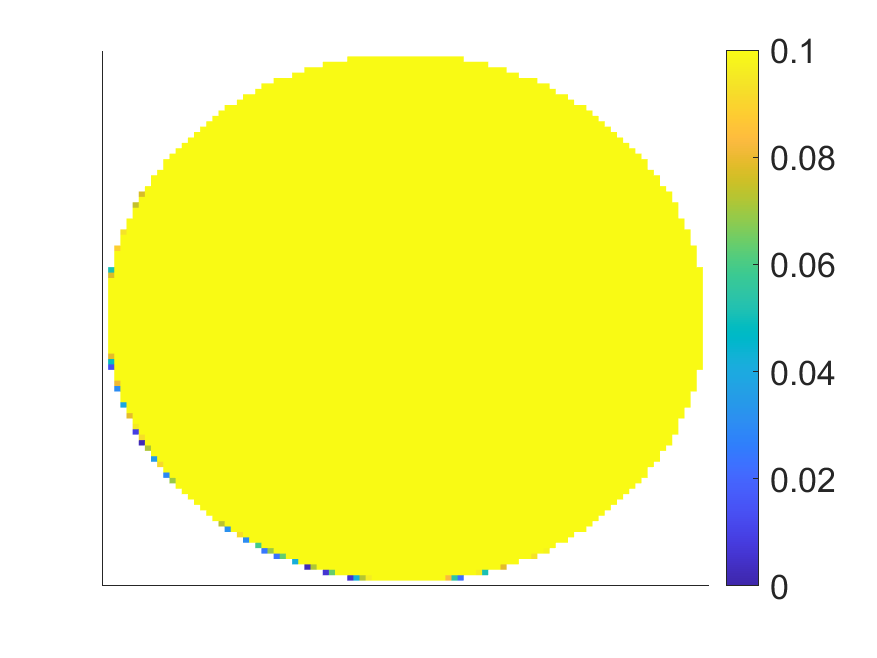}\\
\includegraphics[width=.32\textwidth]{true_solution-11} & \includegraphics[width=.32\textwidth]{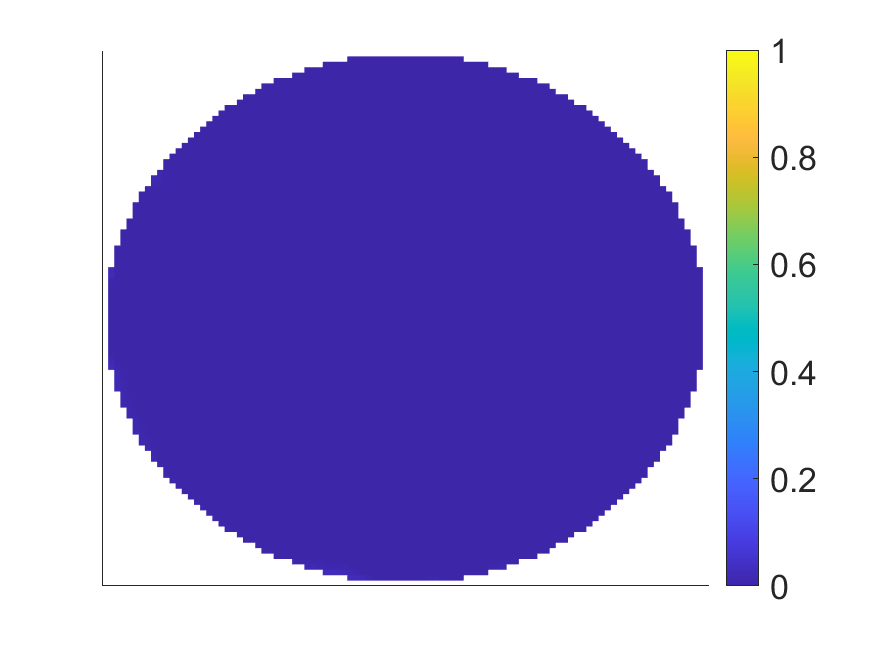} & \includegraphics[width=.32\textwidth]{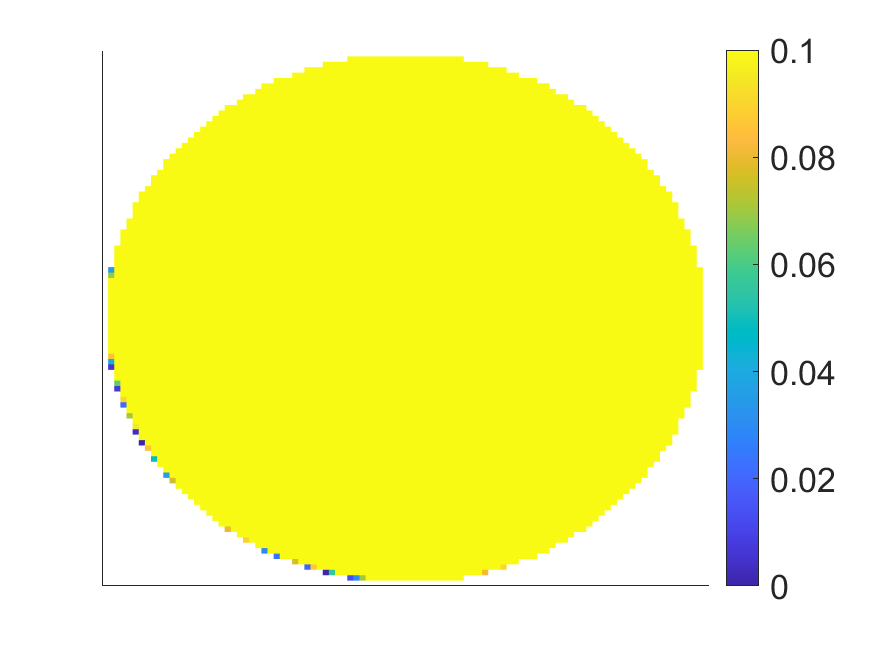}\\
(a) $u^*$ & (b) $\widehat{u}$  & (c) $|u^* - \widehat{u}|$
\end{tabular}
\caption{\label{fig:exam1.1com} The numerical solutions for case (i) in Example~\ref{exam1} (slices at $x_3=0$) by the DVNN, PINN, DRM and PINN-M (from top to bottom).}
\end{figure}

The high accuracy of the DVNN ($e_{\bold\sigma}=2.40\times 10^{-3}$ and $e_u=6.22\times 10^{-3}$) contrasts sharply with the failure of the PINN and PINN-M.  The PINN treats the nonlinear term $|\nabla u|^{p-2}$ explicitly, leading to an extremely complex optimization landscape for $p=1.1$. The DRM, which directly approximates $u$, only reduces the error to $\mathcal{O}(10^{-1})$. In sharp contrast, by the two-step procedure with the gradient-curl representation, the DVNN avoids computing the singular nonlinear term $|\nabla u|^{p-2}$ in the loss. These observations highlight the importance of exactly preserving the divergence-free constraint and properly handling the strongly nonlinear term.

\begin{table}[hbt!]
\centering
\begin{threeparttable}
\caption{\label{table:exam1.1}The parameters and results for Example~\ref{exam1} (i). 
}
\centering
\begin{tabular}[5pt]{c|c|c|c|c|c}
\toprule
method &  architecture & learning rate & epoch & $e_u$ & $e_{\boldsymbol{\sigma}}$\\
\midrule
DVNN & \makecell{$\phi$: 3-20-20-20-1\\
$\boldsymbol{\psi}$: 3-20-20-20-3} & \makecell{$5 \times 10^{-3}$\\
$5 \times 10^{-3}$} & \makecell{ 21k\\
 10k} & 6.22e-3  & 2.40e-3 \\ 
\hline
PINN & 3-20-20-20-1 & $5 \times 10^{-3}$ & 10k & 2.22e0  & 1.59e0\\
\hline
DRM & 3-20-20-20-20-20-1 & $5 \times 10^{-3}$ & 10k & 8.47e-1  & 5.24e-1\\
\hline
PINN-M & \makecell{$u$: 3-20-20-20-1\\
$\boldsymbol{p}$: 3-20-20-20-3} & $5 \times 10^{-3}$ & 10k & 1.71e0  & 2.01e0\\
\bottomrule
\end{tabular}
\end{threeparttable}
\end{table}

The variations of the empirical losses $\widehat{L}_{\rm p}$ (for the irrotational component $\phi$) and $\widehat{L}_{\rm s}$ (for the solenoidal component $\bold\psi$) during the training are shown in Fig.~\ref{fig:exam1.1train}(a) and (b), respectively. The loss $\widehat{L}_{\rm p}$ is trained by ADAM for 1,000 epochs, and then SSBFGS for 20,000 epochs, both with learning rate $5\times10^{-3}$.
The training of $\widehat{L}_{\rm s}$ only uses ADAM for 10,000 steps with learning rate $5 \times 10^{-3}$. The loss $\widehat{L}_{\rm p}$ is stable under ADAM, decreases rapidly in the first 1,000 steps by SSBFGS and reaches high precision after SSBFGS refinement, while the loss $\widehat{L}_{\rm s}$ exhibits mild oscillations due to the large $q$. Algorithm \ref{alg:dualform} is clearly very effective in minimizing the loss $\widehat{L}_{\rm s}$ and the error $e_{\boldsymbol{\sigma}}$ decays monotonically, cf. Fig. \ref{fig:exam1.1train} (c).
The DVNN remains stable even in the singular limit $p\to 1^+$.

\begin{figure}[hbt!]
\centering\setlength{\tabcolsep}{2pt}
\begin{tabular}{ccc}
\includegraphics[width=.32\textwidth]{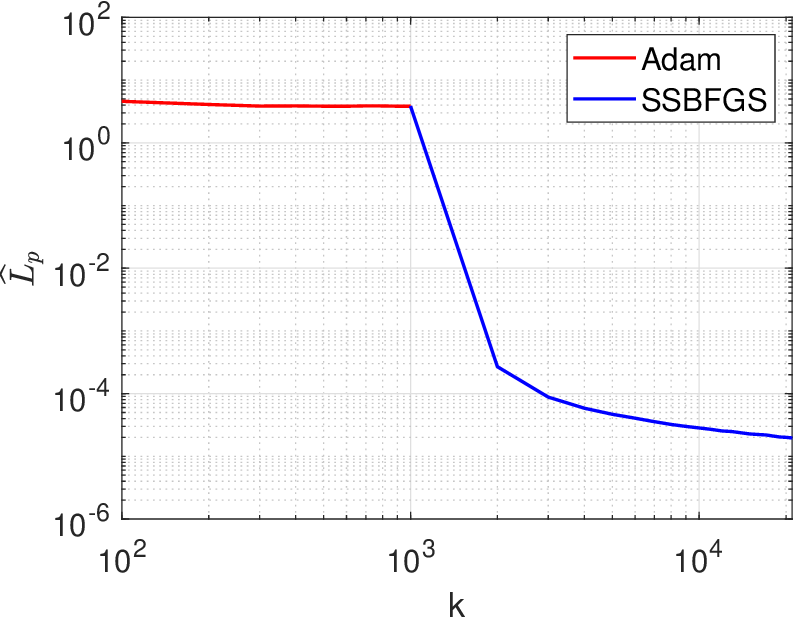} & 
\includegraphics[width=.32\textwidth]{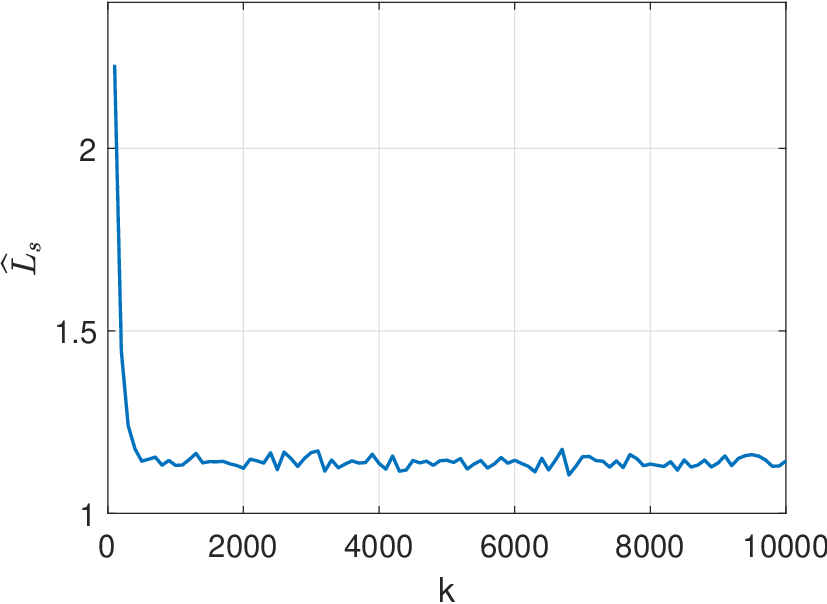} & 
\includegraphics[width=.32\textwidth]{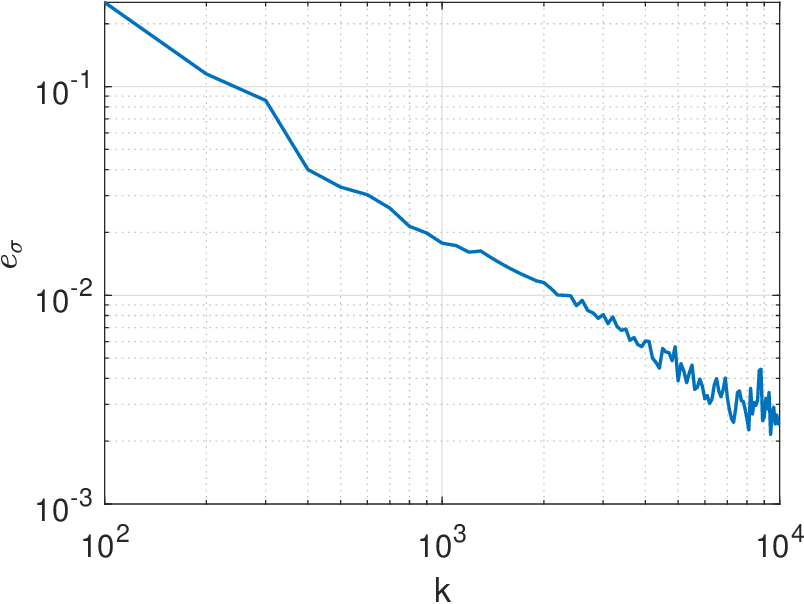}\\
(a) $\widehat {L}_{\rm p}$ & (b) $\widehat{L}_{\rm s}$  & (c) $e_{\boldsymbol{\sigma}}$ 
\end{tabular}
\caption{\label{fig:exam1.1train} The evolution of $\widehat{L}_{\rm p}$, $\widehat{L}_{\rm s}$ and $e_{\boldsymbol{\sigma}}$ versus the iteration $k$ for case (i) in Example~\ref{exam1}.}
\end{figure}

In case (ii), the problem lies in the degenerate limit $p=500$ and is also very challenging in the primal formulation. However, the conjugate exponent $q$ is moderate  ($q\approx 1$) in the dual formulation. Thus, Algorithm \ref{alg:dualform} works even for large values of $p$. When $p=500$, the PINN, DRM, and PINN-M all completely fail, since the explicit presence of $p$ in the losses leads to extremely large terms, making the training infeasible.
We use a 3-20-20-20-1 network to approximate $\phi$ and a 3-20-20-20-20-20-3 network to approximate $\boldsymbol{\psi}$.
The numerical approximation for $\boldsymbol{\sigma}$ is shown in Fig.~\ref{fig:exam1.2}, with a relative error $e_{\boldsymbol{\sigma}} = 8.11 \times 10^{-5}$. The numerical results for the two cases confirm that the DVNN is able to handle both singular and degenerate cases. 

\begin{figure}[hbt!]
\centering\setlength{\tabcolsep}{2pt}
\begin{tabular}{cccc}
\includegraphics[width=.32\textwidth]{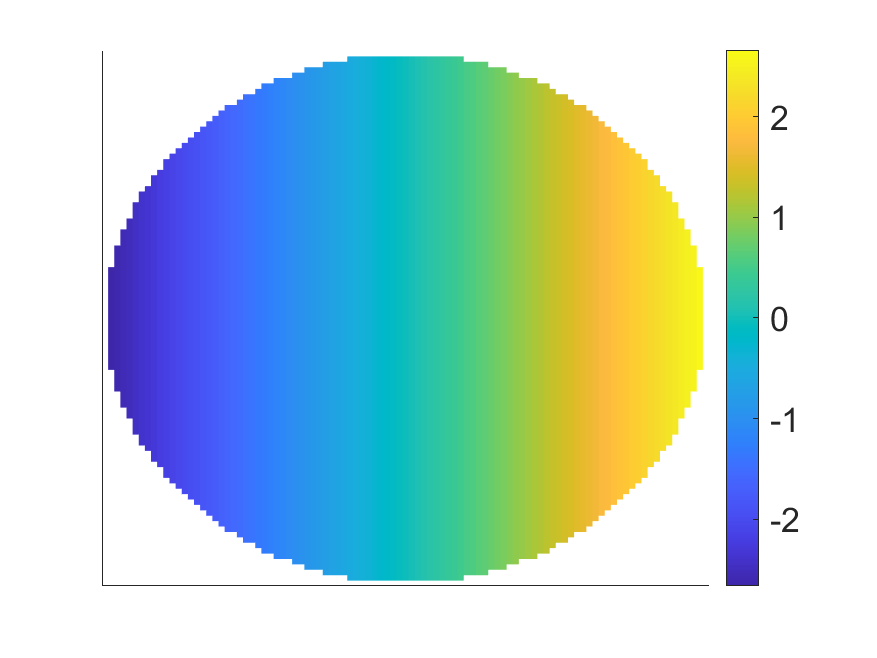} & \includegraphics[width=.32\textwidth]{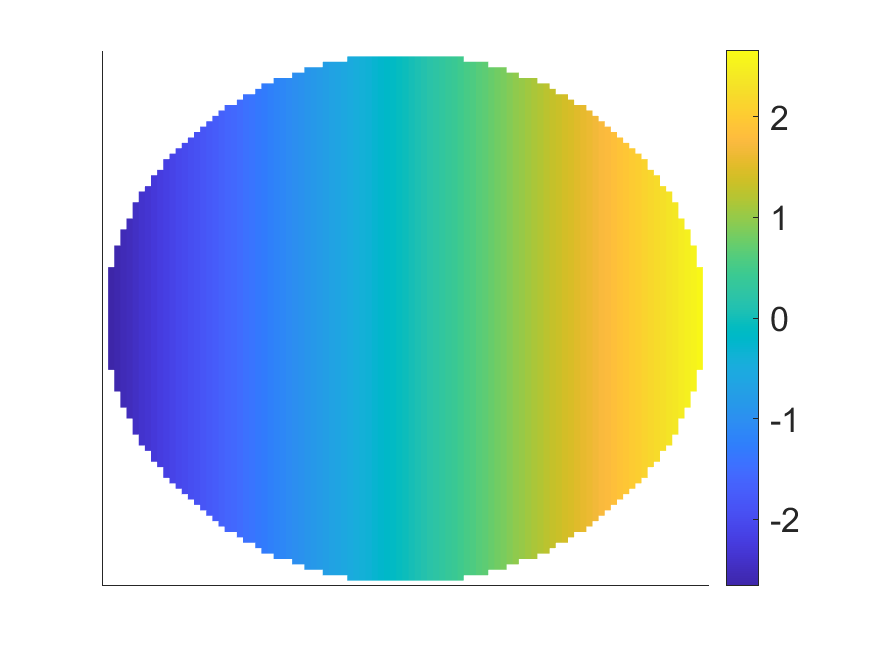} & \includegraphics[width=.32\textwidth]{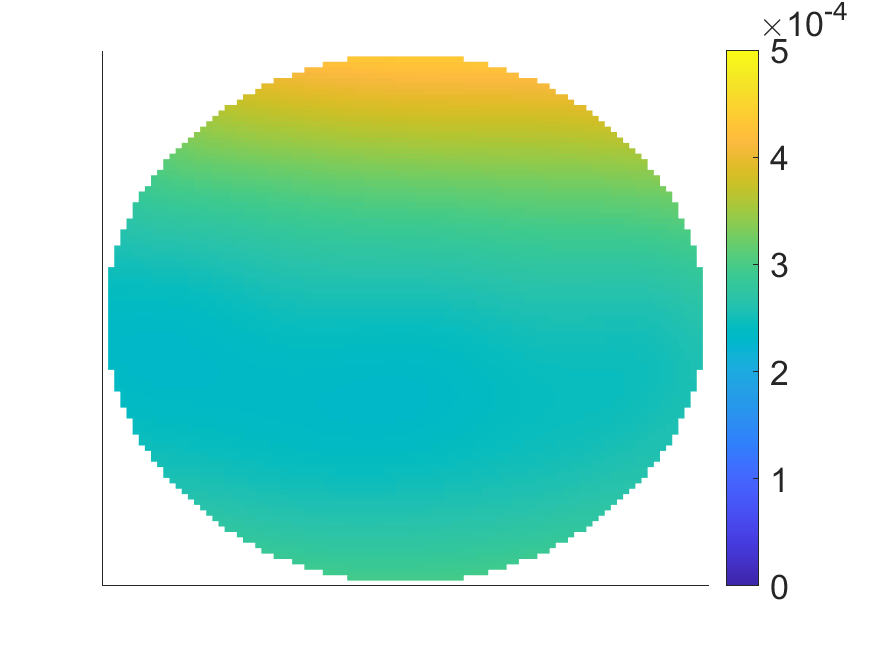}\\
\includegraphics[width=.32\textwidth]{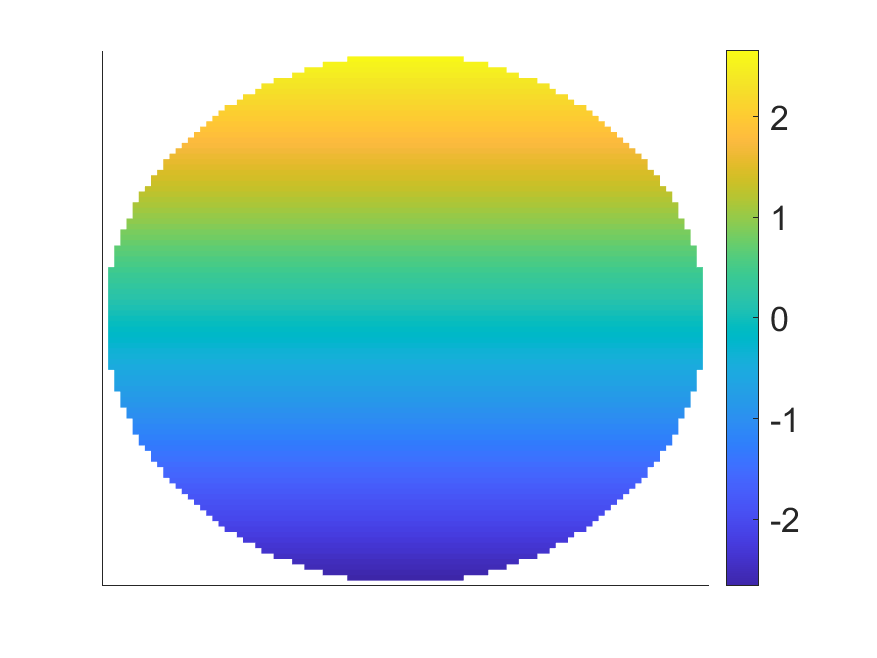} & \includegraphics[width=.32\textwidth]{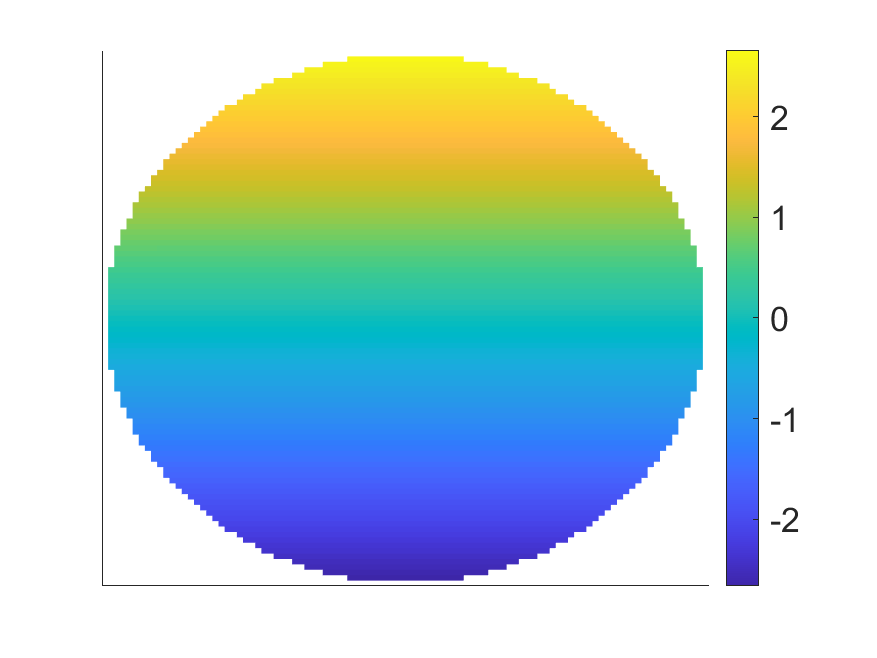} & \includegraphics[width=.32\textwidth]{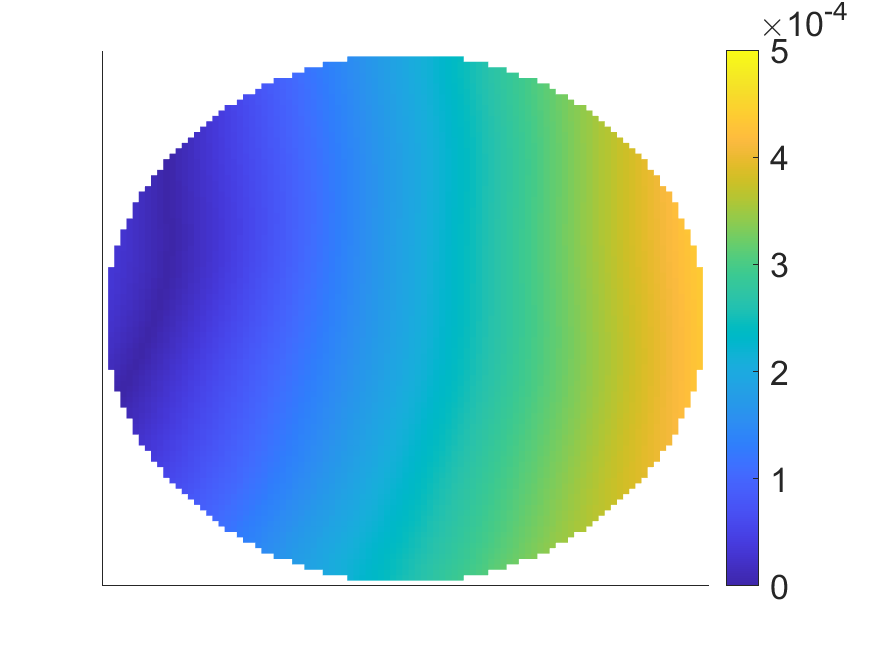}\\
\raisebox{2.5ex}{\hspace{0.8em}\includegraphics[width=.26\textwidth]{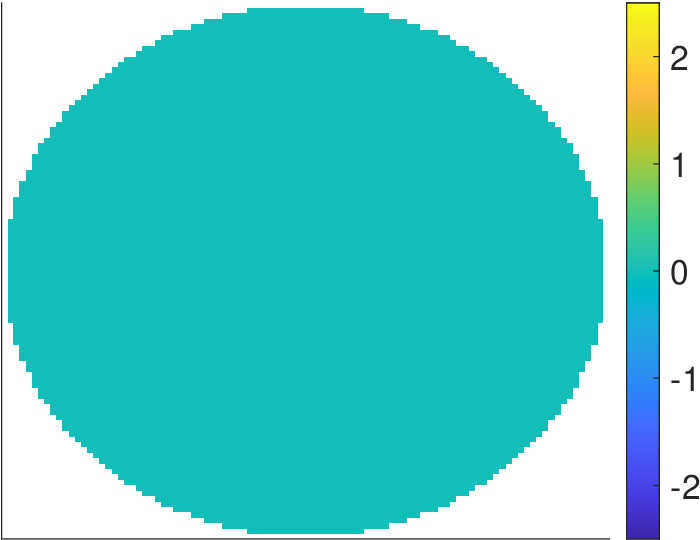}}&
\includegraphics[width=.32\textwidth]{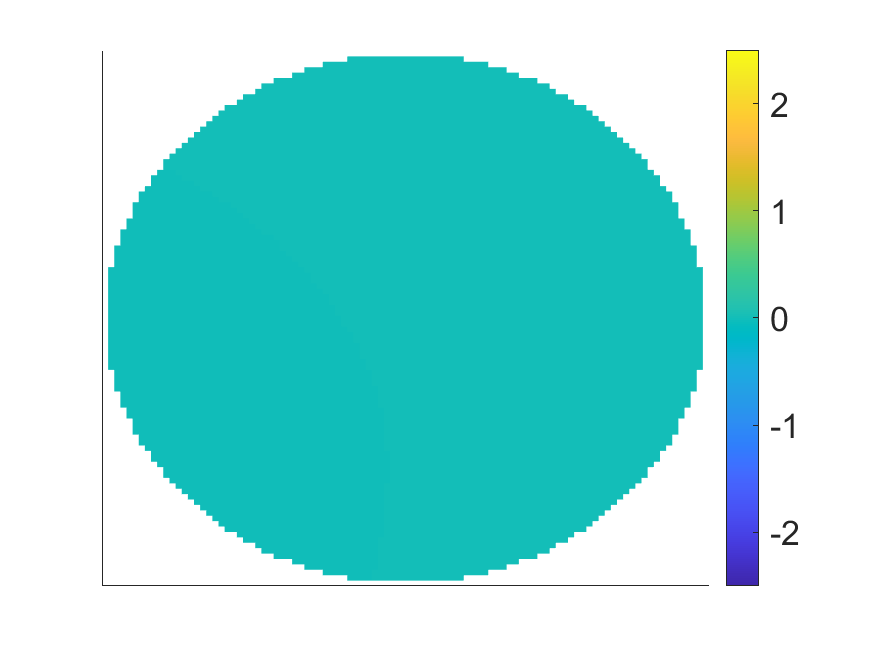} & \includegraphics[width=.32\textwidth]{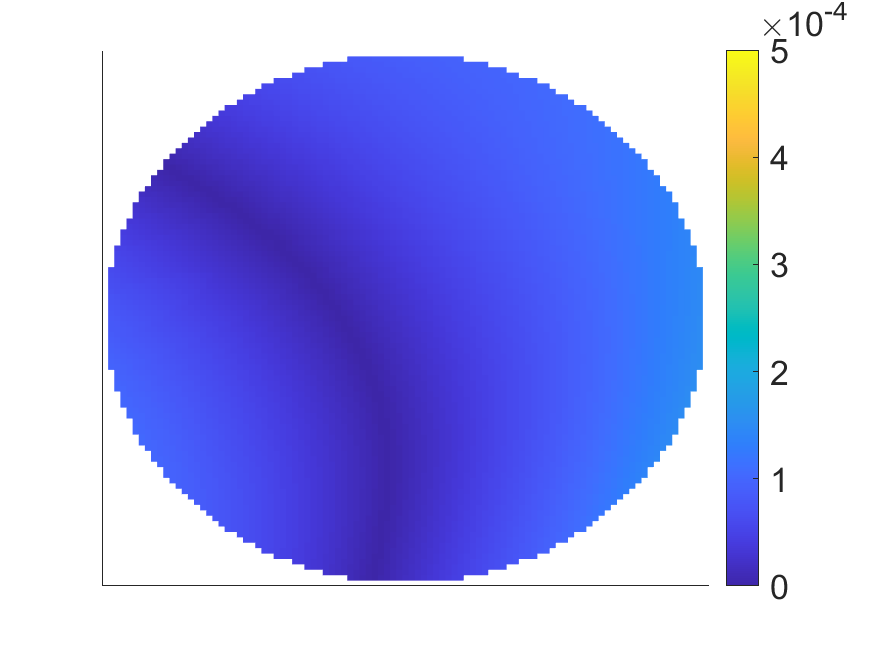}\\
 (a) $\bold{\sigma}^*$ & (b) $\bold{\sigma}_{\widehat{\theta},\widehat{\eta}}$  & (c) $|\bold{\sigma}^*-\bold{\sigma}_{\widehat{\theta},\widehat{\eta}}|$
\end{tabular}
\caption{\label{fig:exam1.2} The DNN approximations of $\boldsymbol{\sigma}^\ast=(\sigma_1^\ast,\sigma_2^\ast,\sigma_3^\ast)$ (from top to bottom) for case (ii) in Example~\ref{exam1} (slices at $x_3=0$).}
\end{figure}

The second example is about the torsion creep problem without an analytic solution.
\begin{example}\label{exam2}
$\Omega=(-1,1)^3$, $f(x)=1$ and  $g=0$.
\end{example}

The problem has no analytic solution, but
$\lim_{p \to \infty} u^\ast_p(x) = d(x, \partial\Omega)$ uniformly in  $\Omega$~\cite[Theorem 1]{Kawohl1990}.
We numerically study the asymptotic behavior of the solution $u$ as $p \to \infty$, with a 3-20-20-20-1 network to approximate $\phi$ and a 3-20-20-20-3 network to approximate $\boldsymbol{\psi}$. The numerical solution $\widehat{u}$ is obtained using the representation $
\widehat{u}(x_1,x_2,0) = \widehat{u}(x_1,-1,0)+\int_{-1}^y\partial_{x_2}\widehat{u}(x_1,t,0){\rm d}t.$
Fig.~\ref{fig:exam2} shows the results for $p=2$, $p=10$, and $p=200$.
When $p = 2$, the problem reduces to the standard Poisson equation, and since $f$ and $g$ are smooth, the solution $u$ is smooth.
At $p=10$, the solution $u$ seems to approach the limit and some non-smooth characteristics emerge across two diagonals.
When $p = 200$, although the non-smoothness on the ridge set is not so sharp, the result in Fig.~\ref{fig:exam2}(d) is very similar to that of the asymptotic solution. Since
the distance function $d(x,\partial\Omega)$ has low Sobolev regularity due to its non-differentiability along two diagonals of $\Omega$, the accuracy is affected by the approximation capability of the neural network (cf. Lemma~\ref{lem:tanh-approx}). The plots for $\nabla u$ at $p=200$ in Fig.~\ref{fig:exam2.3} clearly show the singularity along two diagonals. 

\begin{figure}[hbt!]
\centering\setlength{\tabcolsep}{1pt}
\begin{tabular}{cccc}
\includegraphics[height=2.6cm]{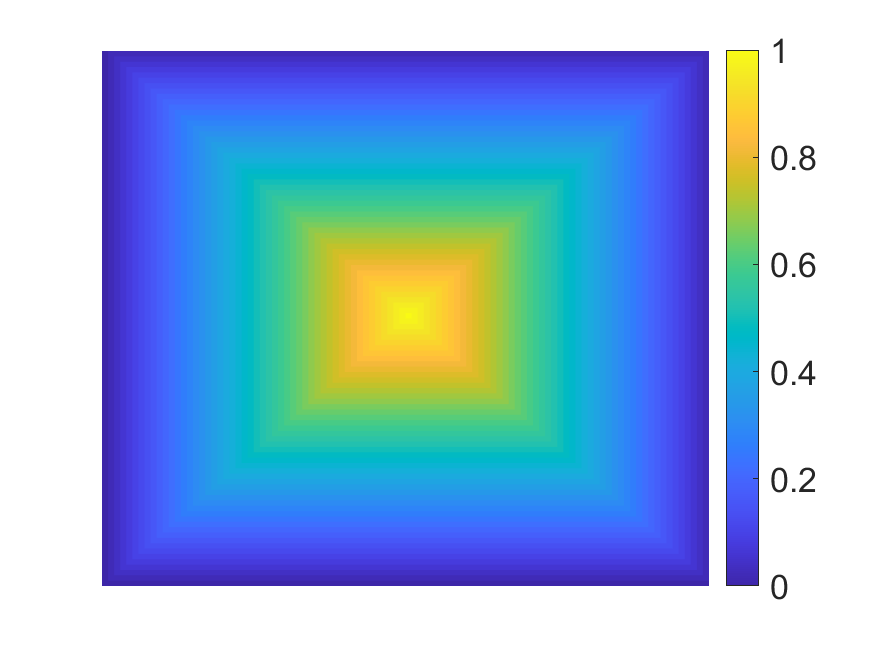} & \includegraphics[height=2.6cm]{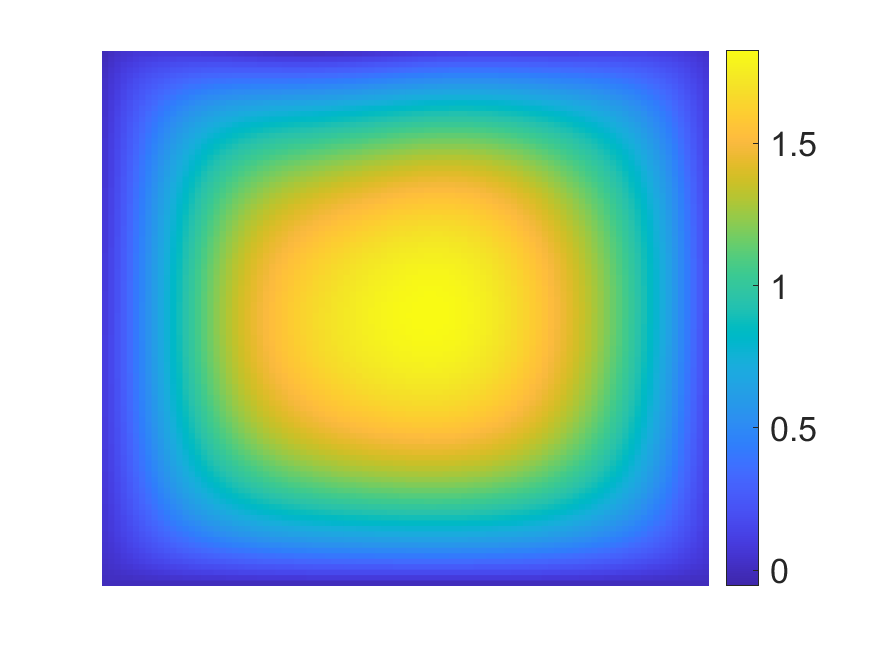} & \includegraphics[height=2.6cm]{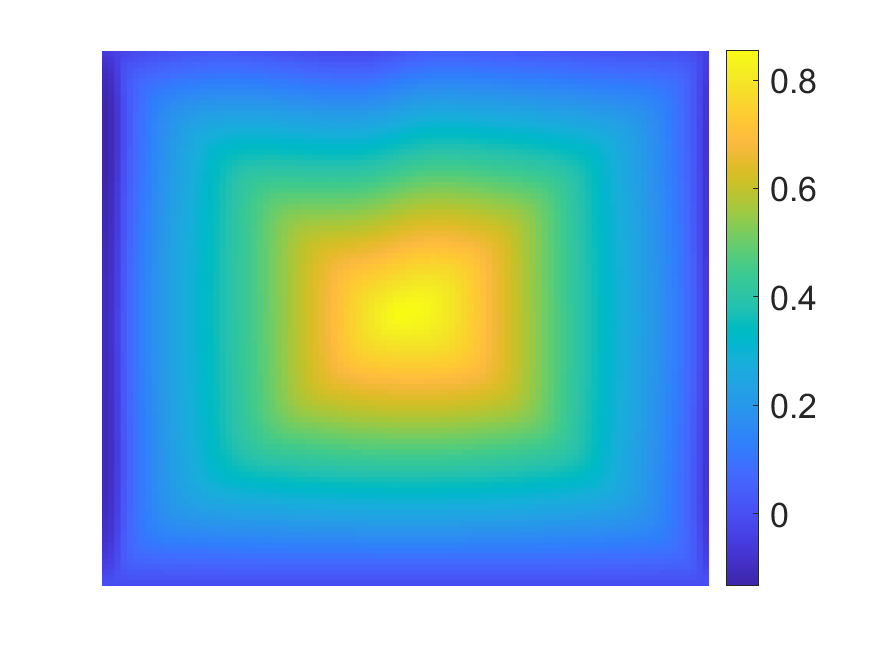}& \includegraphics[height=2.6cm]{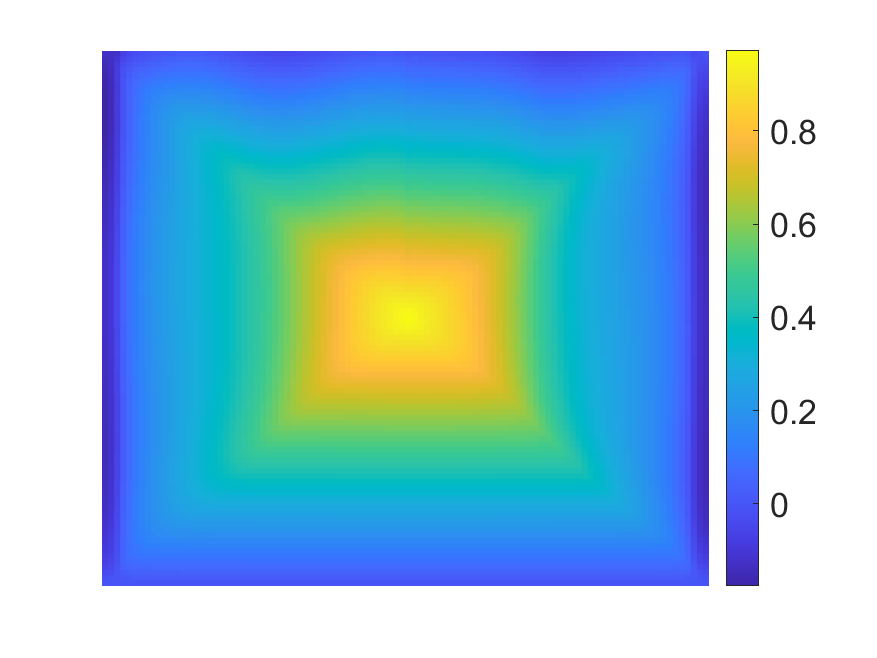}\\
(a) asymptotics & (b) $p=2$  & (c) $p=10$ & (d) $p=200$
\end{tabular}
\caption{\label{fig:exam2} The numerical solutions of $u^\ast$ for Example~\ref{exam2}  (slices at $x_3=0$).}
\end{figure}

\begin{figure}[hbt!]
\centering\setlength{\tabcolsep}{2pt}
\begin{tabular}{ccc}
\includegraphics[width=.32\textwidth]{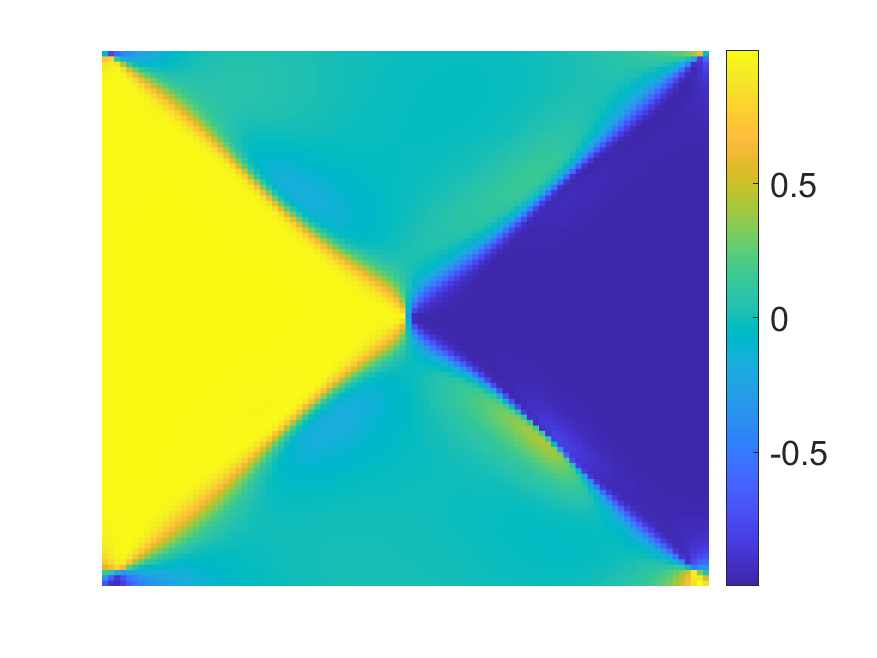} & \includegraphics[width=.32\textwidth]{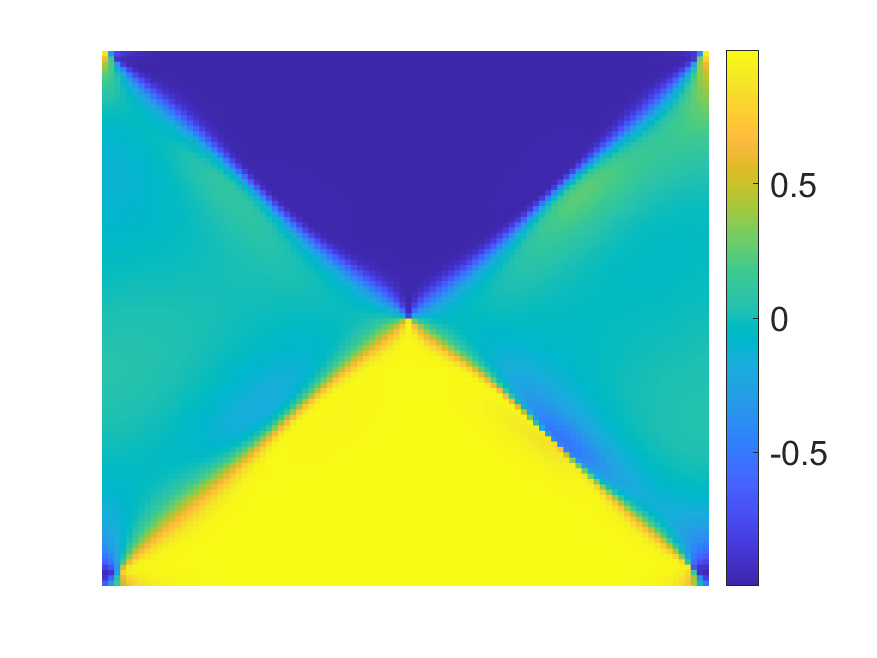} & \includegraphics[width=.32\textwidth]{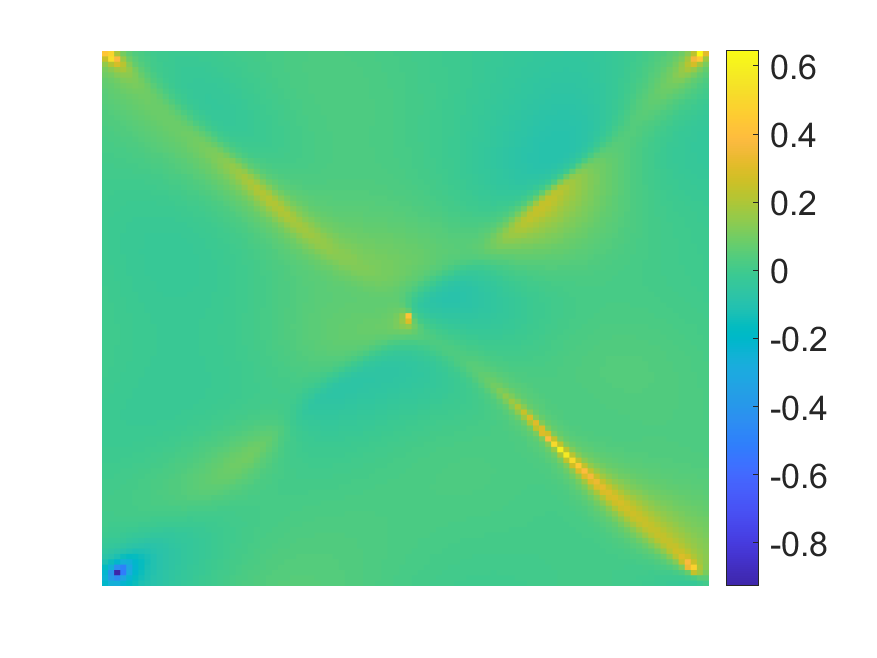}\\
(a) $u_{x_1}$ & (b) $u_{x_2}$ & (c) $u_{x_3}$
\end{tabular}
\caption{\label{fig:exam2.3} The numerical solutions of $\nabla u$ for Example~\ref{exam2} with $p=200$ (slices at $x_3=0$).}
\end{figure}

Next we extend the DVNN to the $p(x)$-Laplace equation. Since the $L^{p(\cdot)}(\Omega)$ norm is inconvenient to evaluate, we compute the relative error $e$ defined by $e:=\|\boldsymbol{\sigma}-\widehat {\boldsymbol{\sigma}}\|_{L^{q(\cdot)}(\Omega)}^{q(\cdot)}/\|\boldsymbol{\sigma}\|_{L^{q(\cdot)}(\Omega)}^{q(\cdot)}$, the $L^2(\Omega)$ (denoted by $e_2$) and $L^1(\Omega)$ (denoted by $e_1$) relative errors instead.
\begin{example}\label{exam3}
$\Omega=(-1,1)^3$, the solution $u(x)=\tfrac{x_1+x_2+x_3}{\sqrt{3}}$ and $f(x)=0$, with 
\begin{equation*}
p(x)=\left\{
\begin{aligned}
1.2,\quad x_1\leq0,\\
100,\quad x_1>0.
\end{aligned}\right.
\end{equation*}
\end{example}

The variable exponent $p(x)$ is piecewise constant, with extreme values in two subregions.
Since $f(x)=0$, the irrotational component $\phi$ in the first step is $0$.
So it suffices to train only the  loss
\begin{equation*}
L_{\rm s}(\boldsymbol{\psi})=\tfrac{1}{q}\|\nabla\times\boldsymbol{\psi}\|_{\bold{L}^q(\Omega)}^q+\langle \nabla\times\boldsymbol{\psi}\cdot \bold{n}, g\rangle.
\end{equation*}
We employ a 3-20-20-20-3 network to approximate $\boldsymbol{\psi}$.
The piecewise-constant $p(x)$ induces a sharp interface at the plane $x_1=0$.
The flux approximation $\bold{\sigma}_{\widehat\theta,\widehat\eta}$ by the DVNN is accurate, with relative errors $e=5.13\times10^{-3}$, $e_2=1.10\times10^{-2}$ and $e_1=8.36\times10^{-3}$; see Fig.~\ref{fig:exam3} for an illustration.
The high accuracy is attributed to the convexity of the loss $L_{\rm s}$ and the gradient-curl representation, which enforces the divergence-free condition
$\nabla\cdot\bold{\tau}=0$ exactly in $\Omega$, including across the interface.

\begin{figure}[hbt!]
\centering\setlength{\tabcolsep}{2pt}
\begin{tabular}{ccc}
\includegraphics[width=.32\textwidth]{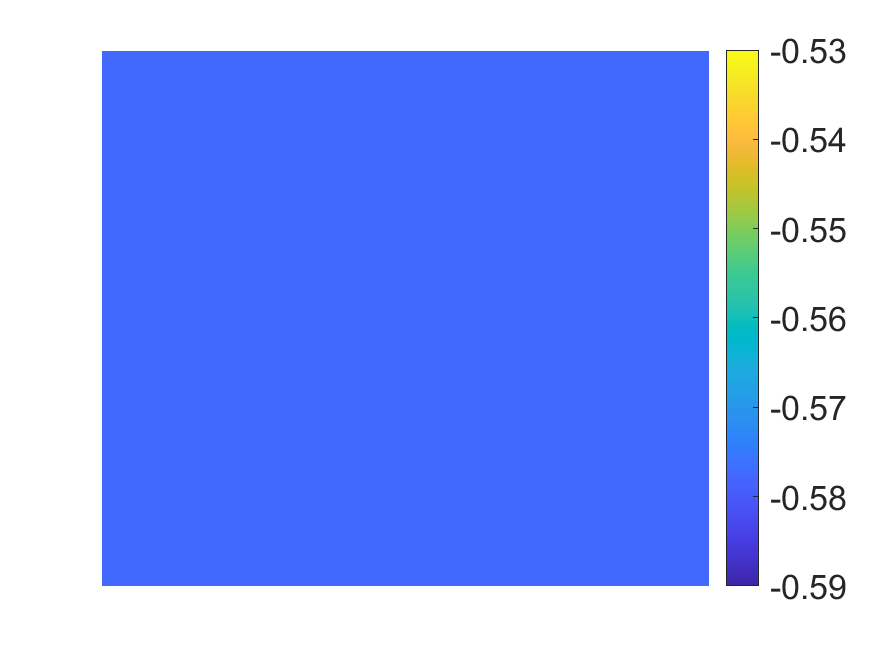} & \includegraphics[width=.32\textwidth]{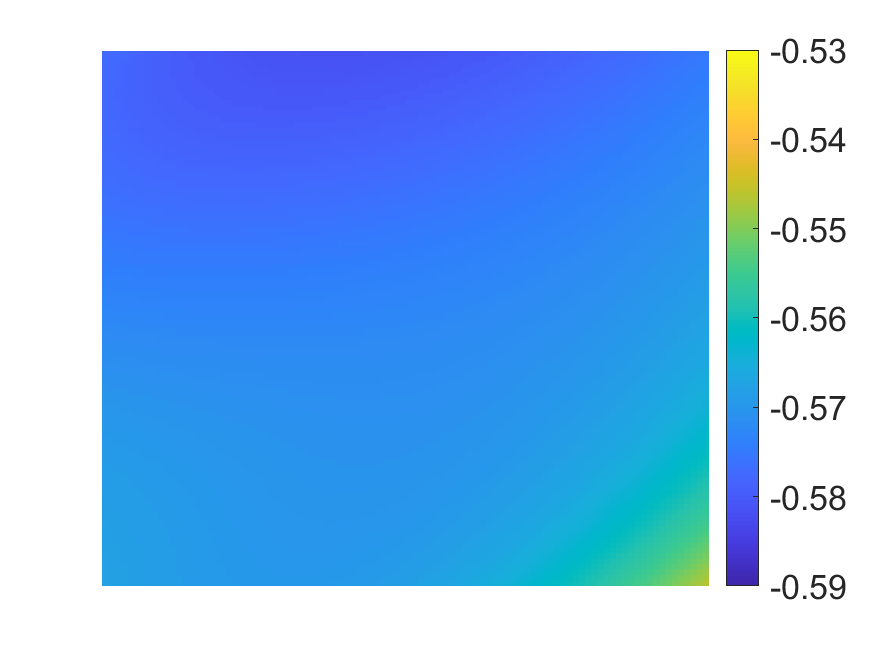} & \includegraphics[width=.32\textwidth]{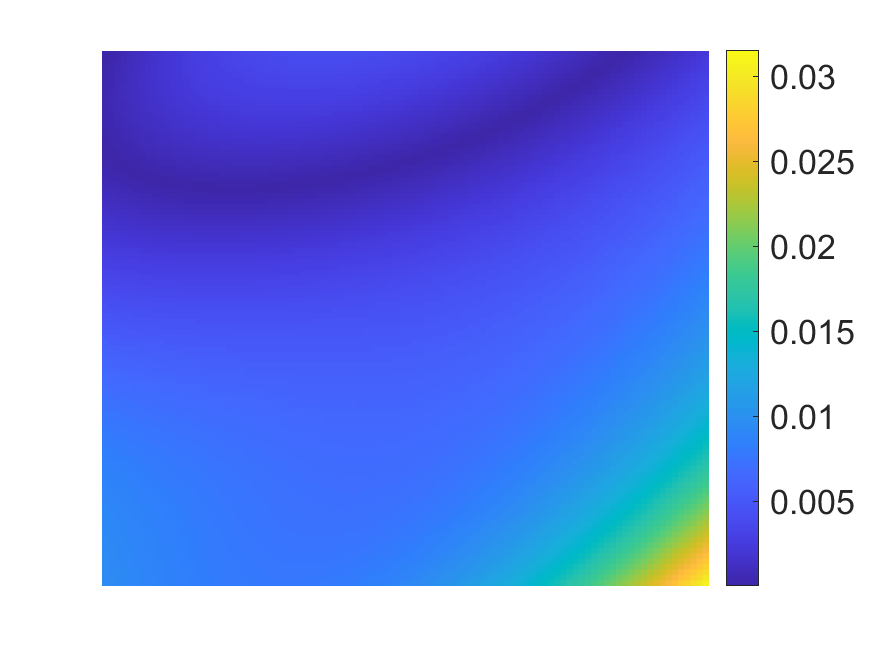}\\
(a) $\sigma_1^*$ & (b) $\sigma_{\widehat{\theta},\widehat{\eta},1}$  & (c) $|\sigma_1^*-\sigma_{\widehat{\theta},\widehat{\eta},1}|$
\end{tabular}
\caption{\label{fig:exam3} The DNN approximations of $\sigma_1$ for Example~\ref{exam3},  slices at $x_3=0$.}
\end{figure}

The last example is adapted from~\cite[Example 5.3]{aragon2023effective} and extended from 2d to 3d.
\begin{example}\label{exam4}
$\Omega=(-1,1)^3$, $p(x)=1+(\frac{x_1+x_2+x_3}{3}+3)^{-1}$, $u(x)=\sqrt{3}e^3(e^{\frac{x_1+x_2+x_3}{3}+3}-1)$ and $f(x)=0$.
\end{example}

Like in Example~\ref{exam3}, it suffices to solve the optimization problem at the second step.
We use a 3-20-20-20-3 network to approximate $\boldsymbol{\psi}$.
The approximation $\sigma_{\widehat\theta,\widehat\eta,1}$ of  $\sigma_1^*$ is shown in Fig.~\ref{fig:exam4} with a relative error $e=3.13\times10^{-8}$, $e_2=1.14\times10^{-2}$ and $e_1=9.43\times10^{-3}$. This indicates that the DVNN can deal with a smooth yet rapidly varying exponent $p(x)$ successfully.

\begin{figure}[hbt!]
\centering\setlength{\tabcolsep}{2pt}
\begin{tabular}{ccc}
\includegraphics[width=.32\textwidth]{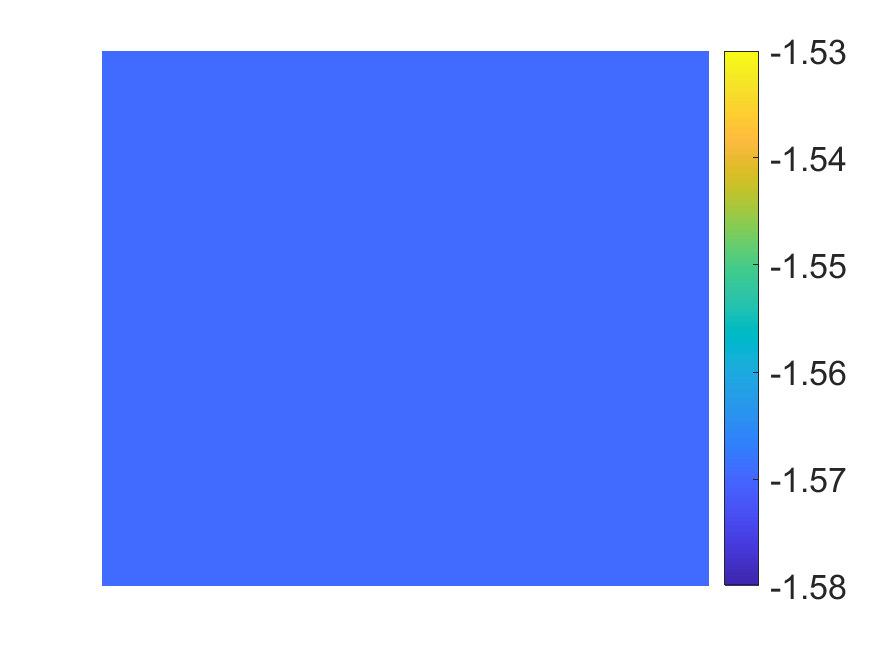} & \includegraphics[width=.32\textwidth]{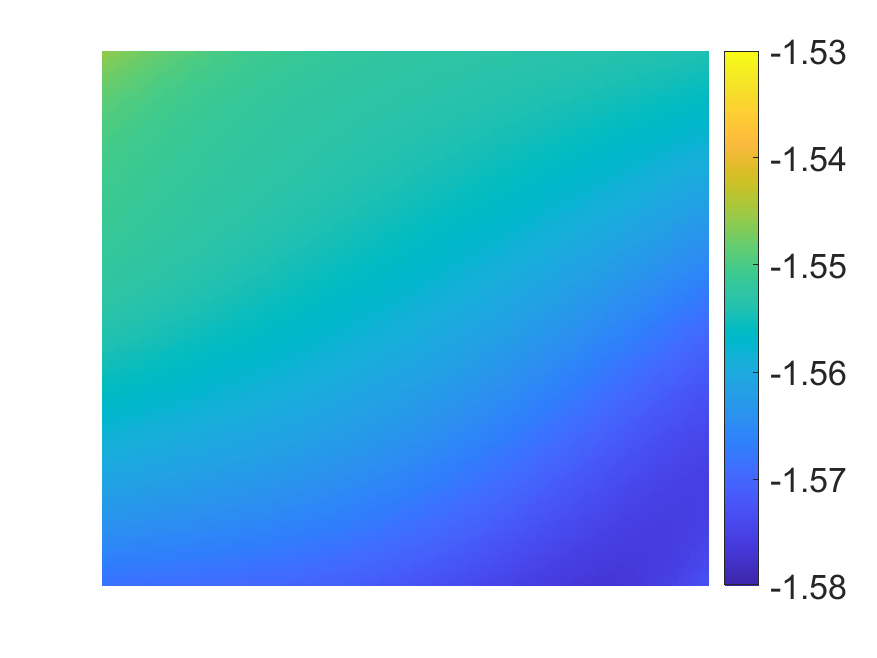} & \includegraphics[width=.32\textwidth]{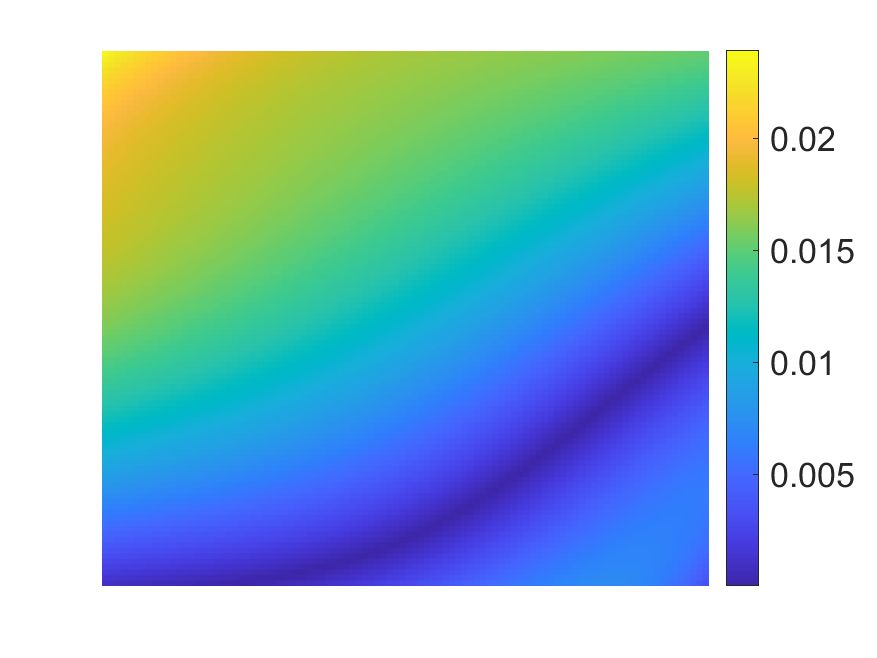}\\
(a) $\sigma_1^*$ & (b) $\sigma_{\widehat{\theta},\widehat{\eta},1}$  & (c) $|\sigma_1^*-\sigma_{\widehat{\theta},\widehat{\eta},1}|$
\end{tabular}
\caption{\label{fig:exam4} The DNN approximations of $\sigma_1$ for Example~\ref{exam4},  slices at $x_3=0$.}
\end{figure}

In summary, the numerical experiments confirm that the DVNN can achieve high accuracy in the range $1.1\le p\le 500$ and for both constant and spatially varying exponents $p(x)$, whereas
the PINN, DRM, and  PINN-M either stagnate or diverge. The DVNN maintains robust convergence, due to the divergence-free constraint preservation, and the decoupling into a linear Poisson solver and a convex unconstrained minimization.
These features make the DVNN the first neural method that can reliably approximate $p$-Laplace (at extreme values of $p$) and variable-exponent problems.

\section{Conclusions}\label{sec:conclusion}
In this work, we have developed a novel neural solver, i.e., dual variational neural network, for solving $p$-Laplace problems with extreme $p$ values. It is based on a dual formulation for the flux variable and the associated Helmholtz decomposition, which enable the construction of robust and accurate neural network approximations in standard Sobolev spaces.
The numerical experiments validate its robustness and accuracy in a wide range of $p$.
Moreover, the extension to variable-exponent $p(x)$-Laplace problems is direct and effective. 

\appendix

\section{Proof of Theorem \ref{thm:finalthm}}\label{append}
By Theorem \ref{thm:errordecom}, it suffices to bound $\mathcal{E}_{\rm app}$ and $\mathcal{E}_{\rm stat}$. The notation $|\cdot|_{\ell^0}$ and $|\cdot|_{\ell^\infty}$ denote the number of nonzero entries and the maximum norm of a vector, respectively. We recall an approximation property \cite[Proposition 4.8]{GuhringRaslan:2021}.

\begin{lemma}\label{lem:tanh-approx}
Let $s\in\mathbb{N}\cup\{0\}$ and $p\in[1,\infty]$ be fixed, and $v\in W^{k,p}(\Omega)$ with $k\geq  s+1$. Then for any  $\epsilon>0$, there exists at least one $v_\theta$ of depth $\mathcal{O}\big(\log(d+k)\big)$, with $|\theta|_{\ell^0}$ bounded by $\mathcal{O}\big(\epsilon^{-\frac{d}{k-s-\mu s}}\big)$ and  $|\theta|_{\ell^\infty}$ by $\mathcal{O}(\epsilon^{-2-\frac{2(d/p+d+s+\mu s)+d/p+d}{k-s-\mu s}})$, where $\mu>0$ is arbitrarily small, such that  $\|v-v_\theta\|_{W^{s,p}(\Omega)} \leq \epsilon.$
\end{lemma}

Under Assumption \ref{assump:regularity}(i), for $\epsilon>0$, with $d=3$, $k=3$, $s=2$ and $p=q$ in Lemma~\ref{lem:tanh-approx},  there exists 
$
\phi_\theta \in \mathcal{N}_\phi(c, c\epsilon^{-\frac{3}{1-\mu}},c\epsilon^{-\frac{15+9/q}{1-\mu}})$,
such that $ \|\phi^*-\phi_\theta  \|_{W^{2,q}(\Omega)} \leq \epsilon$.
Applying Lemma \ref{lem:tanh-approx} componentwise admits the existence of
$\boldsymbol{\psi}_\eta \in \mathcal{N}_{\boldsymbol{\psi}}(c, c\epsilon^{-\frac{3}{1-\mu}},c\epsilon^{-\frac{13+9/q}{1-\mu}})$
such that $ \|\boldsymbol{\psi}^*-\boldsymbol{\psi}_\eta  \|_{\bold{W}^{1,q}(\Omega)} \leq \epsilon$. 

The analysis of $\mathcal{E}_{\text{stat}}$ requires two concepts: Rademacher complexity \cite{AnthonyBartlett:1999,Bartlett2003RademacherAG} and the covering number.
Now we decompose the statistical error $\mathcal{E}_{\rm stat}$ into four terms:
{\small\begin{align*}
\mathcal{E}_{\rm stat}& \leq|\Omega|  \sup _{\phi_{\theta} \in \mathcal{N}_\phi}\left|N_d^{-1}\sum_{j=1}^{N_d}h_{d,\rm p}(X_j)-\mathbb{E}[h_{d,\rm p}(X)]\right|^{s_d}+|\partial\Omega|\sup _{\phi_{\theta} \in \mathcal{N}_\phi}\left|N_b^{-1}\sum_{j=1}^{N_b} h_{b,\rm p}(Y_j)-\mathbb{E}[h_{b,\rm p}(Y)]\right|^{s_d}\\
&+|\Omega|\sup_{\psi_{\eta} \in \mathcal{N}_{\boldsymbol{\psi}}}\left|N_d^{-1}\sum_{j=1}^{N_d} h_{d,\rm s}(X_j)-\mathbb{E}[h_{d,\rm s}(X)]\right|^{s_b}+|\partial\Omega|\sup_{\psi_{\eta} \in \mathcal{N}_{\boldsymbol{\psi}}}\left|N_b^{-1}\sum_{j=1}^{N_b}h_{b,\rm s}(Y_j)-\mathbb{E}[h_{b,\rm s}(Y)]\right|^{s_b}.
\end{align*}}
with the exponents $s_d=\min\{\frac{q}{4},\frac{2}{q^2}\}$, $s_b=\min\{\frac{1}{2},\frac{1}{q}\}$ and 
\begin{align*}
h_{d,\rm p}(X;\phi_\theta)&=|\Delta \phi_\theta(X)-f(X)|^q,&&
h_{b,\rm p}(Y;\phi_\theta)=|\phi_\theta(Y)|^q+|\nabla\phi_\theta(Y)|^q,\\
h_{d,\rm s}(X;\boldsymbol{\psi}_\eta)&=\tfrac{1}{q}|\nabla \phi_{\widehat{\theta^*}}(X)+\nabla\times\boldsymbol{\psi}_\eta(X)|^q,&&
h_{b,\rm s}(Y;\boldsymbol{\psi}_\eta)=g\nabla\times\boldsymbol{\psi}_{\eta}(Y)\cdot \bold{n}.
\end{align*}
This decomposition motivates four neural network (NN) function classes
$\mathcal{H}_{d,\rm p} = \{h_{d,\rm p}(X;\phi_\theta): \phi_\theta \in  \mathcal{N}_\phi\}$, $\mathcal{H}_{b,\rm p} = \{h_{b,\rm p}(Y;\phi_\theta): \phi_\theta \in  \mathcal{N}_\phi\}$, $\mathcal{H}_{d,\rm s} = \{h_{d,\rm s}(X;\boldsymbol{\psi}_\eta): \boldsymbol{\psi}_\eta \in  \mathcal{N}_{\boldsymbol{\psi}}\}$ and $\mathcal{H}_{b,\rm s} = \{h_{b,\rm s}(Y;\boldsymbol{\psi}_\eta): \boldsymbol{\psi}_\eta \in  \mathcal{N}_{\boldsymbol{\psi}}\}$.
Then we bound their Rademacher complexities using Dudley's formula (\cite[Theorem 9]{lu2021priori},~\cite[Theorem 1.19]{wolf2018mathematical}). It remains to estimate the covering number of each associated NN function class. This relies on quantitative Lipschitz constants of functions in the NN function class in terms of $\theta$; see~\cite[Lemma 3.4 and Remark 3.3]{JinLiLu:2022} and~\cite[Lemma 5.3]{Jin:DNN-Control}. 

\begin{lemma}\label{lem:NN-Lip}
Let $L$, $W$ and $B$ be the depth, width and maximum weight bound of an NN function class $\mathcal{N}_\phi$, with $N_\theta$ nonzero weights. Then for any $\phi_\theta\in\mathcal{N}_\phi$, the following estimates hold:
\begin{enumerate}
\item[{\rm(i)}] $\|\phi_\theta\|_{L^\infty(\Omega)}\leq WB$,   $\|\phi_\theta-\phi_{\tilde{\theta}}\|_{L^\infty(\Omega)}\leq 2LW^LB^{L-1}|\theta-\tilde\theta|_{\ell^\infty}$;
\item[{\rm(ii)}] $\|\nabla \phi_\theta\|_{\bold{L}^\infty(\Omega)}\leq \sqrt{d}W^{L-1}B^L$,
$\|\nabla (\phi_\theta-\phi_{\tilde{\theta}})\|_{\bold{L}^\infty(\Omega)}\leq  \sqrt{d}L^2W^{2L-2}B^{2L-2}|\theta-\tilde\theta|_{\ell^\infty}$;
\item[{\rm(iii)}] $\|\Delta \phi_\theta\|_{L^\infty(\Omega)}\leq dLW^{2L-2} B^{2L} $, 
$\|\Delta (\phi_\theta-\phi_{\tilde{\theta}})\|_{L^\infty(\Omega)}\leq  4dN_\theta L^2W^{3L-3}B^{3L-3}|\theta-\tilde\theta|_{\ell^\infty}$.
\end{enumerate}
\end{lemma}

Note that similar estimates to those in Lemma \ref{lem:NN-Lip} also hold when $L^\infty(\Omega)$ is replaced with $L^\infty(\partial\Omega)$, and also for each component of any function in $\mathcal{N}_{\boldsymbol{\psi}}$.
Now we can bound $\mathcal{E}_{\rm stat}$.
\begin{theorem}\label{thm:err-stat}
Under Assumption~\ref{assump:regularity}(ii), for small $\tau $, with probability at least $1- 4 \tau$,
$$\mathcal{E}_{\rm stat}\leq c(e_d + e_b),$$
with  $c=c(q,\|f\|_{L^\infty(\Omega)},\|g\|_{L^\infty(\partial\Omega)})$, and the errors $e_d$ and $e_b$  defined by
\begin{align*}
e_d & \le c\left\{\begin{aligned}
    &{ B^{\frac{q^2}{2}L} N^{\frac{q^2}{2}(L-1)} \big(P(B,N,N_d,\tau)\big)^{\frac{1}{2}}}{N_d^{-\frac{q}{8}}},\quad&&1<q<2,\\
    &{ B^{\max\{\frac{4}{q},1\}L} N^{\max\{\frac{4}{q},1\}(L-1)} \big(P(B,N,N_d,\tau)\big)^{\frac{1}{q}}}{N_d^{-\frac{1}{q^2}}},\quad &&q\geq2,
\end{aligned}\right. \\
e_b & \le c\left\{\begin{aligned}
    &{ B^{\max\{\frac{q^2}{4},\frac{1}{2}\}L} N^{\max\{\frac{q^2}{4},\frac{1}{2}\}(L-1)} \big(P(B,N,N_b,\tau)\big)^{\frac{1}{2}}}{N_b^{-\frac{q}{8}}},\quad&&1<q<2,\\
    &{ B^{\frac{2}{q}L} N^{\frac{2}{q}(L-1)} (P(B,N,N_b,\tau))^{\frac{1}{q}}}{N_b^{-\frac{1}{q^2}}},\quad &&q\geq2.
\end{aligned}\right.
\end{align*}
with $P(B,N,N_d,\tau)=N_\theta^\frac12\big( \log^\frac12 B + \log^\frac12 N + \log^\frac12 N_d)  + \log^\frac12 \frac{1}{\tau}$ and $N=\max\{N_\theta,N_\eta\}$.
\end{theorem}

\begin{proof}
Fix $m \in \mathbb{N }$, $B \in [1, \infty)$, $\epsilon \in  (0,1)$,
and $\mathbb{B}_{B} := \{x\in\mathbb{R}^m:\ |x|_{\ell^\infty}\leq B\}$. By~\cite[Proposition 5]{CuckerSmale:2002}, there holds
$ \log \mathcal{C}(\mathbb{B}_{B},|\cdot|_{\ell^\infty},\epsilon)\leq m\log (4B\epsilon^{-1})$. For  $\mathcal{H}_{d,\rm p}$, by Lemma \ref{lem:NN-Lip}, we have
$$\begin{aligned}
\||\Delta \phi_\theta-f|^q-|\Delta \phi_{\tilde{\theta}}-f|^q\|_{L^\infty(\Omega)}\leq& q\|\Delta \phi_\theta-\Delta\phi_{\tilde{\theta}}\|_{L^\infty(\Omega)}(\|\Delta \phi_\theta-f\|_{L^\infty(\Omega)}^{q-1}+\|\Delta \phi_{\tilde{\theta}}-f\|_{L^\infty(\Omega)}^{q-1})\\
\leq&24qN_\theta L^2W^{3L-3}B^{3L-3}(3LW^{2L-2} B^{2L}+\|f\|_{L^\infty(\Omega)})^{q-1}|\theta-\tilde\theta|_{\ell^\infty}\\
\leq& c(\|f\|_{L^\infty(\Omega)})qN_\theta L^{q+1}W^{(L-1)(3q+1)}B^{2Lq+L-3}|\theta-\tilde\theta|_{\ell^\infty},
\end{aligned}$$
and consequently, the covering number $\mathcal{C}(\mathcal{H}_{d,\rm p},\|\cdot\|_{L^{\infty}(\Omega)},\epsilon)$ satisfies 
\begin{align*}
\log \mathcal{C}(\mathcal{H}_{d,\rm p},\|\cdot\|_{L^{\infty}(\Omega)},\epsilon)&\leq
\log \mathcal{C}(\Theta ,|\cdot|_{\ell^\infty},\Lambda_{d,\rm p}^{-1}\epsilon)  \leq cN_\theta \log(4B \Lambda_{d,\rm p}\epsilon^{-1}).
\end{align*}
with $\Lambda_{d,\rm p}=c(\|f\|_{L^\infty(\Omega)},q)N_\theta L^{q+1}W^{(L-1)(3q+1)}B^{2Lq+L-3}$. By Lemma~\ref{lem:NN-Lip}, we have $M_{\mathcal{H}_{d,\rm p}} = c(\|f\|_{L^\infty(\Omega)})L^qW^{2q(L-1)}B^{2Lq}.$
Then by Dudley's lemma (\cite[Theorem 9]{lu2021priori},~\cite[Theorem 1.19]{wolf2018mathematical}), since $1\leq B$, $2\leq L$ and $1\leq W \leq N_\theta$, $1\leq L\leq c\log5$, cf. Lemma~\ref{lem:tanh-approx}, we can bound the Rademacher complexity $\mathfrak{R}_{N_d}(\mathcal{H}_{d,\rm p}) $ by
\begin{align*}
\mathfrak{R}_{N_d}(\mathcal{H}_{d,\rm p})
\leq&4N_d^{-\frac12} +12N_d^{-\frac12}\int^{M_{\mathcal{H}_{d,\rm p}}}_{N_d^{-\frac12}}{\big(cN_\theta \log(4B\Lambda_{d,\rm p}\epsilon^{-1})\big)}^{\frac12}\ {\rm d}\epsilon\\
\leq& 4N_d^{-\frac12}+12N_d^{-\frac12}M_{\mathcal{H}_{d,\rm p}}\big(cN_\theta \log(4B\Lambda_{d,\rm p} N_d^{\frac12})\big)^\frac12 \\
\leq& c(q,\|f\|_{_{L^\infty(\Omega)}}) N_d^{-\frac12} B^{2qL} N_\theta^{2qL-2q+\frac{1}{2}} \big( \log^\frac12 B + \log^\frac12 N_\theta + \log^\frac12 N_d \big).
\end{align*}
Repeating the preceding argument yields
\begin{align*}
\mathfrak{R}_{N_b}(\mathcal{H}_{b,\rm p})&\leq c(q,\|g\|_{_{L^\infty(\partial\Omega)}})N_b^{-\frac{1}{2}}B^{qL} N_\theta^{qL-q+\frac{1}{2}}(\log^{\frac{1}{2}} B+\log^{\frac{1}{2}} N_\theta+\log^{\frac{1}{2}} N_b),\\
\mathfrak{R}_{N_d}(\mathcal{H}_{d,\rm s})&\leq c(q,\|f\|_{_{L^\infty(\Omega)}})N_d^{-\frac{1}{2}}B^{qL} N_\eta^{qL-q+\frac{1}{2}}(\log^{\frac{1}{2}} B+\log^{\frac{1}{2}} N_\eta+\log^{\frac{1}{2}} N_d),\\
\mathfrak{R}_{N_b}(\mathcal{H}_{b,\rm s})&\leq c(q,\|g\|_{_{L^\infty(\partial\Omega)}})N_b^{-\frac{1}{2}}B^{L} N_\eta^{L-\frac{1}{2}}(\log^{\frac{1}{2}} B+\log^{\frac{1}{2}} N_\eta+\log^{\frac{1}{2}} N_b).
\end{align*}
Finally, the desired result follows from the PAC bound \cite[Theorem 3.1]{Mohri:2018}.
\end{proof}

\begin{proof}[Proof of Theorem \ref{thm:finalthm}]
By Lemma~\ref{lem:tanh-approx}, there exist two DNNs $
\phi_\theta \in \mathcal{N}_{\phi}(c, c\epsilon^{-\frac{3}{1-\mu}},c\epsilon^{-\frac{15+9/q}{1-\mu}})$ and $\boldsymbol{\psi}_\eta \in \mathcal{N}_{\boldsymbol{\psi}}(c, c\epsilon^{-\frac{3}{1-\mu}},c\epsilon^{-\frac{13+9/q}{1-\mu}})$, such that 
$\mathcal{E}_{\rm app}\leq c\epsilon^{\frac{q}{4}}$ for $1<q<2$ and $\mathcal{E}_{\rm app}\leq c\epsilon^{\frac{2}{q^2}}$ for $ q\geq2$. 
Then in Theorem \ref{thm:err-stat}, we select the numbers $N_d$ and $N_b$ of sampling points as given by \eqref{eqn:ndnb}. The proof is completed by combining the two explicit bounds on $\mathcal{E}_{\rm app}$ and $\mathcal{E}_{\rm stat}$ with Theorem \ref{thm:errordecom}.
\end{proof}

\bibliographystyle{siam}
\bibliography{refer}

\end{document}